\documentclass[11pt]{amsart}
\usepackage{amsmath}
\usepackage{amssymb, verbatim}
\usepackage{pstricks,pst-node,pst-plot}
\usepackage{tikz, tikz-cd}
\usepackage[all,cmtip]{xy}
\usepackage{comment}
\usepackage{color}

\title[Mirror Symmetry for del Pezzo/rational elliptic surface pairs]{The SYZ mirror symmetry conjecture for del Pezzo surfaces and rational elliptic surfaces}
\author[T. C. Collins]{Tristan C. Collins}
  \email{tristanc@mit.edu}
  \address{Department of Mathematics, Massachusetts Institute of Technology, 77 Massachusetts Avenue, Cambridge, MA 02139}
 \thanks{T.C.C is supported in part by NSF grant DMS-1810924, NSF CAREER grant DMS-1944952 and an Alfred P. Sloan Fellowship. }

 \author[A. Jacob]{Adam Jacob}
  \email{ajacob@math.ucdavis.edu}
  \address{Department of Mathematics, University of California, Davis, 1 Shields Ave, Davis, CA, 95616}
  \thanks{A.J. is supported in part by a Simons collaboration grant}
  
  \author[Y.-S. Lin] {Yu-Shen Lin}
   \email{yslin@bu.edu}
  \address{Department of Mathematics, Boston University, 11 Cummington Mall, Boston, MA 02215}
    \thanks{Y.-S. L. is supported in by Simons collaboration grant \# 635846 and NSF grant DMS \#2204109.}

\theoremstyle{plain}
\newtheorem{thm}{Theorem}[section]
\newtheorem{prop}[thm]{Proposition}
\newtheorem{defn}[thm]{Definition}
\newtheorem{lem}[thm]{Lemma}
\newtheorem{cor}[thm]{Corollary}
\newtheorem{conj}[thm]{Conjecture}

\theoremstyle{definition}

\newtheorem{rk}[thm]{Remark}

\numberwithin{equation}{section}

\newcommand{\del}{\partial}

\newcommand{\dbar}{\overline{\del}}
\newcommand{\ddb}{\sqrt{-1}\del\dbar}

\newcommand{\ti}[1]{\tilde{#1}}

\newcommand{\be}{\begin{equation}}
\newcommand{\bea}{\begin{eqnarray}}
\newcommand{\eea}{\end{eqnarray}}
 \newcommand{\ee}{\end{equation}}
 \newcommand{\pl}{\partial}

\renewcommand{\leq}{\leqslant}
\renewcommand{\geq}{\geqslant}

\renewcommand{\epsilon}{\varepsilon}
\renewcommand{\phi}{\varphi}

\begin{document}

\maketitle

\begin{abstract}
We prove a version of the Strominger-Yau-Zaslow mirror symmetry conjecture for non-compact Calabi-Yau surfaces arising from, on the one hand, pairs $(\check{Y},\check{D})$ of a del Pezzo surface $\check{Y}$ and $\check{D}$ a smooth anti-canonical divisor and, on the other hand, pairs $(Y,D)$ of a rational elliptic surface $Y$, and $D$ a singular fiber of Kodaira type $I_k$.  Three main results are established concerning the latter pairs $(Y,D)$.  First, adapting work of Hein \cite{Hein}, we prove the existence of a complete Calabi-Yau metric on $Y\setminus D$ asymptotic to a (generically non-standard) semi-flat metric in every K\"ahler class.  Secondly, we prove a uniqueness theorem to the effect that, modulo automorphisms, every K\"ahler class on $Y\setminus D$ admits a unique asymptotically semi-flat Calabi-Yau metric.  This result yields a finite dimensional K\"ahler moduli space of Calabi-Yau metrics on $Y\setminus D$.  Further, this result answers, in this setting, questions of Tian-Yau \cite{TY} and Yau \cite{YauComm}.  Thirdly, building on the authors' previous work \cite{Collins-Jacob-Lin}, we prove that $Y\setminus D$ equipped with an asymptotically semi-flat Calabi-Yau metric  $\omega_{CY}$ admits a special Lagrangian fibration whenever the  de Rham cohomology class of $\omega_{CY}$ is not topologically obstructed.  Combining these results we define a mirror map from the moduli space of del Pezzo pairs $(\check{Y}, \check{D})$ to the complexified K\"ahler moduli of $(Y,D)$ and prove that the special Lagrangian fibration on $(Y,D)$ is $T$-dual to the special Lagrangian fibration on $(\check{Y}, \check{D})$ previously constructed by the authors in \cite{Collins-Jacob-Lin}.  We give some applications of these results, including to the study of automorphisms of del Pezzo surfaces fixing an anti-canonical divisor.
\end{abstract}

\section{Introduction}

The primary goal of this paper is to prove a version of the Strominger-Yau-Zaslow (SYZ) mirror symmetry conjecture for log Calabi-Yau surfaces arising from del Pezzo surfaces and rational elliptic surfaces.  A general formulation of the original SYZ mirror symmetry conjecture is the following: 

\begin{conj}[Strominger-Yau-Zaslow]\label{conj: introSYZ}
Let $(\check{X},\check{\omega})$ be a Calabi-Yau manifold, and $\check{\mathcal{M}}_{cplx}$ denote the moduli space of complex structures on $\check{X}$.  Then, for a complex structure $\check{J}\in \check{\mathcal{M}}_{cplx}$ sufficiently close to a large complex structure limit, the following is true:
\begin{enumerate}
\item $(\check{X},\check{J},\check{\omega})$ admits a special Lagrangian torus fibration $\check{\pi}:\check{X}\rightarrow \check{B}$ onto a base $\check{B}$ equipped with an integral affine structure.
\item There is another Calabi-Yau manifold $(X, J, \omega)$ with a special Lagrangian fibration $\pi: X\rightarrow B$ and $B$ is equipped with an integral affine structure.
\item Let $\mathcal{M}_{\text{K\"ah}}$ denote the complexified K\"ahler moduli space of $X$.  There is a mirror map $q: \check{\mathcal{M}}_{cplx}\rightarrow \mathcal{M}_{\text{K\"ah}}$ which is a local diffeomorphism such that ${\rm Im}(q(\check{J})) = \omega$.
\item There is an isomorphism $\phi: \check{B} \rightarrow B$ exchanging the complex and symplectic affine structures, and such that the Riemannian volumes of the special Lagrangian torus fibers over $\check{b} \in \check{B}, \phi(\check{b})\in B$ are inverse to one another.
\end{enumerate}
\end{conj}

If $X$ is a compact Calabi-Yau of dimension $n$, then the mirror $\check{X}$ is also a compact Calabi-Yau of dimension $n$, and the third point implies the exchange of Hodge numbers
\[
h^{n-1,1}(\check{X}) = h^{1,1}(X).
\]
An important philosophical point is that $\mathcal{M}_{\text{K\"ah}}$ can be viewed as (the complexification of) the symplectic moduli space of Calabi-Yau structures on $X$ thanks to Yau's solution of the Calabi conjecture \cite{Y}.  Similarly, $\check{\mathcal{M}}_{cplx}$ can be viewed as the complex moduli space of Calabi-Yau structures on $\check{X}$, thanks to the Bogomolov-Tian-Todorov theorem \cite{Bog, Tian, Tod}.

According to work of Hitchin \cite{Hit} the integral affine structure on the base of the special Lagrangian torus fibrations in Conjecture~\ref{conj: introSYZ} is inherited from the complex/symplectic geometry of $X,\check{X}$.  Furthermore, Hitchin shows that point (4) of Conjecture~\ref{conj: introSYZ} can be viewed as an instantiation of the central principle of the SYZ conjecture: mirror symmetry is $T$-duality.

The SYZ conjecture has served as a guiding principle in the study of mirror symmetry over the past 20$+$ years, but outside of simple cases like $K3$ surfaces and abelian varieties, there are essentially no examples where it is known to hold in the formulation given above (for some recent progress on a weak version of point (1) of the SYZ conjecture, see \cite{Li19, Li20}).  This has inspired synthetic approaches to mirror symmetry including the Gross-Siebert program \cite{GS}, family Floer mirror construction \cite{Fuk, Abo}, the work of Kontsevich-Soibelman \cite{KS} and Doran-Harder-Thompson \cite{DHT}. Here we use the term ``original" conjecture to distinguish with the latter interpretations. 
   
While mirror symmetry was originally discovered in the context of compact Calabi-Yau manifolds it is now understood to be a rather general phenomenon.  For example, mirror symmetry is expected to apply to Fano manifolds and more generally to manifolds with effective anti-canonical bundle.  In this case the mirror manifold is no longer compact but instead is expected to be a Landau-Ginzburg model consisting of a non-compact complex manifold $M$ equipped with a superpotential $W: M\rightarrow \mathbb{C}$. 

If $Y$ is a compact K\"ahler manifold and $D \in |-K_{Y}|$ is an anti-canonical divisor, Auroux \cite{Aur1} formulated conjectures to the effect that mirror symmetry for $Y$ could be obtained from SYZ mirror symmetry applied to the {\em non-compact} manifold $X= Y\setminus D$.  Note that $X$ has trivial canonical bundle and hence can be regarded as a non-compact Calabi-Yau manifold, though the existence of a complete Ricci-flat K\"ahler metric on $X$ does not follow from Yau's solution of the Calabi conjecture \cite{Y}. Motivated by the original SYZ conjecture, Auroux \cite{Aur1,Aur2} conjectured the existence of special Lagrangian torus fibrations on $X$ and furthermore made several detailed conjectures about the structure of these fibrations.  Note, however, the in Auroux's formulation the background symplectic form is not required to be be a complete Calabi-Yau metric.

In this paper we focus on the case of rational elliptic surfaces and del Pezzo surfaces.  Recall that a rational elliptic surface (RES) is a rational surface with a relatively minimal elliptic fibration onto $\mathbb{P}^1$ which admits a section.  Let $\check{Y}$ be a rational elliptic surface or a del Pezzo surface and assume that $\check{D}\in |-K_{\check{Y}}|$ is a smooth divisor. Tian-Yau \cite{TY} proved the existence of a complete Ricci-flat metric on the non-compact manifold $\check{X}= \check{Y}\setminus \check{D}$.  In this setting the authors recently proved part (1) of the SYZ conjecture, verifying some conjectures of Auroux.

   \begin{thm}[Theorem 1.2, \cite{Collins-Jacob-Lin}] \label{513}
    	Let $\check{Y}$ be a del Pezzo surface or a rational elliptic surface and $\check{D}$ a smooth anti-canonical divisor. Then, for any choice of simple closed loop $\gamma \in H_1(D,\mathbb{Z})$, $\check{X}=\check{Y}\setminus \check{D}$, equipped with the Tian-Yau metric admits a special Lagrangian fibration $\check{\pi}_{\gamma}:\check{X}\rightarrow \mathbb{R}^2$, where the fibre of $\check{\pi}_{\gamma}$ is homotopic to the unit $S^1$-bundle over a representative of $\gamma$. 
    \end{thm}

The techniques developed in \cite{Collins-Jacob-Lin}, particularly those establishing Theorem~\ref{513} will play an important role in our proof of Conjecture~\ref{conj:  introSYZ} for del Pezzo surfaces and rational elliptic surfaces.  

Before outlining our results we note that homological mirror symmetry between rational elliptic surfaces and del Pezzo surfaces is quite well  understood.   Auroux-Katzarkov-Orlov \cite{AKO} proved that the derived category of coherent sheaves on a del Pezzo surface $\check{Y}$ of degree $k$ is equivalent to the Fukaya-Seidel category of the complement $Y\setminus D$ where $Y$ is a rational elliptic surface and $D$ is an $I_k$ singular fiber. Lunts-Przyjalkowski \cite{LunPr} proved mirror symmetry of Hodge diamonds, where the Hodge numbers are defined following a proposal of Katzarkov-Kontsevich-Pantev \cite{KKP}.  Doran-Thompson \cite{DT} studied the mirror correspondence in the sense of lattice polarized mirror symmetry, as motivated by the Doran-Harder-Thompson conjecture \cite{DHT}. Gross-Hacking-Keel \cite{GHK} have constructed formal mirrors to rational elliptic surfaces with an $I_k$ singular fiber along the lines of the Gross-Siebert program.  Similarly, formal mirrors for del Pezzo surfaces were constructed via the Gross-Siebert program by Carl-Pumperla-Siebert \cite{CPS}.

Let us now explain the central results of this paper which allow us to prove SYZ mirror symmetry for del Pezzo surfaces and rational elliptic surfaces.  We pause for the following important remark:
\begin{rk}
In what follows, rather than stating the precise results we obtain with all relevant technical assumptions, we will state the results in as much precision as necessary to motivate the discussion and provide references to the precisely stated theorems in the body of the text.
\end{rk}

Let $\check{Y}$ be a del Pezzo surface of degree $k$ and $\check{D}\in |-K_{\check{Y}}|$ a smooth anti-canonical divisor and let $\check{X}= \check{Y}\setminus \check{D}$.  The complex moduli space of del Pezzo pairs $(\check{Y}, \check{D})$ is well understood thanks to the Torelli theorem of McMullen \cite{M}. The existence of Calabi-Yau metrics on $\check{X}$ was established by Tian-Yau \cite{TY} and the existence of special Lagrangian fibrations was established by the authors in \cite{Collins-Jacob-Lin} (see Theorem~\ref{513}). Thus, in the context of Conjecture~\ref{conj: introSYZ}, point (1) can be taken to be understood.  

Now suppose $\pi: Y\rightarrow \mathbb{P}^1$ is a rational elliptic surface and $D \in |-K_{Y}|$ is a singular fiber of Kodaira type $I_k$; namely, a wheel of $k$ rational curves with self-intersection $-2$.  In order to address point $(2)$ of Conjecture~\ref{conj: introSYZ}, a first step is to understand the existence of complete Calabi-Yau metrics on $X=Y\setminus D$.  Suppose $\omega_0$ is a K\"ahler metric on $Y$. Hein \cite{Hein} constructed complete Calabi-Yau metrics in $[\omega_0|_{X}] \in H_{dR}^{2}(X,\mathbb{R})$ asymptotic to the standard semi-flat metrics constructed by Greene-Shapere-Vafa-Yau \cite{GSVY}.  However, these metrics cover only a codimension $1$ slice of the full K\"ahler cone of $X$.  Furthermore, the construction of \cite{Hein} yields infinitely many Calabi-Yau metrics in a given de Rham class, and leaves open the possibility of (infinitely many) distinct Calabi-Yau metrics even within a fixed Bott-Chern cohomology class.  In this paper we adapt Hein's construction to prove the existence of Calabi-Yau metrics in {\em every} de Rham class in $H_{dR}^{2}(X,\mathbb{R})$ containing a K\"ahler form.  These Calabi-Yau metrics are, in general, asymptotic to {\em non-standard} semi-flat metrics; see Section~\ref{subsec: nonSt} for discussion. 
 Furthermore, we address the uniqueness of these metrics (see Theorem~\ref{thm: unique}) confirming an expectation of Tian-Yau \cite{TY} in this setting (see Proposition~\ref{prop: reducedModuli}) and answering a question of Yau \cite{YauComm}.

\begin{thm}\label{thm: introCY}
Let $\pi:Y\rightarrow \mathbb{P}^1$ be a rational elliptic surface, $D$ a singular fiber of type $I_k$, and let $F$ denote a fiber of $\pi$. Let $X = Y\setminus D$ and $\pi:X \rightarrow \mathbb{C}$ be the induced elliptic fibration.  Define ${\rm Aut}_{0}(X,\mathbb{C})$ to be the fiber preserving biholomorphisms of $X$ which are homotopic to the identity. Then:
\begin{itemize}
\item[$(i)$] For every de Rham cohomology class $[\omega] \in H^2(X,\mathbb{R})$ containing a K\"ahler form $\omega$ and having $[\omega].[F]$ sufficiently small, there is a complete Calabi-Yau metric asymptotic to a (possibly non-standard) semi-flat K\"ahler metric.  Furthermore, this metric is {\em unique} modulo the action of ${\rm Aut}_{0}(X,\mathbb{C})$.
\item[$(ii)$] If $\omega_{CY}$ is a complete Calabi-Yau metric on $X$ asymptotic to a {\em quasi-regular} semi-flat metric, then there is a special Lagrangian torus fibration $\pi: (X,\omega_{CY}) \rightarrow \mathbb{R}^2$.
\end{itemize}
\end{thm}
\begin{rk}
For the precise statement of part$(i)$ we refer the reader to Corollary~\ref{cor: HeinUnique} and Proposition~\ref{prop: reducedModuli}.   Part $(ii)$ is the conclusion of Theorem~\ref{154}.  
\end{rk} 

The third author \cite{Lin} computed the superpotential for the special Lagrangian fibrations on del Pezzo surfaces constructed by the authors in \cite{Collins-Jacob-Lin}.  Combining this work with Theorem~\ref{thm: introCY} confirms the conjectural picture of Auroux \cite{Aur1} for del Pezzo surfaces and rational elliptic surfaces.

Theorem~\ref{thm: introCY} plays two roles in the proof of the Conjecture~\ref{conj: introSYZ} for del Pezzo/RES pairs.  First, it establishes point (2) of Conjecture~\ref{conj: introSYZ}.  Secondly, and equally importantly, it yields a finite dimensional complexified K\"ahler moduli space $\mathcal{M}_{\text{K\"ah}}$ parametrizing Calabi-Yau metrics on $X$.  Using Theorem~\ref{thm: introCY}, we construct the mirror map $q$ from the complex moduli of del Pezzo pairs to the complexified K\"ahler moduli of a particular rational elliptic surface with an $I_k$ fiber.  In order to define the mirror map we introduce a notion of large complex structure limit for a del Pezzo pair $(\check{Y}, \check{D})$, which to the authors knowledge has not appeared in the literature before (see Definition~\ref{def: LCSL}).

\begin{thm}\label{thm: SYZintro}
Let $\check{\mathcal{M}}_{cplx}$ denote the complex moduli space of pairs $(\check{Y}, \check{D})$ consisting of a del Pezzo surface $\check{Y}$ of degree $k$ and a smooth anti-canonical divisor $\check{D} \in |-K_{\check{Y}}|$. Let $\mathcal{M}_{\text{K\"ah}}$ denotes the complexified K\"ahler moduli consisting pairs $(Y,D)$ with $Y$ being a rational elliptic surface and $D$ being an $I_k$-fibre. There is a rational elliptic surface $\pi:Y\rightarrow \mathbb{P}^1$ with an $I_k$ singular fiber $D$ and local diffeomorphism onto its image 
\[
q:\check{\mathcal{M}}_{cplx} \rightarrow \mathcal{M}_{\text{K\"ah}}
\]
 such that Conjecture~\ref{conj: introSYZ} holds.
\end{thm} 

A corollary of this is an equality of the geometrically defined Hodge numbers

\begin{cor}
In the setting of Theorem~\ref{thm: SYZintro} we have
\[
\dim_{\mathbb{C}}\check{\mathcal{M}}_{cplx}=\dim_{\mathbb{C}}\mathcal{M}_{\text{K\"ah}} = 10-k.
\]
Furthermore, by the calculation Lunts-Przyjalkowski \cite{LunPr} and Ballico-Gasparim-Rubilar-San Martin \cite{BGRS}, these dimensions agree with the algebro-geometric Hodge numbers proposed by Katzarkov-Kontsevich-Pantev \cite{KKP}.
\end{cor}

For the precise statement of Theorem~\ref{thm: SYZintro} we refer the reader to Theorem~\ref{SYZ MS}. 

In the course of establishing Theorems~\ref{thm: introCY} and Theorem~\ref{thm: SYZintro} we obtain several intermediate results along the way.  For example, we identify the symplectic structure associated to the hyperK\"ahler rotated Tian-Yau metric on a del Pezzo surface (see Proposition~\ref{prop: hkCal}), answering a question of Yau \cite{YauComm}.  Combining these results we make deductions regarding the automorphisms of del Pezzo pairs $(\check{Y}, \check{D})$,  recovering, for example, a classical result concerning automorphisms of $\mathbb{P}^2$ fixing a plane cubic \cite{Web}.

We now outline the organization of this paper.  In Section~\ref{sec: semiflatmet} we discuss the basic properties of rational elliptic surfaces and semi-flat Calabi-Yau metrics which will play an essential role in this paper.  While some of this discussion has appeared elsewhere (e.g. \cite{GW, Hein}) we give a thorough discussion adapted to our applications. Section~\ref{sec: semiflatmet} contains the basic existence theorem for complete Calabi-Yau metrics asymptotic to (non-standard) semi-flat metrics, based on Hein's work \cite{Hein}, see Theorem~\ref{thm: nonStHein}.  Section~\ref{sec: semiflatmet} describes the de Rham and Bott-Chern moduli of K\"ahler metrics on the complement of an $I_k$ singular fiber in a RES, which is an important step in establishing our subsequent uniqueness results.   In Section~\ref{sec: slag} we prove the existence of special Lagrangian fibrations on rational elliptic surfaces equipped with complete Calabi-Yau metrics asymptotic to (quasi-regular) semi-flat metrics.  The arguments in this section build on the authors previous work \cite{Collins-Jacob-Lin}.  In Section~\ref{sec: moduli} we prove the basic uniqueness theorem for the Calabi-Yau metrics constructed in Section~\ref{sec: semiflatmet}, and discuss some applications.  In Section~\ref{sec: MS} we define the mirror map and combine our previous results to prove mirror symmetry for rational elliptic surfaces and del Pezzo surfaces and give some applications.  Finally, the paper concludes with two appendices.  Appendix~\ref{app: hkRot} computes explicitly the hyperK\"ahler rotation of the Calabi model, which serves as the asymptotic geometry for the Tian-Yau metric.  This calculation is used in Section~\ref{sec: moduli}. Appendix~\ref{app: HeinApp} discusses the relevant modifications to Hein's work \cite{Hein} needed to obtain the basic existence theorem in Section~\ref{sec: semiflatmet}.
\\
\\
{\bf Acknowledgements}:   The third author would like to thank P. Hacking for several very helpful discussions.  The authors are grateful to S.-T. Yau for his interest and encouragement.  The authors would also like to thank the referees, whose comments and suggestions have greatly improved the paper.

\section{Rational Elliptic Surfaces, Elliptic fibrations and semi-flat metrics}\label{sec: semiflatmet}

In this section we will collect some basic facts about the differential and algebraic geometry of rational elliptic surfaces.  The primary objective will be to explain the existence of Calabi-Yau metrics on rational elliptic surfaces, following Hein's work \cite{Hein}, and to understand the parameter space for these metrics which will be necessary to understand the moduli.    

Recall that, for $k\geq 2$, an $I_{k}$ singular fiber of an elliptic fibration is a wheel of $k \geq 2$ rational curves with self-intersection $(-2)$.  An $I_1$ singular fiber is a nodal rational curve, while an $I_0$ fiber is smooth.  In this paper an $I_k$ fiber will mean $I_{k}$ for $k\geq 1$.  Suppose that $\pi:Y \rightarrow \mathbb{P}^1$ is a rational elliptic surface with an $I_k$ singular fiber $D= \pi^{-1}(\infty)$ and let $X=Y\setminus D$.  Since $D \in|-K_{Y}|$ there is a unique (up to scale) rational $(2,0)$ form $\Omega$ on $Y$ with a simple pole on $D$.   To describe the Ricci-flat K\"ahler metric on $X$ we need to first explain the local models near $D$.  These metrics, called semi-flat K\"ahler metrics, were discovered by Greene-Shapere-Vafa-Yau \cite{GSVY} and have been subsequently studied by several authors (see, e.g. \cite{GW, GTZ, Hein}).  Our discussion of the semi-flat metrics will follow \cite{GW, Hein}, but with a slightly different emphasis, suited for our later purposes.

\subsection{The model fibration}

Let $\Delta = \{ z\in \mathbb{C}: |z|<1\}$ and suppose we have an elliptic fibration $\pi: X_{\Delta} \rightarrow \Delta$, which we assume has no singular fibers in $X_{\Delta^*}= \pi^{-1}(\Delta^* )$, where $\Delta^*= \Delta \setminus\{0\}$ is the punctured disk.  We suppose that $\pi^{-1}(0) = D$ is a singular fiber of type $I_k$.  Let $\sigma: \Delta^*\rightarrow X_{\Delta^*}$ be a holomorphic section of the fibration.  The choice of $\sigma$ fixes a natural origin in each fiber of $X_{\Delta^*}$.  By choosing a basis of $H_1(\pi^{-1}(z))$ compatible with the topological monodromy we can use the Abel-Jacobi map with respect to $\sigma$, denoted $F_{AJ,\sigma}$, to obtain a holomorphic map identifying
\[
F_{AJ,\sigma} : X_{\Delta^*} \rightarrow X_{mod}
\]
where $\pi_{mod}:X_{mod}\rightarrow \Delta^*$ is the {\em model fibration}
\[
X_{mod} := \mathbb{C}/ \Lambda(z), \qquad \Lambda (z) := \mathbb{Z} \oplus \mathbb{Z}\frac{ k}{2\pi \sqrt{-1}}\log(z).
\]
Note that, by definition, $F_{AJ,\sigma}: \pi^{-1}(z) \rightarrow \pi_{mod}^{-1}(z)$.  Under this identification the section $\sigma:\Delta^*\rightarrow X_{\Delta^*}$ is mapped to the canonical zero section of $X_{mod}$. 

\begin{rk}
The choice of a section $\sigma:\Delta^*\rightarrow X_{\Delta^*}$ gives each fiber the structure of a complex Lie group.  Thus, for any holomorphic section $\sigma': \Delta^*\rightarrow X_{\Delta^*}$ we get fiberwise translation maps
\[
T_{\sigma'}: X_{\Delta^*} \rightarrow X_{\Delta^*} \qquad T_{\sigma'}(x,z) = (x+\sigma'(z), z).
\]
While these maps are only defined relative to the choice of section $\sigma$, we will suppress this dependence.  We hope this does not cause any confusion.
\end{rk}

It will be useful to have

\begin{lem}[Lemma 3.28, \cite{FM}]\label{lem: localSecDes}
After fixing a section $\sigma: \Delta^*\rightarrow X_{\Delta^*}$, and identifying $X_{\Delta^*}$ with $X_{mod}$, any holomorphic section $\eta: \Delta^* \rightarrow X_{mod}$ can be identified with a multivalued holomorphic function
\[
\eta(z) = h(z) + \frac{a}{2\pi \sqrt{-1}}\log z + \frac{b}{(2\pi\sqrt{-1})^2} (\log(z))^2
\]
where $h(z) : \Delta^*\rightarrow \mathbb{C}$ is holomorphic, $\frac{2b}{k} \in \mathbb{Z}$, and $a+b \in \mathbb{Z}$.
\end{lem}

We will need the following definition, following \cite[Definition 1.2]{Hein}

\begin{defn}\label{defn: badCycle}
Consider the topological monodromy representation of $\pi_1(\Delta^*)$, which is conjugate to $\begin{pmatrix} 1 & k\\ 0 & 1\end{pmatrix}$ in the mapping class group of any fiber $F$ over $\Delta^*$. Let $\gamma \subset F$ be a simple loop such that $[\gamma]\in H_1(F,\mathbb{Z})$ is indivisible and invariant under the monodromy.
\begin{itemize}
\item[$(i)$] A bad 2-cycle $[C] \in H_{2}(X, \mathbb{Z})$ is one that arises from the following process, up to isotopy. With $\gamma$ as above, move $\gamma$ around the puncture by lifting a simple loop $\overline{\gamma}\subset \Delta^*$ up to every point in $\gamma$ such that the union of the translates of $\gamma$ is a $2$-torus embedded in $\pi^{-1}(\overline{\gamma})$.
\item[$(ii)$] A {\em quasi}-bad 2-cycle $[C] \in H_{2}(X, \mathbb{Z})$ arises in the same way, except we lift a {\em non}-simple loop $\overline{\gamma}\subset \Delta^*$ up to every point in $\gamma$ such that the union of the translates of $\gamma$ is a $2$-torus embedded in $\pi^{-1}(\overline{\gamma})$.  We will say that a quasi-bad cycle $C$ is an $m$-quasi-bad cycle if it covers a simple loop in the base $m$-times.
\end{itemize}
In either case, we orient $C$ so that
\[
\int_{C}\frac{dx\wedge dz}{z}  \in \sqrt{-1}\mathbb{R}_{>0}.
\]
\end{defn}

For an $I_k$ singular fiber the topological monodromy representation on $H_1(F,\mathbb{Z})$ admits a unique invariant loop in each fiber.  On the other hand, it is not hard to see that there are {\em many} choices of bad (and quasi-bad) cycle.  Concretely, if we fix a section $\sigma: \Delta^* \rightarrow X_{\Delta^*}$, and identify $X_{\Delta^*}$ with $X_{mod}$ via $F_{AJ,\sigma}$.  The cycle $C_0(\ell):=\{{\rm Im}(x)= 0, |z|= \ell\}$ is one possible choice for the bad cycle in $X_{mod}$, (and hence $F_{AJ, \sigma}^{-1}(C_0(\ell))$ is a bad cycle in $X$).  But, for any $m\in \mathbb{Z}$, one could equally well take
\[
C_{m}(\ell) :=\left \{ {\rm Im}(x)= m(-\frac{k}{2\pi}\log\ell) \frac{\theta}{2\pi}, \quad |z|=\ell \right\}.
\]
The cycle $F_{AJ,\sigma}^{-1}(C_{m}(\ell))$ is another choice of bad cycle in $X$.  It is not hard to see that, up to isotopy and orientation, these are all possible choices of bad cycles.  

The quasi-bad cycles can be treated similarly.  Given $m_1,m_2 \in \mathbb{Z}$ relatively prime with $m_1>0$, consider 
\[
C_{m_1,m_2}(\ell) := \left \{ {\rm Im}(x)= \frac{m_2}{m_1}(-\frac{k}{2\pi}\log\ell) \frac{\theta}{2\pi}, \quad |z|=\ell \right\}.
\]
which is an embedded torus covering the loop $\{|z|=\ell\}$ $m_1$ times.  Again, up to orientation and isotopy, it is easy to see that these examples constitute all quasi-bad cycles.  Note that for different values of $\ell$, one obtains isotopic cycles, and hence, to ease notation, we will suppress the dependence on $\ell$ when it is irrelevant. 

\begin{lem}\label{lem: badCycleIntegerAct}
With notation as above, the following holds:
\begin{itemize}
\item[$(i)$] All bad $2$-cycles in $H_{2}(X, \mathbb{Z})$ are isotopic to $F_{AJ, \sigma}^{-1}(C_{m})$ for $m\in \mathbb{Z}$ and, for any two bad cycles $[C],[C']$ we have
\[
[C]-[C'] = m [F]
\]
for some $m \in \mathbb{Z}$, where $[F]$ is the class of the fiber. 
\item[$(ii)$] All quasi-bad $2$-cycles in $H_2(X,\mathbb{Z})$ are isotopic to $F_{AJ, \sigma}^{-1}(C_{m_1,m_2})$ for $m_1,m_2\in \mathbb{Z}$ relatively prime.  If we fix a bad cycle $C$, then any quasi-bad cycle can be written as
\[
[C_{m_1,m_2}] = m_1[C]+m_2[F] \in H_2(X_{\Delta^*}, \mathbb{Z}).
\]

 \item[$(iii)$]  Every bad cycle (resp. quasi-bad cycle) is isotopic to a cycle $F_{AJ, \sigma}^{-1} T_{\eta_\ell}^{-1}(C_{0})$ where $\eta_\ell  = \frac{\ell k}{2(2\pi\sqrt{-1})^2} (\log(z))^2$ and $C_{0} = \left\{ Im(x)=0, |z|={\rm const} \right\} \subset X_{mod}$ (resp. $F_{AJ,\sigma}^{-1} T_{\eta_{\ell}}^{-1}(C_{m_1, m_2})$ where $m_1, m_2 \in \mathbb{Z}_{>0}$ are relatively prime, $\frac{m_2}{m_1} \in (0,1)$).

 \end{itemize}
\end{lem}
\begin{proof}
To prove $(i)$, we only need to prove the statement for bad cycles in $(iii)$.  This follows immediately from our description of the bad cycles. If we write $z=re^{\sqrt{-1}\theta}$ then
\[
\tilde{x} := T_{\eta_\ell}(x) = x+ \left(\frac{\ell k}{8\pi^2}(\theta^2-(\log(r))^2) + \sqrt{-1}\ell \left(\frac{-k}{2\pi} \log r\right) \frac{\theta}{2\pi}\right)
\]
and so
\[
T_{\eta_\ell}^{-1}(C_0) = \left\{ {\rm Im}(x) = -\ell \left(\frac{-k}{2\pi} \log r\right) \frac{\theta}{2\pi}, \quad |z|={\rm const}\right\},
\]
which is what we claimed.

The proof of $(ii)$ is similar, using the fact that any rational number $q$ can be written as $q= q'+\ell$ where $q'\in (0,1)$, and $\ell = \lfloor q \rfloor$.
\end{proof}

As a corollary of this result it makes sense to define

\begin{defn}\label{defn: induceBadCycle}
We say that a bad $2$-cycle $C$ is induced by a section $\sigma: \Delta^* \rightarrow X_{\Delta^*}$ if
\[
[C]= [F_{AJ, \sigma}^{-1}(\left\{ {\rm Im}(x)=0, |z|= {\rm const} \right\})] \in H_2(X_{\Delta^*},\mathbb{Z})
\]
where $(x,z)$ are the usual coordinates on $X_{mod}$.  By Lemma~\ref{lem: badCycleIntegerAct} all bad $2$-cycles are induced by $\sigma$ for some choice of $\sigma$.
\end{defn}

\subsection{Compactifications and coordinates near the $I_k$ fiber}\label{sec: compAndCoord}

It is well-known (see, for example \cite{FM}) that any choice of a local section $\sigma:\Delta^* \rightarrow X_{\Delta^*}$ induces a compactification $Y_{\sigma} \supset X$ where $Y_{\sigma}$ is a rational elliptic surface such that $Y_{\sigma}\setminus X$ consists of an $I_k$ fiber and $\sigma$ extends to a holomorphic section $\overline{\sigma}:\Delta \rightarrow Y_{\sigma}$.  We will briefly explain how this is done since it will also allow us to describe local coordinates near the $I_k$ fiber.

Consider the open set $U = \{(u,v) \in \mathbb{C}^2 : |uv|<1\}$.  Take $k+1$ copies of this set $U_i$, indexed by $i=0,\ldots,k$ with coordinates $(u_i,v_i)$.  Glue these sets according to the map
\[
\{(u_i,v_i) \in U_i : v_i\ne 0\} \longmapsto (v_i^{-1}, u_iv_i^2) \in U_{i+1},
\]
where we identify $(u_0,v_0) = (u_k,v_k)$, and let $\mathcal{X}$ be the resulting space.  It is not hard to check that $\mathcal{X}$ is smooth. Note that the identifications preserve the product $u_iv_i $, and so there is a well-defined map $\pi:\mathcal{X} \rightarrow \Delta$ defined by $\pi(u_i,v_i) = u_iv_i=:z$.  From the definition of $z$ we can write
\[
(u_i,v_i) = (u_0z^{-i}, v_0 z^i)
\]
and thus for $z\in \Delta^*$ we have $\pi^{-1}(z) = \mathbb{C}^*/(z^k)$\footnote{Here we use the multiplicative structure on $\mathbb{C}^*$.}.  If we set $u_0 = e^{2\pi\sqrt{-1}x}, v_0= ze^{-2\pi\sqrt{-1}x}$, then we get an identification 
\[
\pi^{-1}(\Delta^*) = \mathbb{C}/ \left( \mathbb{Z} \oplus \mathbb{Z}\frac{ k}{2\pi \sqrt{-1}}\log(z)\right).
\]
Finally, one checks that $\pi^{-1}(0)$ is a chain of $k$ rational curves intersecting transversally, given in $U_i$ by the sets $\{u_i=0\}$ and $\{v_i=0\}$ with each of the $k$ intersection points corresponding to the origin for $0\leq i \leq k-1$.  We can therefore use the coordinates $(u_i, v_i)$, $0\leq i\leq k-1$ as coordinates near the $I_k$ singular fiber.  The $(2,0)$ form $\frac{dx\wedge dz}{z}$ defines a meromorphic $(2,0)$ form on $\mathcal{X}$ with a simple pole on the $I_k$ fiber, since
\[
\frac{dx\wedge dz}{z} = \frac{du\wedge dv}{uv}.
\]

As before, we fix a section $\sigma:\Delta^*\rightarrow X_{\Delta^*}$, and identify $X_{\Delta^*}$ with the model fibration.  Using $\sigma$ we can glue $X$ to the above compactification of $X_{mod}$ resulting in a compact complex space $Y_{\sigma}$.  It is not hard to show (using, for example, classification of surfaces, see e.g. the proof of \cite[Theorem 6.4]{Collins-Jacob-Lin}) that $Y_{\sigma}\rightarrow \mathbb{P}^1$ is a rational elliptic surface.  Furthermore, since $\sigma$ is identified with the canonical zero section in $X_{mod}$ it trivially extends to a section $\overline{\sigma}:\Delta \rightarrow Y_{\sigma}$.  In fact, we have

\begin{lem}\label{lem: allCompSame}
Fix a section $\sigma: \Delta^*\rightarrow X_{\Delta^*}$.  Let $\eta: \Delta^* \rightarrow X_{\Delta^*}$ be another section, which is identified with the multivalued holomorphic function
\[
\eta(z) = h(z) + \frac{a}{2\pi \sqrt{-1}}\log z + \frac{b}{(2\pi\sqrt{-1})^2} (\log(z))^2
\]
as in Lemma~\ref{lem: localSecDes}.  Then
\begin{itemize}
\item[$(i)$] The compactifications $Y_{\sigma}, Y_{\eta}$ are biholomorphic.
\item[$(ii)$] $\eta$ extends to a holomorphic section $\overline{\eta}: \Delta \rightarrow Y_{\sigma}$ if and only if $h$ extends to a holomorphic function on $\Delta$, $b=0$ and $a\in \mathbb{Z}$.
\end{itemize}
\end{lem}
\begin{proof}
There are several ways to prove statement $(i)$ (e.g. by appealing to \cite[Chapter 7, Theorem 8]{Fried}).  But for the sake of concreteness we describe an explicit construction of the map $ Y_{\sigma} \rightarrow Y_{\eta}$.

First, assume that $\eta$ extends to a holomorphic section on $Y_{\sigma}$.  Explicitly, we have
\[
Y_{\sigma} = X \cup_{\sigma} \mathcal{X} 
\]
where we identify points in $X_{\Delta^*} \subset X$ with points in $X_{mod}\subset \mathcal{X}$ via the fiberwise Abel-Jacobi map relative to $\sigma$.  Consider the map
\[
\Phi_{(\sigma,\eta)}: Y_{\sigma} \rightarrow Y_{\eta}
\]
given by
\[
\Phi_{(\sigma, \eta)}(x) := \begin{cases} x & \text{ if } x \in X\\
T_{-\eta}(x) & \text{ if } \in \mathcal{X}
\end{cases}
\]
where $T_{-\eta}$ denotes the  inverse of translation by $\eta$ with respect to $\sigma$. Note that since $\eta$ extends to a holomorphic section on $Y_{\sigma}$, $T_{-\eta}$ induces a well-defined holomorphic map $T_{-\eta}:\mathcal{X} \rightarrow \mathcal{X}$.  To show that $\Phi$ gives a holomorphic map $\Phi:Y_{\sigma}\rightarrow Y_{\eta}$ it suffices to show that it is well-defined.  Explicitly, we must show that, if $x\in X_{\Delta^*}$ then
\[
x = F_{AJ, \eta}^{-1}\circ T_{-\eta}\circ  F_{AJ, \sigma}(x)
\]
but this is a tautology. 

Next consider the case when $\sigma, \eta : \mathbb{C} \rightarrow X$ are global holomorphic sections.  Let $T_{\eta}$ denote the holomorphic map which agrees with translation by $\eta$ with respect to $\sigma$ on the smooth fibers of $X$.  Note since $\sigma, \eta$ are global holomorphic sections, it is an easy consequence of the Riemann extension theorem that $T_{\eta}$ extends to an invertible holomorphic map $T_{\eta}: X\rightarrow X$.  Consider the map
\[
\Psi_{(\sigma, \eta)}(x):= \begin{cases} T_{\eta}(x) & \text{ if } x \in X\\
 x &\text{ if } x\in \mathcal{X}
 \end{cases}
 \]
 Again, to show that $\Psi_{(\sigma, \eta)}$ defines a holomorphic map $\Psi_{(\sigma, \eta)}: Y_{\sigma} \rightarrow Y_{\eta}$ it suffices to show that the map is well-defined.  Explicitly, we need to show that if $x \in X_{\Delta^*}$ then
 \[
 F_{AJ, \eta}^{-1}\circ F_{AJ, \sigma}(x) = T_{\eta}(x).
 \]
  Again, this is a tautology.  
  
  Now we address the general case.  Let $\sigma, \eta : \Delta^*\rightarrow X_{\Delta^*}$ be local holomorphic sections, and let $Y_{\sigma}$ (resp. $ Y_{\eta}$) be the rational elliptic surface obtained by compactifying $X$ using $\sigma$ (resp. $\eta$).  Since $Y_{\sigma}, Y_{\eta}$ are rational elliptic surfaces, they admit global holomorphic sections $\tilde{\sigma}: \mathbb{P}^1\rightarrow Y_{\sigma}$, and $\tilde{\eta}:\mathbb{P}^1 \rightarrow Y_{\eta}$.  Let $Y_{\tilde{\sigma}}, Y_{\tilde{\eta}}$ be the compactifications of $X$ obtained from $\tilde{\sigma}, \tilde{\eta}$.  Then we consider the composition
  \[
\Phi^{-1}_{(\eta, \tilde{\eta})} \circ \Psi_{(\tilde{\sigma}, \tilde{\eta})} \circ \Phi_{(\sigma, \tilde{\sigma})}: Y_{\sigma} \longrightarrow Y_{\eta}
  \]
  This map is holomorphic with holomorphic inverse, and gives the desired biholomorphic map.    
  Statement $(ii)$ is  \cite[Lemma 3.28]{FM}.  

\end{proof}

\subsection{The standard semi-flat metric}

To describe the standard semi-flat metric it is useful to pass to the universal cover of the model fibration.  Therefore, we define a coordinate $y= -\log(z)$, and let
\[
\mathcal{H}_{>0}= \{ y\in \mathbb{C} : {\rm Re}(y)>0\}
\]
which we regard as the universal cover of $\Delta^*$.  Let $x$ denote the standard coordinate on $\mathbb{C}$.  If $\Omega$ is a holomorphic $(2,0)$ form on $X_{\Delta^*}$ with a simple pole on $\pi^{-1}(0)$, then by the calculation in the preceding subsection we can write
\[
\Omega = \kappa(z) \frac{dx\wedge dz}{z}
\]
for $\kappa(z): \Delta \rightarrow \mathbb{C}$ a holomorphic function with $\kappa(0)\ne 0$.  By scaling we may as well assume that $\kappa(0)=1$. We pull $\Omega$ back to the holomorphic volume form $\Omega= \kappa(e^{-y})dy\wedge dx$ on the universal cover $\mathbb{C}\times \mathcal{H}_{>0}$.

\begin{defn}\label{defn: sf metric}
Fix $\epsilon >0$.  The {\em standard model semi-flat metric} for an $I_k$ fiber relative to the holomorphic volume form $\Omega$ is the metric on the covering elliptic fibration 
\[
\mathbb{C}/\left(\mathbb{Z} \oplus\frac{\sqrt{-1}ky}{2\pi} \mathbb{Z}\right) \rightarrow \mathcal{H}_{>0}
 \]
 given by the formula
\[
\begin{aligned}
\omega_{sf, \epsilon} &:=  \sqrt{-1}|\kappa(e^{-y})|^2 W^{-1}\frac{dy\wedge d\bar{y}}{\epsilon}\\
&\quad + \frac{\sqrt{-1}}{2} W\epsilon \left(dx-\sqrt{-1}\frac{{\rm Im}(x)}{{\rm Re}(y)}dy\right)\wedge \overline{\left(dx-\sqrt{-1}\frac{{\rm Im}(x)}{{\rm Re}(y)}dy\right)}
\end{aligned}
\]
where
\[
W=\frac{4\pi}{k(y+\bar{y})}. 
\]
\end{defn}

We can write the standard semi-flat metric in terms of the coordinates $(x,z)$ on $X_{mod}$ as
\[
\begin{aligned}
\omega_{sf,\epsilon} &=  \sqrt{-1}|\kappa(z)|^2 \frac{k|\log|z||}{2\pi\epsilon}\frac{dz\wedge d\bar{z}}{|z|^2} \\
&\quad + \frac{\sqrt{-1}}{2} \frac{2\pi\epsilon}{k|\log|z||} \left(dx+B(x,z)dz\right)\wedge \overline{\left(dx+B(x,z)dz\right)}
\end{aligned}
\]
where $B(x,z) = -\frac{{\rm Im}(x)}{\sqrt{-1}z|\log|z||}$. Here $2\pi\epsilon$ is the area of the elliptic fibre with respect to the semi-flat metric $\omega_{sf,\epsilon}$. The diameter (and injectivity radius) of $\pi^{-1}(z)$ is asymptotic to $2\pi \epsilon$ (and $(\log{|z|})^{-1}$ respectively).

\begin{defn}\label{defn: rel sf metric}
Fix $\epsilon >0$.  The {\em standard} semi-flat metric (for an $I_k$ fiber) on $\pi: X_{\Delta^*}\rightarrow \Delta^*$, relative to a section $\sigma:\Delta^* \rightarrow X_{\Delta^*}$ and the holomorphic volume form $\Omega$ is defined by
\[
\omega_{sf, \sigma, \epsilon} := F_{AJ, \sigma}^*\omega_{sf, \epsilon}.
\]
\end{defn}

A point we are trying to emphasize is that the standard model semi-flat metric is defined on the universal cover of the model fibration $\pi: X_{mod}\rightarrow \Delta^*$, {\em not} on the fibration $\pi: X_{\Delta^*} \rightarrow \Delta^*$. A semi-flat metric {\em relative to the section} $\sigma$, $\omega_{sf, \sigma}$ is then induced by using $\sigma$ to identify $X$ and $X_{mod}$.   The model semi-flat metric is Ricci-flat on $X_{mod}$, and {\em flat} along the fibers of $\pi_{mod}$ (hence the title ``semi-flat"), and thus the same holds for any induced semi-flat metric $\omega_{sf, \sigma,\epsilon}$ on $X_{\Delta^*}$.

\subsection{Non-standard semi-flat metrics}\label{subsec: nonSt}

Let us describe a construction of semi-flat metrics which are not standard, in the sense of Definition~\ref{defn: sf metric}. As before we pass to the universal cover of the model fibration,  $\mathbb{C}\times \mathcal{H}_{>0}$, with coordinates  $y= -\log(z)$, and  $x$ the standard coordinate on $\mathbb{C}$, and let $\Omega= \kappa(e^{-y})dy \wedge dx$ be the holomorphic $(2,0)$ form on the universal cover $\mathbb{C}\times \mathcal{H}_{>0}$.  Let $h(y)$ be a holomorphic function on $\mathcal{H}_{>0}$ and consider the translation map
\[
T_{h}(x,y) = (x+h(y), y).
\]
Let $\omega_{sf,\epsilon}$ denote the standard semi-flat metric on $\mathbb{C}\times \mathcal{H}_{>0}$, as in Definition~\ref{defn: sf metric}.  Pulling back $\omega_{sf,\epsilon}$ by $T_{h}$ yields
\begin{equation}\label{eq: pullBackUseApp}
\begin{aligned}
T_{h}^*\omega_{sf,\epsilon} &= \omega_{sf,\epsilon} + \frac{\sqrt{-1}W}{2}\epsilon \overline{\left( h'(y)-\sqrt{-1}\frac{{\rm Im}(h)}{{\rm Re}(y)}\right)}dx\wedge d\bar{y}\\
&\quad + \frac{\sqrt{-1}W}{2}\epsilon \left( h'(y)-\sqrt{-1}\frac{{\rm Im}(h)}{{\rm Re}(y)}\right)dy\wedge d\bar{x}\\
&\quad +\frac{\sqrt{-1}W}{2}\epsilon \bigg|\left( h'(y)-\sqrt{-1}\frac{{\rm Im}(h)}{{\rm Re}(y)}\right)\bigg|^2 dy\wedge d\bar{y}.
\end{aligned}
\end{equation}
Now we ask: under what conditions does $T_{h}^*\omega_{sf,\epsilon}$ induce a well-defined K\"ahler metric on $X_{mod}$?  Since $\omega_{sf, \epsilon}$ is invariant under translations by the lattice $\Lambda(z)$, this is evidently true for $T_{h}^*\omega_{sf, \epsilon}$ as well (since translations commute).  Therefore the only requirement for $T_{h}^*\omega_{sf,\epsilon}$ to descend to $X_{mod}$ is that the function
\[
G(y) :=  h'(y)-\sqrt{-1}\frac{{\rm Im}(h(y))}{{\rm Re}(y)}
\]
satisfies $G(y)= G(y+2\pi \sqrt{-1})$.  We have the following simple lemma.

\begin{lem}\label{lem: qrsfTrans}
Suppose $h: \mathcal{H}_{>0} \rightarrow \mathbb{C}$ is a holomorphic function such that
\[
 h'(y)-\sqrt{-1}\frac{{\rm Im}(h(y))}{{\rm Re}(y)}= h'(y+2\pi \sqrt{-1})-\sqrt{-1}\frac{{\rm Im}(h(y+2 \pi\sqrt{-1}))}{{\rm Re}(y)}.
\]
Then 
\begin{equation}\label{eq: qrsfTrans}
h(y)= f(e^{-y})+ \frac{c_0}{2\pi\sqrt{-1}} y+ \frac{b_0}{(2\pi\sqrt{-1})^2} y^2
\end{equation}
for some holomorphic function $f(e^{-y})$ and $b_0, c_0 \in \mathbb{R}$.
\end{lem}
\begin{proof}
Consider the function $\tilde{h}(y) := h(y+2\pi\sqrt{-1})- h(y)$.  By assumption we have
\[
\tilde{h}'(y)-\sqrt{-1}\frac{{\rm Im}(\tilde{h}(y))}{{\rm Re}(y)}=0.
\]
It follows easily that $\tilde{h}(y)= a_0 + \frac{a_1}{2\pi\sqrt{-1}} y$ for some $a_0, a_1 \in \mathbb{R}$, \cite[p. 516]{GW}.  Let $2b_0=a_1$ and $c_0+b_0 =a_0$  and consider
\[
H(y) = h(y) - \frac{c_0}{2\pi\sqrt{-1}} y - \frac{b_0}{(2\pi\sqrt{-1})^2} y^2.
\]
Then
\[
H(y+2\pi\sqrt{-1}) -H(y) = a_0 + \frac{a_1}{2\pi\sqrt{-1}} y - c_0 - \frac{2b_0}{2\pi\sqrt{-1}} y -b_0=0.
\]
Thus $H(y)$ is a holomorphic function invariant under $y \mapsto y+2\pi\sqrt{-1}$ and so we can write $H(y) = f(e^{-y})$ for some holomorphic function $f$. 
\end{proof}

Let $h$ be as in Lemma~\ref{lem: qrsfTrans}, and assume for simplicity that $f=0$ in~\eqref{eq: qrsfTrans}.  We define
\begin{equation}\label{eq: non-St SF}
\begin{aligned}
&\omega_{sf, b_0, \epsilon}:= T_{h}^*\omega_{sf, \epsilon} =  \sqrt{-1}\frac{|\kappa(e^{-y})|^2}{\epsilon} W^{-1}dy\wedge d\bar{y}\\
&\quad + \frac{\sqrt{-1}}{2} W\epsilon \left(dx-\Gamma(x,y,b_0) dy\right) \wedge \overline{\left(dx-\Gamma(x,y,b_0) dy\right)}
\end{aligned}
\end{equation}
where
\[
\Gamma(x,y,b_0) =  \sqrt{-1}\frac{{\rm Im}(x)}{{\rm Re}(y)}+ \frac{b_0}{2\pi^2}{\rm Re}(y).
\]

We can also rewrite this in terms of the holomorphic coordinates $(x,z)$, as
\begin{equation}\label{eq: non-St SF2}
\begin{aligned}
\omega_{sf, b_0, \epsilon} &= \sqrt{-1}\frac{|\kappa(z)|^2}{\epsilon} W^{-1}\frac{dz\wedge d\bar{z}}{|z|^2}\\
&\quad + \frac{\sqrt{-1}}{2} W \epsilon \left(dx+\widetilde{\Gamma}(x,z,b_0) dz\right) \wedge \overline{\left(dx+\widetilde{\Gamma}(x,z,b_0) dz\right)}
\end{aligned}
\end{equation}
where, $W= \frac{1}{\frac{k}{2\pi}|\log|z||}$ as before, and
\[
\widetilde{\Gamma}(x,z,b_0) = B(x,z)dz+ \frac{b_0}{2\pi^2}\frac{|\log|z||}{z},\qquad  B(x,z) = -\frac{{\rm Im}(x)}{\sqrt{-1}z|\log|z||}. 
\]

\begin{rk}
It is important to note that if $\frac{2b_{0}}{k} \notin \mathbb{Z}$, then $T_{h}^*\omega_{sf,\epsilon}$ is {\em not} a standard semi-flat metric on $X_{mod}$ for any choice of section $\sigma : \Delta^*\rightarrow X_{mod}$.
\end{rk}

\begin{defn}\label{defn: non-stand sf metric}
Fix $\epsilon >0$, and $b_0\in \mathbb{R}$ with  $\frac{2b_{0}}{k} \notin \mathbb{Z}$.   A {\em non-standard model semi-flat metric} (for an $I_k$ fiber) relative to the holomorphic volume form $\Omega$ is the metric on the covering elliptic fibration 
\[
\mathbb{C}/\left(\mathbb{Z} \oplus\frac{\sqrt{-1}ky}{2\pi} \mathbb{Z}\right) \rightarrow \mathcal{H}_{>0}
 \]
 given by the formula~\eqref{eq: non-St SF}.  We say that this metric is {\em quasi-regular} if $\frac{2b_0}{k} \in \mathbb{Q}\setminus \mathbb{Z}$, and irregular if $\frac{2b_0}{k} \in \mathbb{R}\setminus \mathbb{Q}$.
\end{defn}

\begin{defn}\label{defn: rel non-ST sf metric}
Fix $\epsilon >0$, and $b_0\in \mathbb{R}$ such that $\frac{2b_0}{k} \notin \mathbb{Z}$.  The {\em non-standard} semi-flat metric (for an $I_k$ fiber) on $\pi: X_{\Delta^*}\rightarrow \Delta^*$, relative to a section $\sigma:\Delta^* \rightarrow X_{\Delta^*}$ and the holomorphic volume form $\Omega$ is defined by
\[
\omega_{sf, \sigma, b_0, \epsilon} := F_{AJ, \sigma}^*\omega_{sf, b_0, \epsilon}.
\]
\end{defn}

\begin{rk}\label{rk: nonStTansbyMult}
A non-standard semi-flat metric $\omega_{sf,\sigma, b_0, \epsilon}$ can be viewed as a {\em standard} semi-flat metric $\omega_{sf, \sigma', \epsilon}$ but with $\sigma':\Delta^*\rightarrow X_{\Delta^*}$ a {\em multivalued} section. $\omega_{sf,\sigma, b_0, \epsilon}$ is quasi-regular precisely when $\sigma'$ is finitely-many valued, and irregular otherwise.
\end{rk}

The following lemma is meant to indicate the importance of non-standard semi-flat metrics for our purposes.  Namely, the lemma shows that the cohomology classes of standard semi-flat metrics always lie in a countable union of hyperplanes of $H^{2}_{dR}(X_{\Delta^*}, \mathbb{R})$.  On the other hand, cohomology classes of non-standard semi-flat metrics in fact {\em generate} $H^{2}_{dR}(X_{\Delta^*}, \mathbb{R})$; see Corollary~\ref{cor: local deRham}.  

\begin{lem}\label{lem: cohomProps}
Fix a section $\sigma: \Delta^*\rightarrow X_{\Delta^*}$ and let $[C] \in H_2(X_{\Delta^*},\mathbb{Z})$ be the bad cycle induced by $\sigma$, and let $[F]$ denote the class of the fiber. 
\begin{itemize}
\item[$(i)$] If $\omega_{sf, \sigma', \epsilon}$ is any standard semi-flat metric then 
\[
[\omega_{sf, \sigma', \epsilon}]_{dR}.[C] = m[\omega_{sf, \sigma', \epsilon}]_{dR}.[F]
\]
for some $m\in \mathbb{Z}$.
\item[$(ii)$] If $\omega_{sf,\sigma, b_0, \epsilon}$ is a non-standard semi-flat metric then 
\[
[\omega_{sf, \sigma, b_0, \epsilon}]_{dR}.[C]  = \frac{2b_0}{k} [\omega_{sf, \sigma, b_0, \epsilon}]_{dR}.[F],
\]
where $\frac{2b_0}{k} \notin \mathbb{Z}$.  In particular, $[\omega_{sf, \sigma, b_0, \epsilon}]_{dR}$ does not pair to zero with any bad-cycle.
\item[$(iii)$] If $\frac{2b_0}{k} \in \mathbb{Q}\setminus \mathbb{Z}$, then $[\omega_{sf, \sigma, b_0, \epsilon}]_{dR}.[C']=0$ for some {\em quasi}-bad cycle $[C']\in H_{2}(X_{\Delta^*}, \mathbb{Z})$.
\item[$(iv)$] If $\frac{2b_0}{k} \in \mathbb{R}\setminus \mathbb{Q}$, then $[\omega_{sf, \sigma, b_0, \epsilon}]_{dR}.[C]\ne 0$ for any cycle $[C]\in H_2(X_{\Delta^*}, \mathbb{Z})$.
\end{itemize}
\end{lem}
\begin{proof}
The first point is immediate, since if $[C']$ is the bad cycle induced by $[\sigma']$ then by Lemma~\ref{lem: badCycleIntegerAct} we have $[C]=[C']+m[F]$ for some $m\in \mathbb{Z}$.  Then the claim follows from the fact that $[\omega_{sf, \sigma', \epsilon}].[C']=0$.
 
The second point is a straightforward calculation.  Using $\sigma$ to identify $X_{\Delta^*}$ with $X_{mod}$ it suffices to compute the integral of $\omega_{sf, b_0, \epsilon}$ over the cycle $C= \{ {\rm Im}(x)=0, |z|=r\}$.  Writing $x= x_1+\sqrt{-1}x_2$,~\eqref{eq: non-St SF2} gives
\[
\begin{aligned}
\omega_{sf,b_0,\epsilon}\big|_{C} &= \frac{\sqrt{-1}}{2} \frac{\epsilon}{\frac{k}{2\pi}|\log r|} (dx_1 + \frac{b_0}{2\pi^2}|\log|r|| \sqrt{-1}d\theta)\wedge(dx_1 - \frac{b_0}{2\pi^2}|\log|r|| \sqrt{-1}d\theta)\\
&= \frac{\epsilon b_0}{k\pi} dx_1 \wedge d\theta
\end{aligned}
\]
Using the choice of orientation for $C$ and integrating yields the result.  

For the third claim, take $m_1, m_2 \in \mathbb{Z}$ relatively prime with $m_1>0$ so that $\frac{2b_0}{k} = -\frac{m_2}{m_1}$.  Then it is straightforward to check that the quasi-bad cycle $C_{m_1,m_2}= \{\frac{m_2}{m_1}(-\frac{k}{2\pi}\log r) \theta, |z|=r\}$ satisfies $\omega_{sf,b_0,\epsilon}\big|_{C_{m_1,m_2}} = 0$, and so $C_{m_1,m_2}$ is in fact Lagrangian for the symplectic form $\omega_{sf,b_0,\epsilon}$.

Finally, assume $\frac{2b_0}{k} \in \mathbb{R}\setminus \mathbb{Q}$.  By \cite[Lemma 4.3]{GW} (see also \cite[Claim 1]{Hein}), $H_2(X_{\Delta^*}, \mathbb{Z})$ is generated by $[C], [F]$.  If we consider any cycle $m_1[C]+m_2[F] \in H_2(X_{\Delta^*},\mathbb{Z})$, for $m_1, m_2 \in \mathbb{Z}$, we have
\[
[\omega_{sf, b_0, \epsilon}]. (m_1[C]+m_2[F] ) = \frac{2b_0}{k}m_1\epsilon + m_2 \epsilon \ne 0.
\]
\end{proof}

\begin{cor}\label{cor: local deRham}
Let $[\omega]_{dR} \in H^{2}_{dR}(X_{\Delta^*}, \mathbb{R})$ be the de Rham cohomology class of a smooth K\"ahler metric.  Then there exists a (possibly non-standard) semi-flat metric $\omega_{sf, \sigma, b_0, \epsilon}$ such that $[\omega]= [\omega_{sf, \sigma, b_0, \epsilon}]$.
\end{cor}
\begin{proof}
Fix a section $\sigma : \Delta^*\rightarrow X_{\Delta^*}$ and let $[C]$ be the bad cycle induced by $\sigma$.  By \cite[Lemma 4.3]{GW} (see also \cite[Claim 1]{Hein}) we have that $[\omega]_{dR}= [\omega_{sf, \sigma, b_0, \epsilon}]_{dR}$ if and only if 
\[
\begin{aligned}
\, [\omega]_{dR}.[F] &= [\omega_{sf, \sigma, b_0, \epsilon}]_{dR}.[F]\\
\, [\omega]_{dR}. [C] &= [\omega_{sf, \sigma, b_0, \epsilon}]_{dR}.[C]. 
\end{aligned}
\]
Let $\epsilon = [\omega]_{dR}.[F]$.  Then by definition we have $[\omega_{sf, \sigma, b_0, \epsilon}]_{dR}.[F]=\epsilon = [\omega]_{dR}.[F]$.  Now write $ [\omega].[C] = \frac{2b_0}{k}\epsilon$ for $b_0\in \mathbb{R}$. Then by Lemma~\ref{lem: cohomProps} we have $[\omega]_{dR}.[C]= [\omega_{sf, \sigma, b_0, \epsilon}]_{dR}.[C]$ and the corollary follows.
\end{proof}

One can easily check that the semi-flat metrics defined above are complete at the $I_k$ fiber.  Even more precisely, if we fix a point $p_0:= (x_0,z_0)\in X_{\Delta^*}$, then the distance to a point $p:=(x,z)\in X_{\Delta^*}$ with respect to a semi-flat metric is given by
\[
r(p)= d_{\omega_{sf}}(p_0,p) \sim (-\log|z|)^{3/2}.
\]
As a consequence, the semi-flat metric is complete as $z\rightarrow 0$.
Recall that, after fixing a section $\sigma:\Delta^*\rightarrow X_{\Delta^*}$, we get well-defined, fiber preserving translation maps $T_{\eta}:X_{\Delta^*}\rightarrow X_{\Delta^*}$.  It will be important for us to understand the effect of translation on the semi-flat metric. For a general choice of $\eta$, the translation $T_{\eta}$ is not an isometry for $\omega_{sf, \sigma, \epsilon}$, and can change the geometry drastically; see Lemma~\ref{lem: transAsymp} below.  However, it is a tautology that $T_{\eta}^*\omega_{sf, \sigma, \epsilon} = \omega_{sf, \widetilde{\sigma}, \epsilon}$ for some other choice of section $\tilde{\sigma}: \Delta^*\rightarrow X_{\Delta^*}$.  

Before stating the existence theorem recall that the Bott-Chern cohomology of a complex manifold is given by
\[
H^{p,q}_{BC}(X) := \frac{ {\rm Ker} d: \Lambda^{p,q} \rightarrow \Lambda^{p+1,q} \oplus \Lambda^{p,q+1}}{{\rm Im} \ddb: \Lambda^{p-1,q-1} \rightarrow \Lambda^{p,q}}.
\]
The Bott-Chern cohomology refines the de Rham cohomology, in the sense that $H^{p,q}_{BC}(X) \rightarrow H^{p+q}_{dR}(X)$.

The following result is due to Hein \cite{Hein} under two additional assumptions; see Remark~\ref{rk: HeinDiff} below and Appendix~\ref{app: HeinApp} for the necessary adaptions to obtain this result from Hein's work.

 \begin{thm}\label{thm: nonStHein}
    	Let $\pi:Y\rightarrow \mathbb{P}^1$ be a rational elliptic surface, $D$ a fiber of type $I_k$ over $\infty \in \mathbb{P}^1$.  Let $\Omega$ be the (unique up to scale) holomorphic volume form on $X:=Y\setminus D$ with a simple pole on $D$.  Let $\Delta^{*}$ denote a punctured neighborhood of $\infty$ and fix a holomorphic section $\sigma:\Delta^{*} \rightarrow X_{\Delta^*}$.  Let $[F]$ denote the homology class of a fiber, and $[C]$ be the bad $2$-cycle induced by $\sigma$.  Suppose that $\omega_0$ is any K\"ahler metric on $X$, and let
	\[
	[\omega_0]_{dR}.[F]=\epsilon,\qquad [\omega_0]_{dR}.[C]=\frac{2b_0}{k}\epsilon
	\]
	for some $b_0\in \mathbb{R}$. Then the following holds:
	\begin{itemize}
	\item[$(i)$] there is a holomorphic function $h(z):\Delta^*\rightarrow \mathbb{C}$ depending only on $[\omega_0]_{BC} \in H^{1,1}_{BC}(X,\mathbb{R})$,
	\item[$(ii)$] for all $\alpha >0$ there is a complete K\"ahler metric $\omega_{\alpha}$ on $X$ with $[\omega_{\alpha}]_{BC}= [\omega_0]_{BC}$ in $H^{1,1}_{BC}(X,\mathbb{R})$ such that, in a neighborhood of $D$ we have
	\[
	\omega_{\alpha} = \alpha T^*_{h}\omega_{sf, \sigma, b_0, \frac{\epsilon}{\alpha}} 
	\]
	and $\omega_{\alpha}= \omega_0$ outside a neighborhood of $D$.
	\end{itemize}
	
	\noindent Furthermore for all $\alpha > 0$ there exists a complete, Ricci-flat metric $\omega_{CY}$ on $X$ with 
	\[
	\omega_{CY} = \omega_{\alpha}  + \ddb \phi
	\]
	 satisfying the following:
    	\begin{enumerate}
		\item[$(iii)$] $\omega_{CY}$ solves the Monge-Amp\`ere equation 
		\[
		\omega_{CY}^2= \alpha^2 \Omega \wedge \overline{\Omega}. \nonumber
		\]
    		\item[$(iv)$] The injectivity radius of $\omega_{CY}$ has asymptotics $inj(x)\sim r(x)^{-1/3}$. 
    		\item[$(v)$] The curvature of $\omega_{CY}$ satisfies $|\nabla^k Rm|_{\omega_{CY}}\lesssim r^{-2-k}$ for every $k\in \mathbb{N}$. 
    		\item[$(vi)$] $\omega_{CY}$ is asymptotic to the semi-flat metric $\omega_{\alpha}$ above in the following sense; there is a constant $\delta>0$ such that, for every $k\in \mathbb{N}$, there holds
    		  \begin{align*}
    		      | \nabla^k\phi|_{\omega_{CY}} \sim O(e^{-\delta r^{2/3}}),
    		  \end{align*} 
    	\end{enumerate}
	where $r$ denotes the distance from a fixed base point in $X$. 
	\end{thm} 
	
	A few remarks are in order.
	
\begin{rk}\label{rk: HeinDiff}
Theorem~\ref{thm: nonStHein} is obtained by Hein \cite{Hein} under two additional assumptions:
\begin{itemize}
\item[$(i)$] The semi-flat metric is  standard, and $[\omega_0]_{dR}.[C']=0$ for some bad cycle $[C']$.  In particular, by Lemma~\ref{lem: cohomProps}, the cohomology class $[\omega_0]_{dR}$ lies in a countable union of hyperplanes in $H^2_{dR}(X,\mathbb{R})$.
\item[$(ii)$] The background K\"ahler form $\omega_0$ has $\int_{X}\omega_0^2 <+\infty$.  For a K\"ahler class which is not restricted from a compactification of $X$ (in particular, a K\"ahler class not pairing to zero with any bad cycle) it is not clear how to construct K\"ahler forms with this property.  In fact, we expect such forms cannot exist in general.
\item[$(iii)$] Hein's original result \cite{Hein} assumed $\alpha > \alpha_0$ for some sufficiently large $\alpha_0$.  Chen-Chen \cite{CC} explained how this assumption could be removed.
\end{itemize}
Due to these differences, we have opted to give a sketch of the relevant adaptions of Hein's work in Appendix~\ref{app: HeinApp}.
\end{rk}

\begin{rk}
The construction in Theorem~\ref{thm: nonStHein} depends on various background choices (for example, the background metric $\omega_0$).  In particular, Theorem~\ref{thm: nonStHein} may produce {\em many} Calabi-Yau metrics in any given Bott-Chern K\"ahler class  on $X$.  We will explicitly rule this out in Section~\ref{sec: moduli} by establishing a uniqueness theorem.
\end{rk}

\subsection{Dependence on parameters}

In order to prove the original formulation of SYZ mirror symmetry for rational elliptic surfaces we must construct a finite dimensional moduli space of Ricci-flat metrics on $X$.  The first step in this direction is to understand precisely the parameters in Theorem~\ref{thm: nonStHein}. The main parameters are:
\begin{enumerate}
\item A choice of section $\sigma:\Delta^*\rightarrow Y$.
\item A K\"ahler metric $\omega_0$ on $X$.
\item The Bott-Chern cohomology class $[\omega_0]_{BC} \in H^{1,1}_{BC}(X,\mathbb{R})$, which enters both through the de Rham and Bott-Chern cohomology classes of the Calabi-Yau metric and the holomorphic function $h(z)$ in Theorem~\ref{thm: nonStHein} part $(i)$.
\end{enumerate}

For the remainder of this section we will be primarily interested in understanding how the Calabi-Yau metrics produced by Theorem~\ref{thm: nonStHein} are parametrized.  Together with a uniqueness result proved in Section~\ref{sec: moduli}, this will yield a description of the moduli of Calabi-Yau metrics asymptotic to semi-flat metrics.

The most obvious parameters in appearing in Theorem~\ref{thm: nonStHein} are the de Rham and Bott-Chern cohomology class of K\"ahler metrics $\omega_0$ on $X$.  

\subsubsection{de Rham Cohomology}

We begin by calculating the de Rham cohomology.

\begin{lem}\label{lem: deRham}
Let $\pi: Y\rightarrow \mathbb{P}^1$ be a rational elliptic surface, and $\pi^{-1}(\infty)=D=\sum_{i=1}^{k} D_i$ be a singular fiber of type $I_k$, with irreducible components $D_i$.  Let $X=Y\setminus D$.  Fix a choice of bad cycle $C\subset X$, and define, for each $m\in \mathbb{Z}$
\[
V_{m} = \left\{ [\omega]_{dR} \in H^{2}_{dR}(X,\mathbb{R}): \omega \text{ is K\"ahler }, m[\omega].[F] =  [\omega].[C]\right\}.
\]
Then $V_m \cong V$ is (non-canonically) isomorphic to a fixed open cone $V\subset \mathbb{R}^{10-k}$ given by
\[
V=\mathcal{K}_{Y}/ \left\{ {\rm Span}_{\mathbb{R}}\{ D_i\}_{1 \leq i \leq k}\right\}
\]
where $\mathcal{K}_{Y}\subset H^{1,1}(Y,\mathbb{R})$ is the K\"ahler cone of $Y$. Furthermore,  each de Rham class in $V_m$ can be represented by the restriction of a K\"ahler metric from some rational elliptic surface $Y_{\sigma}$ compactifying $X$.
\end{lem}
\begin{proof}
Fix a cohomology class $[\omega_0]_{dR} \in V_{m}$ for some $m\in \mathbb{Z}$.  By assumption $[\omega_0]_{dR}.[F]=\epsilon$ and $[\omega_0]_{dR}.[C]= \epsilon m$.  By Lemma~\ref{lem: badCycleIntegerAct}, after changing the bad cycle to $[C'] =[C]-m[F] $, we can assume that $[\omega_0]_{dR}.[C']=0$ and we can choose a section $\sigma:\Delta^*\rightarrow X_{\Delta^*}$ so that $C'$ is induced by $\sigma$.  Let $Y_{\sigma}$ be the rational elliptic surface compactifiying $X$ induced by this choice of $\sigma$, and let $D_{\sigma}\subset Y_{\sigma}$ be the $I_k$ fiber so that $X= Y_{\sigma}\setminus D_{\sigma}$; recall that by Lemma~\ref{lem: allCompSame}, all such compactifications are biholomorphic.

Since $Y_{\sigma}$ is obtained by blowing-up $\mathbb{P}^2$ at the base locus of a smooth pencil of cubics, we have $H^{1,1}(Y_{\sigma},\mathbb{R}) \cong \mathbb{R}^{10}$.  Since $D_{\sigma}$ is a cycle of $k$-rational curves the restriction map $ H^{1,1}(Y_{\sigma},\mathbb{R}) \rightarrow H^{2}_{dR}(X,\mathbb{R})$ has a $k$-dimensional kernel, and $10-k$ dimensional image.  Note that $H^{2,0}(Y_{\sigma}) = H^{0}(Y_{\sigma}, K_{Y_{\sigma}})=0$ since $-K_{Y_{\sigma}}$ is effective.  Now we consider the exact sequence in relative homology
\[
\mathbb{Z} \simeq H_{3}(Y,X)\rightarrow H_2(X)\rightarrow H_2(Y_{\sigma})
\]
where $H_{3}(Y_{\sigma},X)$ is generated by the bad cycle $C'$ (which is a homologous to zero in $Y_{\sigma}$).  Dualizing we have
\begin{equation}\label{eq: cohExactSeq}
\mathbb{Z}^k \simeq H^2(Y_{\sigma},X) \rightarrow H^2(Y_{\sigma})\rightarrow H^{2}(X) \rightarrow H^3(Y_{\sigma},X) \simeq \mathbb{Z}
\end{equation}
Now any class in $H^{2}(X)$ which pairs to zero with the bad cycle $C'$ is in the image of the restriction map $H^{2}(Y_{\sigma}) \rightarrow H^{2}(X)$.  On the other hand, since $H^{2}(Y_{\sigma}) = H^{1,1}(Y_{\sigma})$, every closed $2$-form on $X$ pairing to zero with $C'$ is cohomologous to the restriction of a $(1,1)$ form on $Y_{\sigma}$.  Thus, we have
\[
V_{m} =\mathcal{K}_{Y_{\sigma}}/ \left\{ {\rm Span}_{\mathbb{R}}\{ (D_{\sigma})_i\}_{1 \leq i \leq k}\right\}.
\]
But since any two compactifications are biholomorphic by Lemma~\ref{lem: allCompSame}, we have
\[
V_m= \mathcal{K}_{Y_{\sigma}}/ \left\{ {\rm Span}_{\mathbb{R}}\{ (D_{\sigma})_i\}_{1 \leq i \leq k}\right\}=\mathcal{K}_{Y}/ \left\{ {\rm Span}_{\mathbb{R}}\{ (D_i\}_{1 \leq i \leq k}\right\}= V.
\]
\end{proof}

The main corollary of this result that will be relevant to us is

\begin{cor}\label{cor: KahlerCone}
Let $\mathcal{K}_{dR,X} \subset H^{2}_{dR}(X,\mathbb{R})$ denote the set of de Rham cohomology classes which can be represented by K\"ahler forms.  Then $\mathcal{K}_{dR, X}$ is a convex cone with non-empty interior in $H^{2}_{dR}(X, \mathbb{R}) \cong \mathbb{R}^{11-k}$.
\end{cor}
\begin{proof}
The proof is essentially trivial.  That $\mathcal{K}_{dR, X}$ is a convex cone is obvious, while $\dim H^{2}_{dR}(X, \mathbb{R}) = 11-k$ follows from~\eqref{eq: cohExactSeq}.  Finally, by Lemma~\ref{lem: deRham}, for each $m\in \mathbb{Z}$, the sets $V_{m}\subset \mathcal{K}_{dR, X}$ are disjoint, open convex cones contained in the $10-k$ dimensional hyperplanes $V_m$, for $m \in \mathbb{Z}$.  Since $\mathcal{K}_{dR, X}$ is convex, we get that
\begin{equation}\label{eq: bigUnionKahCone}
\text{Convex Hull}\left(\bigcup_{m\in\mathbb{Z}} V_m\right) \subset \mathcal{K}_{dR, X}.
\end{equation}
Since $V_{m}\cap V_{m'}=\emptyset$ if $m\ne m'$, we are done.
\end{proof}

\begin{rk}
It would be interesting to understand the K\"ahler cone $\mathcal{K}_{dR, X}$ completely, but a priori it could behave rather differently than in the compact case.  For instance, it is not even clear that $\mathcal{K}_{dR, X}$ is open in $H^{2}_{dR}(X,\mathbb{R})$.  In any event, it is not hard to see that the containment in~\eqref{eq: bigUnionKahCone} is proper since one can take convex combinations of K\"ahler forms and {\em semi-positive} forms on $X$.  Since $Y$ is obtained by blowing-up $\mathbb{P}^2$ at $9$ points, $Y$ (and hence $X$) admits many semi-positive classes.
\end{rk}

We also note the following corollary of Lemma~\ref{lem: deRham}, which indicates that, for a generic choice of K\"ahler metric on $X$, Theorem~\ref{thm: nonStHein} produces a complete Calabi-Yau metric asymptotic to a non-standard semi-flat metric.  Since the proof is a trivial consequence of Lemma~\ref{lem: cohomProps}, we omit it.

\begin{cor}
Suppose $[\omega_0], [\omega_1] \in H^{2}_{dR}(X,\mathbb{R})$ are K\"ahler classes such that there are bad $2$-cycles $[C_0], [C_1]$ with $[C_1]=[C_0]+m[F]$ for some $m\in \mathbb{Z}$ and
\[
[\omega_0].[C_0]=0, \qquad [\omega_1].[C_1]=0, \qquad [\omega_0].[F]= \epsilon = [\omega_1].[F].
\]
Consider the K\"ahler class $[\omega_{\delta}]:= [(1-\delta)\omega_0 + \delta \omega_1]$ obtained as a convex combination of $[\omega_0], [\omega_1]$.  For every $\delta \in [0,1]$ there is a semi-flat metric $\omega_{sf, \sigma, b_0, \epsilon}$ such that 
\[
[ \omega_{sf, \sigma, b_0, \epsilon}] = [\omega_{\delta}] \in H^2(X_{\Delta^*}, \mathbb{Z}).
\]
Furthermore, $\omega_{sf, \sigma, b_0, \epsilon}$ is non-standard unless $\delta m= n$ for some $n\in \mathbb{Z}$.
\end{cor}

\subsubsection{Bott-Chern Cohomology}
The goal of this subsection, and the next, is to elucidate the connection between local and global automorphisms and the Bott-Chern cohomology.  This will be essential in establishing the uniqueness of Calabi-Yau metrics modulo automorphisms; see Proposition~\ref{prop: reducedModuli}.  For our later purposes it will be useful to understand the Bott-Chern cohomology of $X$ and $X_{\Delta^*}$.  Therefore we will consider the elliptic fibration $X\rightarrow C$ where $C= \mathbb{C},$ or $\Delta^*$, with the understanding that $X\rightarrow \Delta^*$ is $X_{\Delta^*}$.  The exposition here follows analogous discussions in \cite{GW, Hein} with modifications to suit our purposes.  Consider the map 
\begin{equation}\label{eq: exactSeq1}
0\rightarrow K_0 \rightarrow H^{1,1}_{BC}(X,\mathbb{R}) \rightarrow H^{2}_{dR}(X,\mathbb{R}) \rightarrow 0
\end{equation}
where $K_0$ is defined to be the kernel of the natural map $H^{1,1}_{BC}(X,\mathbb{R}) \rightarrow H^{2}_{dR}(X,\mathbb{R})$; note that the surjectivity  of this map follows, for example,  from Corollary~\ref{cor: KahlerCone}.  Suppose $\alpha_1, \alpha_2$ are closed, real $(1,1)$ forms on $X$ and $[\alpha_1]_{dR}=[\alpha_2]_{dR}$.  Write
\[
\alpha_1-\alpha_2 = d\beta = \del \beta^{0,1} + \dbar\, \overline{\beta^{0,1}}
\]
where $\beta^{0,1}$ is a $(0,1)$-form with $\dbar \beta^{0,1}=0$. Note that $\beta$ is only well-defined modulo closed real $1$-forms, a point that we will return to shortly.  $\beta^{0,1}$ defines an element of $H^{0,1}_{\dbar}(X)$ and it is not hard to see that $[\alpha_1-\alpha_2]_{BC}= 0$ if and only if $\beta^{0,1}$ is the zero class in $H^{0,1}_{\dbar}(X)$.  Furthermore, for any $\dbar$-closed $(0,1)$ form $\beta$, $\alpha= d(\beta + \overline{\beta})$ is an exact, real $(1,1)$ form inducing $\beta$.  Therefore, we have a surjective map
\begin{equation}\label{eq: exactSeq2}
 H^{0,1}_{\dbar}(X) \rightarrow K_0 \rightarrow 0.
\end{equation}
Thus, to understand the Bott-Chern cohomology we first need to understand $H^{0,1}_{\dbar}(X)$. By the Dolbeault isomorphism we have $H^{0,1}_{\dbar}(X) \cong H^{1}(X,\mathcal{O}_{X})$ and the latter group can be understood using the Leray spectral sequence; 
\[
\begin{tikzcd}
 0\arrow[r]
 	&H^{1}(C,R^0\pi_{*}\mathcal{O}_{X}) \arrow[r] 
	\arrow[d, phantom]
	&H^{1}(X, \mathcal{O}_{X}) \arrow[dll, rounded corners, controls={+(0.5,-1) and +(-.3,1.6)}]\\
	 H^{0}(C,R^{1}\pi_{*}\mathcal{O}_{X}) \arrow[r]
	 & H^{2}(X, R^0\pi_{*}\mathcal{O}_{X}) \arrow[r]
	 & H^{2}(X, \mathcal{O}_{X}).
\end{tikzcd}
\]
It follows from \cite[Chapter I, Lemma 3.18]{FM} that the direct image sheaf $R^{1}\pi_{*}\mathcal{O}_{Y} = \mathcal{O}_{\mathbb{P}^1}(-1)$.  Since the restriction of this line bundle to $C$ is trivial, we can identify (non-canonically) $H^{0}(C,R^{1}\pi_{*}\mathcal{O}_{X})\cong H^{0}(C,\mathcal{O}_{C})$.  On the other hand $R^0\pi_{*}\mathcal{O}_{X}= \mathcal{O}_{C}$, and $H^{i}(C,\mathcal{O}_{C})=0$ for $i\geq 1$ since $C$ is Stein.  We therefore obtain
\[
H^{0,1}_{\dbar}(X) \cong H^{1}(X, \mathcal{O}_{X}) \cong  H^{0}(C,R^{1}\pi_{*}\mathcal{O}_{X}).
\]
It remains to understand the kernel of the map~\eqref{eq: exactSeq2}.  By tracing the definitions we have
\[
H^{1}_{dR}(X,\mathbb{R}) \rightarrow H^{0,1}_{\dbar}(X) \rightarrow K_0 \rightarrow 0
\]
where the map $H^{1}_{dR}(X,\mathbb{R}) \rightarrow H^{0,1}_{\dbar}(X)$ sends a $d$-closed $1$-form to its $(0,1)$ part. Thus, it suffices to understand the image of $H^{1}_{dR}(X,\mathbb{R}) \rightarrow H^{0,1}_{\dbar}(X)$. In the global case, when $C=\mathbb{C}$, we have $H^{1}_{dR}(X,\mathbb{R})=0$, and so we deduce 

\begin{lem}\label{lem: BCdescrip}
Let $\pi:Y\rightarrow \mathbb{P}^1$ be a rational elliptic surface with an $I_k$ singular fiber,  $D= \pi^{-1}(\infty)\subset Y$.    Then, writing $\mathbb{C}= \mathbb{P}^1-\{\infty\}$, we have the exact sequence on $X= Y\setminus D$,
\[
0 \rightarrow H^{0}(\mathbb{C},R^{1}\pi_{*}\mathcal{O}_{X}) \cong H^{0,1}_{\dbar}(X)  \rightarrow H^{1,1}_{BC}(X,\mathbb{R}) \rightarrow H^2_{dR}(X,\mathbb{R}) \rightarrow 0.
\]
In particular, since $H^{0}(\mathbb{C},R^{1}\pi_{*}\mathcal{O}_{X}) \cong H^{0}(\mathbb{C}, \mathcal{O}_{\mathbb{C}})$,  we conclude that $H^{1,1}_{BC}(X,\mathbb{R})$ is infinite dimensional.
\end{lem}

When $C=\Delta^*$, $X_{\Delta^*}$ retracts onto the torus bundle $\pi^{-1}(\{|z|=1\})$, and hence $H^{1}(X_{\Delta}^*, \mathbb{R})$ is $2$-dimensional with topological monodromy conjugate to
\[
\begin{pmatrix}
1&k\\0&1
\end{pmatrix}.
\]
In order to describe the degeneracy it is easiest to work in coordinates on $X_{mod}$.  Fix a section $\sigma: \Delta^*\rightarrow X_{\Delta^*}$, and identify $X_{\Delta^*}$ with $X_{mod}$, and let $(x,z)$ be the usual coordinates, and write $x=x_1+\sqrt{-1}x_2$ and $z=re^{\sqrt{-1}\theta}$.  The cohomology group $H^{1}(X_{\Delta^*},\mathbb{R}) $ is generated by the $1$-forms 
\[
d\theta, \quad  W dx_2 + x_2 dW, \quad \text{ where }\quad  W:= \frac{1}{\frac{k}{2\pi}|\log|z||}.
\]
Now $d\theta^{0,1} = \frac{\sqrt{-1}}{2}\frac{d\bar{z}}{\bar{z}} = \sqrt{-1}\,\dbar \log|z|$, so the image of $d\theta$ in $H^{0,1}_{\dbar}(X_{\Delta^*})$ is zero.  On the other hand, we have
\[
(W dx_2 + x_2 dW)^{0,1} = \frac{-\sqrt{-1}}{2}W\left(d\bar{x} +\frac{{\rm Im}(x)}{\sqrt{-1}\bar{z}|\log|z||}d\bar{z}\right) 
\]
and so $H^{1}_{dR}(X_{\Delta^*},\mathbb{R})$ maps to the real line in $H^{0,1}_{\dbar}(X_{\Delta^*})$ spanned by $\sqrt{-1}Wd\bar{x}$.  We now need to understand the image in $R^{1}\pi_{*}\mathcal{O}_{X}$.  Suppose $\zeta$ is a $\dbar$-closed $(0,1)$ form.  We can write
\[
\zeta = \sqrt{-1}f(x,z) \left(W\left(d\bar{x} +\frac{{\rm Im}(x)}{\sqrt{-1}\bar{z}|\log|z||}d\bar{z}\right) \right) + g(x,z)d\bar{z}
\]
For each $z$ we have
\[
[\zeta|_{\pi^{-1}(z)}] = [f_0(z) \sqrt{-1}Wd\bar{x}] \in H^{0,1}(\pi^{-1}(z))
\]
where $f_0(z)$ is the constant term in the fiberwise Fourier series of $f$.  From $\dbar\zeta =0$ one can check that $f_0(z):\Delta^*\rightarrow \mathbb{C}$ is holomorphic.  The section $[f_0(z)\sqrt{-1}Wd\bar{x}]$ is precisely the section of $R^{1}\pi_{*}\mathcal{O}_{X}$ induced by $\zeta$.  With this description it is easy to see that $H^{1}_{dR}(X_{\Delta^*},\mathbb{R})$ generates the constant sections $[\sqrt{-1}cWd\bar{x}] \in H^{0}(\Delta^*,R^{1}\pi_{*}\mathcal{O}_{X})$ for $c\in \mathbb{R}$.  

\begin{lem}\label{lem: LocalBCdescrip}
Let $\pi:X_{\Delta}\rightarrow \Delta$ be an elliptic fibration, smooth outside of $0\in \Delta$, with $\pi^{-1}(0)$ an $I_k$ singular fiber.  Then, relative to the above choice of trivialization for $R^1\pi_{*}\mathcal{O}_{X_{\Delta}}$, we have the following exact sequence
\[
0\rightarrow \mathbb{R} \rightarrow H^{0}(\Delta^*,R^1\pi_{*}\mathcal{O}_{X_{\Delta}})\rightarrow H^{1,1}_{BC}(X_{\Delta^*},\mathbb{R}) \rightarrow H^{2}_{dR}(X_{\Delta}^*,\mathbb{R}) \rightarrow 0.
\]

\end{lem}

Since $X$ is non-compact, it is not immediately clear that the Bott-Chern cohomology classes of K\"ahler metrics form an open subset of $H^{1,1}_{BC}(X,\mathbb{R})$.  Therefore, while we have a detailed understanding of the Bott-Chern cohomology of $X$, it requires further work to understand the cone of Bott-Chern classes represented by K\"ahler metrics.  We will delay this discussion until after we discuss  translation maps.

\subsubsection{Translation maps}

In this section we will describe two aspects of translation maps.  First, we will explain how to construct {\em global} translation maps.  Secondly, we will explain how these translation maps act on Bott-Chern cohomology and how they effect the geometry of semi-flat metrics.

\begin{defn}
Let $\pi:X\rightarrow C$ be an elliptic fibration over a curve $C$.  We define ${\rm Aut}_0(X,C) $ to be the set of fiber preserving holomorphic automorphisms of $X$ homotopic to the identity.
\end{defn}

We have the following lemma

\begin{lem}\label{lem: globTrans}
$\pi:X\rightarrow C$ be a relatively minimal elliptic fibration over a curve $C$ without multiple fibers, and suppose there is a holomorphic section $\sigma: C\rightarrow X$.  There is an inclusion
\[
 H^{0}(C, R^1\pi_{X}\mathcal{O}_{X}) \hookrightarrow {\rm Aut}_{0}(X, C).
\]
\end{lem}
\begin{proof}
The lemma is a  consequence of the construction of the relative Jacobian fibration.  Recall \cite[Section V.9]{BarthPeters} that the relative Jacobian fibration is the fibration
\[
\pi_{Jac}: Jac = R^{1}\pi_{*}\mathcal{O}_{X}/ R^{1}\pi_{*}\mathbb{Z} \rightarrow C.
\]
Furthermore, given the section $\sigma$, there is a natural fiber preserving inclusion $J\hookrightarrow X$ identifying the zero section in $J$ with $\sigma$ \cite[Proposition V.9.1]{BarthPeters}.  Now an element of $s \in H^{0}(C, R^1\pi_{*}\mathcal{O}_{X})$ naturally induces a fiber preserving automorphism of $Jac$ by translation in each fiber.  It is not hard to show that the map which corresponds to translation by $s$ extends to a holomorphic map $T_s:X\rightarrow X$.  This map is clearly fiber preserving and homotopic to the identity.

We claim that the map $T_s:X\rightarrow X$ is independent of the choice of section $\sigma:C\rightarrow X$.  To see this note that if $\sigma':C\rightarrow X$ is another section then for $c\in C$ such that $\pi_{J}^{-1}(c)$ is smooth, the corresponding translation map $T'_{s}: \pi_{J}^{-1}(c) \rightarrow \pi_{J}^{-1}(c)$ can be written as $T_{s}' := T_{-f} \circ T_{s} \circ T_{f}$ for some $f\in \pi_{J}^{-1}(c)$.  But since the group of translations is abelian, $T_{s}'=T_{s}$.
\end{proof}

\begin{rk}\label{rk: locTransDes}
Fix a section $ \sigma:\Delta^*\rightarrow X_{\Delta^*}$.  If $\tau = [\sqrt{-1}h(z)Wd\bar{x}] \in H^{0}(\Delta^*, R^{1}\pi_{*}\mathcal{O}_{X})$ for some holomorphic function $h(z): \Delta^*\rightarrow \mathbb{C}$, then the local translation map $T_{h}(x,z) = (x+h(z), z)$ defined on $X_{\Delta^*}$ is a coordinate description of the translation map constructed in a coordinate invariant fashion in Lemma~\ref{lem: globTrans}.  This can be checked explicitly, following, for example \cite[Section V.9]{BarthPeters}.
\end{rk}

Let us next turn our attention to the effect of translation on the geometry of the semi-flat metric. Fix a reference section $\sigma: \Delta^*\rightarrow X_{\Delta^*}$, and some other section $\eta:\Delta^*\rightarrow X_{\Delta^*}$. We will first compute $T_{\eta}^*\omega_{sf, \sigma, b_0, \epsilon}$, with the aim of understanding how the asymptotics depend on $\eta$.  In order to simplify the notation, let us suppress the dependence of the metrics on $\sigma$ and denote $\omega_{sf, \sigma, b_0, \epsilon} = \omega_{sf, b_0, \epsilon}$ with the understanding that $\sigma$ is fixed.

By Lemma~\ref{lem: localSecDes} a section of $\pi_{mod}:X_{mod}\rightarrow \Delta^*$ corresponds to a multivalued holomorphic function
\[
\eta(z) = h(z) + \frac{a}{2\pi \sqrt{-1}}\log z + \frac{b}{(2\pi\sqrt{-1})^2} (\log(z))^2
\]
such that $a+b \in \mathbb{Z}$, $\frac{2b}{k}\in \mathbb{Z}$, and $h(z)$ is a holomorphic function on $\Delta^*$.  In order to understand the translation $T_{\eta}^*\omega_{sf,b_0,\epsilon}$ it is useful to introduce the following frame  of $(1,0)$ forms considered in \cite{GW}.
\[
\begin{aligned}
\Theta_{v} &:= W (dx+B(x,z)dz)& \qquad  \Theta_{h} &:= dz\\
\del_{v} &:= W^{-1}\del_x& \qquad \del_{h}&:= \del_z-B\del_x,
\end{aligned}
\]
where
\[
W= \frac{1}{\frac{k}{2\pi}|\log|z||} \qquad B(x,z) = -\frac{{\rm Im}(x)}{\sqrt{-1}z|\log|z||}.
\]
\begin{rk}
Note that $\Theta_{v}$ is a globally defined $(1,0)$ form on $X_{\Delta^*}$ since it is invariant under the translations $x\mapsto x+1$ and $x\mapsto x+\frac{k}{2\pi\sqrt{-1}}\log(z)$.  Furthermore, one can check that $\del \Theta_{v}=0$.
\end{rk}
In this frame, a (possibly non-standard) semi-flat metric $\omega_{sf,b_0,\epsilon}$ can be written as
\[
\omega_{sf,b_0,\epsilon} = \frac{\sqrt{-1}}{W}\left(|\kappa(z)|^2\frac{\Theta_{h}\wedge \overline{\Theta}_{h}}{\epsilon |z|^2} + \frac{\epsilon}{2}\left(\Theta_{v}+\frac{b_0}{2kz}\Theta_{h}\right)\wedge \left(\overline{\Theta_{v}+\frac{b_0}{2kz}\Theta_{h}}\right)\right).
\]
Consider the map $T_{\eta}(x,z) = (x+\eta(z),z)$.  Note that
\[
T_{\eta}^*(Bdz) = \frac{-{\rm Im}(x+\eta(z))}{\sqrt{-1}z(-\log|z|)}dz =  Bdz - \frac{{\rm Im}(\eta(z))}{\sqrt{-1}z(-\log|z|)}\Theta_{h},
\]
and so
\[
\begin{aligned}
T_{\eta}^*\Theta_{v} &= \Theta_{v} + W\left(\eta'(z) -  \frac{{\rm Im}(\eta(z))}{\sqrt{-1}z(-\log|z|)}\right)\Theta_{h}.\\
\end{aligned}
\]
Define 
\[
\widetilde{B}(\eta,z):= W\left(\eta'(z) -  \frac{{\rm Im}(\eta(z))}{\sqrt{-1}z(-\log|z|)}\right).
\]
We have

\begin{equation}\label{eq: transBySectionSF}
\begin{aligned}
 T_{\eta}^*\omega_{sf,b_0,\epsilon}-\omega_{sf,b_0,\epsilon} &= W^{-1}\frac{\sqrt{-1}\epsilon}{2}\left(  \widetilde{B}\Theta_{h}\wedge\left(\overline{\Theta_{v}+\frac{b_0}{2kz}\Theta_{h}}\right)\right)\\
 &+ W^{-1}\frac{\sqrt{-1}\epsilon}{2}\left( \overline{\widetilde{B}}\left(\Theta_{v}+\frac{b_0}{2kz}\Theta_{h}\right)\wedge \overline{\Theta}_{h} + |\widetilde{B}|^2\Theta_{h}\wedge \overline{\Theta}_{h}\right).
 \end{aligned}
\end{equation}

\begin{lem}\label{lem: transAsymp}
Let $\eta$ be a local section of $\pi: X_{\Delta^*}\rightarrow \Delta^*$, and write
\[
\eta(z) = h(z) + \frac{a}{2\pi \sqrt{-1}}\log z + \frac{b}{(2\pi\sqrt{-1})^2} (\log(z))^2
\]
where $a+b \in \mathbb{Z}$, $\frac{2b}{k}\in \mathbb{Z}$ and $h(z)$ is a holomorphic function on $\Delta^*$.  Then, with the notation above, we have
\begin{enumerate}
\item[$(i)$] If $h(z)$ has a pole at $0$, then $\omega_{sf,b_0,\epsilon}$ and  $T_{\eta}^*\omega_{sf, b_0,\epsilon}$ define K\"ahler metrics which are not uniformly equivalent.
\item[$(ii)$] If $h(z)$ is holomorphic at $0$, but $b\ne 0$, then $\omega_{sf,b_0,\epsilon}, T_{\eta}^*\omega_{sf,b_0,\epsilon}$ are uniformly equivalent, but there is a constant $C>1$ such that
\[
C^{-1}|b| \leq \big|\omega_{sf,b_0,\epsilon}-T_{\eta}^*\omega_{sf,b_0,\epsilon}\big|_{\omega_{sf,b_0,\epsilon}} \leq C|b|
\]
\item[$(iii)$] If $h(z)$ is holomorphic at $0$, and $b=0$, then in addition to being uniformly equivalent, we have the decay
\[
\big|(\omega_{sf,b_0,\epsilon} - T_{\eta}^*\omega_{sf,b_0, \epsilon})\big|_{\omega_{sf,b_0,\epsilon}} \sim O(r^{-\frac{4}{3}})
\]
where $r$ is the distance from a fixed point with respect to $\omega_{sf, b_0, \epsilon}$.
\item[$(iv)$] If $h(z)$ is holomorphic at $0$, and $b=0$, and $h(0)\in \mathbb{R}$ then we have the improved decay
\[
|\omega_{sf,b_0,\epsilon}-T_{\eta}^*\omega_{sf,b_0,\epsilon}|_{\omega_{sf,b_0,\epsilon}} \leq Ce^{-r^{2/3}}
\]
\end{enumerate}
\end{lem}
\begin{proof}
To ease notation, set $\omega_{sf}=\omega_{sf,b_0,\epsilon}$.  In the frame $\{\Theta_{v}, \Theta_{h}\}$ the Riemannian metric $g$ associated to $\omega_{sf,b_0,\epsilon}$ is given by
\[
g = W^{-1}\left(|\kappa(z)|^2\frac{|\Theta_{h}|^2}{\epsilon |z|^2} + \frac{\epsilon}{2}\big|\Theta_{v}+\frac{b_0}{2kz}\Theta_{h}\big|^2\right)
\]
from which it follows that
\begin{equation}\label{eq: unifEquiTrans}
\begin{aligned}
|\omega_{sf}-T_{\eta}^*\omega_{sf}|^2_{g} &\sim\left( \frac{2\pi}{k}\right)^2\frac{|z|^2}{(-\log|z|)^2}\bigg|\eta'(z) -  \frac{{\rm Im}(\eta(z))}{\sqrt{-1}z(-\log|z|)}\bigg|^2 \\
&+ \left( \frac{2\pi}{k}\right)^4\frac{|z|^4}{(-\log|z|)^4}\bigg|\eta'(z) -  \frac{{\rm Im}(\eta(z))}{\sqrt{-1}z(-\log|z|)}\bigg|^4
\end{aligned}
\end{equation}
If the Laurent series of $h$ contains a term like $z^{-M}$ for some $M\geq 1$ then
\[
\bigg|\eta'(z) -  \frac{{\rm Im}(\eta(z))}{\sqrt{-1}z(-\log|z|)}\bigg|^2 \geq C^{-1} |z|^{-2(M+1)},
\]
for some constant $C$ and hence the right hand side of~\eqref{eq: unifEquiTrans} behaves at least like $(|z|^M(-\log|z|))^{-4}$ which is unbounded as $z\rightarrow 0$, establishing $(i)$.  We may therefore assume that $h$ is holomorphic at $0$.  Now a direct calculation shows that
\[
\eta'(z) -  \frac{{\rm Im}(\eta(z))}{\sqrt{-1}z(-\log|z|)} = h'(z) -  \frac{{\rm Im}(h(z))}{\sqrt{-1}z(-\log|z|)} + \frac{2b}{(2\pi\sqrt{-1})^2}\frac{\log|z|}{z}.
\]
Note in particular that the $\log(z)$ term of $\eta$ does not contribute.  Thus, if $b\ne 0$, then we have
\[
\bigg|\eta'(z) -  \frac{{\rm Im}(\eta(z))}{\sqrt{-1}z(-\log|z|)}\bigg|^2 \sim b^2 \frac{(\log|z|)^2}{|z|^2},
\]
which proves $(ii)$.  Now assume $b=0$, and $h$ is holomorphic.  If ${\rm Im}(h(0))\ne 0$ then we have
\[
\bigg|\eta'(z) -  \frac{{\rm Im}(\eta(z))}{\sqrt{-1}z(-\log|z|)}\bigg|^2 \sim \frac{|{\rm Im}(h(0))|^2}{|z|^2(-\log|z|)^2},
\]
and hence we get the estimate
\[
|\omega_{sf}-T_{\eta}^*\omega_{sf}|^2_{g}\sim \frac{|{\rm Im}(h(0))|^2}{(-\log|z|)^4}.
\]
But, as we noted before, if $r$ denotes the distance from a fixed point with respect to $\omega_{sf}$, then we have $r(x,z) \sim (-\log|z|)^{\frac{3}{2}}$, yielding $(iii)$.  Finally, if $h(0)\in \mathbb{R}$ we have
\[
\bigg|\eta'(z) -  \frac{{\rm Im}(\eta(z))}{\sqrt{-1}z(-\log|z|)}\bigg|^2 \leq C
\]
from which we get the bound
\[
|\omega_{sf}-T_{\eta}^*\omega_{sf}|^2_{g} \leq \frac{C|z|^2}{(-\log|z|)^2} \leq Ce^{-2r^{2/3}}
\]
\end{proof}

\begin{rk}\label{rk: transIsom}
Note that the computation above also shows that $T_\eta$ is an isometry of the semi-flat metric if and only if $\eta = c + \frac{a}{2\pi\sqrt{-1}}\log(z)$ for some real constants $c,a$. This is a well-known, but important fact \cite{Hein, GW}.
\end{rk}

The next lemma explains how translating by a section acts on the Bott-Chern cohomology of $X_{\Delta^*}$ in the case of differences of K\"ahler metrics.  It is not difficult to extend this discussion to general real $(1,1)$ forms following \cite{GW, Hein}, but since we won't need this we will not pursue it.  We first analyze the effect of translation on the semi-flat metrics.

\begin{lem}\label{lem: transSemiFlat}
Let $\sigma:\Delta^*\rightarrow X_{\Delta^*}$ be a holomorphic section, and fix $\epsilon >0$.  Let $h:\Delta^* \rightarrow \mathbb{C}$ be a holomorphic function.  Then, with the identification in Lemma~\ref{lem: LocalBCdescrip}, we have
\[
[T_{h}^*\omega_{sf,\sigma, b_0, \epsilon} - \omega_{sf, \sigma, b_0, \epsilon}]_{BC} \longleftrightarrow \frac{\epsilon h}{2} \in H^{0}(\Delta^*, R^1\pi_{*}\mathcal{O}_{X_{\Delta^*}})/\mathbb{R}
\]
\end{lem}
\begin{proof}
This follows essentially from the proof of \cite[Claim 1]{Hein}.  Working in coordinates $(x,z)$ on $X_{mod}$, via the identification induced by $\sigma$, we consider $T_{th}(x,z) = (x+th(z), z)$ for $t\in [0,1]$.  Then by Cartan's magic formula we have $T_{h}^*\omega_{sf, b_0, \epsilon} - \omega_{sf, b_0, \epsilon} = d\zeta$ where
\[
\zeta = \int_{0}^{1} 2{\rm Re}\left(h \frac{\del}{\del x}\right)\lrcorner T_{th}^*\omega_{sf, b_0, \epsilon}
\]
Now a direct calculation shows that the $(0,1)$ part of $\zeta$ is
\[
\zeta^{0,1} =\frac{ \sqrt{-1}\epsilon W}{2}h \left(d\bar{x} + \overline{B(x,z)} d\bar{z} + \frac{1}{2}\overline{B}(h(z),z)d\bar{z} + \frac{b_0}{2\pi^2} \frac{|\log|z||}{\overline{z}}d\bar{z}\right)
\]
Thus, the lemma follows from the discussion leading to Lemma~\ref{lem: LocalBCdescrip}.
\end{proof}

We note the following corollary

\begin{cor}\label{cor: locBCKahler}
Let $[\omega]_{dR} \in H^{2}_{dR}(X_{\Delta^*},\mathbb{R})$ be any K\"ahler class.  Then there exists a (possibly non-standard) semi-flat metric $\omega_{sf,\sigma, b_0, \epsilon}$ and a holomorphic function $h: \Delta^*\rightarrow \mathbb{C}$ such that
\[
[\omega]_{BC}=[T_{h}^*\omega_{sf,\sigma, b_0, \epsilon}]_{BC} \in H^{1,1}_{BC}(X_{\Delta^*}, \mathbb{R}).
\]
 where $T_{h}$ denotes translation with respect to $\sigma$.  Furthermore, $T_{h}^*\omega_{sf, \sigma, b_0, \epsilon}$ is the unique semi-flat metric in $[\omega]_{BC} \in H^{1,1}_{BC}(X_{\Delta^*}, \mathbb{R})$.
\end{cor}

\begin{proof}

Fix a section $\sigma$ and identify $X_{\Delta^*}$ with $X_{mod}$.  Let $\omega_{sf,\epsilon}$ denote the standard semi-flat metric on $X_{mod}$.  By Lemma~\ref{lem: qrsfTrans} and Remark~\ref{rk: nonStTansbyMult},  it suffices to show that there is a multivalued section
\[
\eta= h(z) +\frac{a}{2\pi \sqrt{-1}}\log(z) + \frac{b_0}{(2\pi\sqrt{-1})^2}(\log(z))^2
\]
where $h: \Delta^*\rightarrow \mathbb{C}$ is holomorphic, such that $[\omega]_{BC} = [T_\eta^*\omega_{sf,\epsilon}]$.  By Lemma~\ref{lem: cohomProps} and Corollary~\ref{cor: local deRham}, the de Rham cohomology class $[\omega]_{dR} \in H^{2}_{dR}(X_{\Delta^*}, \mathbb{R})$ uniquely determines $\epsilon, b_0$.  By Remark~\ref{rk: transIsom} we can assume that $a=0$, since translation by $\frac{a}{2\pi\sqrt{-1}}\log(z)$ is an isometry of the semi-flat metric.  Finally, if $\tilde{\eta} = \frac{b_0}{(2\pi\sqrt{-1})^2}(\log(z))^2$, then by Lemma~\ref{lem: LocalBCdescrip}, $h$ is determined up to addition of a real constant, by $[T_{h}^*T_{\tilde{\eta}}^*\omega_{sf, \epsilon}]_{BC}=[\omega]_{BC}$. To see this, fix $[\sqrt{-1}Wd\bar{x}]$ identifying $R^1\pi_{*}\mathcal{O}_{X}\cong \mathcal{O}_{\Delta^*}$.   Suppose that ${[\omega-T_{\tilde{\eta}}^*\omega_{sf, \epsilon}]_{BC}}$ induces a holomorphic function ${s_0(z):\Delta^*\rightarrow \mathbb{C}}$.  By Lemma~\ref{lem: transSemiFlat}, the difference $[\omega - T_{h}^*T_{\tilde{\eta}}^*\omega_{sf, \epsilon}]_{BC}$ induces ${s_0(z)-\sqrt{-1}\frac{\epsilon h}{2}}$, and so if we take ${h= -\frac{2\sqrt{-1}}{\epsilon}s_0}$ then ${[\omega_1 - T_{h}^*\omega_{sf, \sigma, \epsilon}]_{BC}=0}$. By Remark~\ref{rk: transIsom}, translation by a real constant is an isometry of the semi-flat metric and hence the semi-flat metric in the Bott-Chern class $[\omega]_{BC}$ is uniquely determined.
\end{proof}

We can now prove

\begin{prop}\label{prop: BCKahlerCone}
Let $\pi: Y\rightarrow \mathbb{P}^1$ be a rational elliptic surface, and $\pi^{-1}(\infty)=D=\sum_{i=1}^{k} D_i$ be a singular fiber of type $I_k$, with irreducible components $D_i$.  Let $X=Y\setminus D$ and identify $\mathbb{C} = \mathbb{P}^1\setminus\{\infty\}$.  Define the de Rham and Bott-Chern K\"ahler cones by
\[
\begin{aligned}
\mathcal{K}_{dR,X} &= \{ [\omega]_{dR} \in H^2(X,\mathbb{R}) : \omega \text{ is K\"ahler } \},\\
\mathcal{K}_{BC,X} &= \{ [\omega]_{BC} \in H^{1,1}_{BC}(X,\mathbb{R}) : \omega \text{ is K\"ahler } \}.
\end{aligned}
\]
Then
\begin{itemize}
\item[$(i)$] $\mathcal{K}_{dR,X} $ is a convex cone in $H^2_{dR}(X, \mathbb{R})$ with non-empty interior.
\item[$(ii)$] Consider the exact sequence from Lemma~\ref{lem: BCdescrip}:
 \[
 0\rightarrow H^{0}(\mathbb{C}, R^1\pi_{*}\mathcal{O}_{X} ) \rightarrow H^{1,1}_{BC}(X,\mathbb{R})\overset{p}{\longrightarrow} H^{2}_{dR}(X,\mathbb{R}) \rightarrow 0.
 \]
 Then  $\mathcal{K}_{BC,X} = p^{-1}(\mathcal{K}_{dR,X})$. In other words, we have
\[
\mathcal{K}_{BC,X} = \mathcal{K}_{dR,X} \times H^{0}(\mathbb{C},R^1\pi_{*}\mathcal{O}_{X}).
\]
\item[$(iii)$] For any $\tau \in H^{0}(\mathbb{C}, R^1\pi_{*}\mathcal{O}_{X})$, let $\Phi_{\tau} \in {\rm Aut}_0(X, \mathbb{C})$ be the automorphism induced by Lemma~\ref{lem: globTrans}.  Then, for any K\"ahler class $[\omega]_{dR}\in H^{2}_{dR}(X,\mathbb{R})$ the map $\Phi_{\tau}^*: p^{-1}([\omega]_{dR}) \longrightarrow p^{-1}([\omega]_{dR})$ is determined by 
\[
[\Phi_{\tau}^*\omega-\omega]_{BC} \longleftrightarrow \frac{\epsilon}{2} \tau \in H^{0}(\mathbb{C},R^1\pi_{*}\mathcal{O}_{X}).
\]
where $\epsilon = [\omega]_{dR}.[F]$ with $[F]$ the class of a fiber.

\end{itemize}
\end{prop}
\begin{proof}
We have already proved $(i)$ in Lemma~\ref{lem: deRham} and Corollary~\ref{cor: KahlerCone}. To prove $(ii)$, let $\omega_1, \omega_2$ on $X$ with $[\omega_1]_{dR}=[\omega_2]_{dR}$ we have seen that $[\omega_1-\omega_2]_{BC}$ can be identified with an element of $H^{0}(\mathbb{C},R^1\pi_*\mathcal{O}_{X})$.  This proves that
\[
\mathcal{K}_{BC, X} \subset \mathcal{K}_{dR,X} \times H^{0}(\mathbb{C},R^1\pi_*\mathcal{O}_{X})
\]
It suffices to show that we have equality.  This will be accomplished by proving $(iii)$.  Fix $\omega_0$ a K\"ahler form on $X$. Since $Y$ is a rational elliptic surface it admits a global section $\sigma:\mathbb{P}^1\rightarrow Y$.  Let $\tau \in H^{0}(\mathbb{C}, R^1\pi_{*}\mathcal{O}_{X})$, and denote by $\Phi_{\tau} \in {\rm Aut}_0(X, \mathbb{C})$ the automorphism induced by Lemma~\ref{lem: globTrans}.  Since $\Phi_{\tau}$ is homotopic to the identity, $[\Phi_{\tau}^*\omega_0]_{dR}=[\omega_0]_{dR}$. We will compute the section of $R^{1}\pi_{*}\mathcal{O}_{X}$ associated to the class  $[\Phi_{\tau}^*\omega_0-\omega_0]_{BC}$.

Let $\Delta^*$ denote a punctured neighborhood of $\infty \in \mathbb{P}^1$.  Over $\Delta^*$  we can write
\[
\tau = [\sqrt{-1}h(z)W d\bar{x}]
\]
where $h|_{\Delta^*}$ is a holomorphic function.  By Corollary~\ref{cor: locBCKahler} we can find a semi-flat metric $\omega_{sf, \sigma',  b_0, \epsilon}$ such that $[\omega_0]_{BC} = [\omega_{sf, \sigma', b_0, \epsilon}]_{BC}$. Thus we have
\[
[T_{h}^*\omega_0-\omega_0]_{BC} = [T_{h}^*\omega_{sf, \sigma', b_0,\epsilon}-\omega_{sf, \sigma', b_0, \epsilon}]_{BC}
\]
in $H^{1,1}_{BC}(X_{\Delta^*}, \mathbb{R})$.  Let $\epsilon= [\omega_0]_{dR}.[F]$, then by Lemma~\ref{lem: transSemiFlat}, $[T_{h}^*\omega_0-\omega_0]_{BC} \in H^{1,1}_{BC}(X_{\Delta^*,} \mathbb{R})$ induces the section
\[
[\frac{\sqrt{-1} \epsilon}{2}h(z) Wd\bar{z}] = \frac{\epsilon}{2}\tau\big|_{\Delta^*} \in H^{0}(\Delta^*,R^1\pi_{*}\mathcal{O}_{X}).
\]
By Remark~\ref{rk: locTransDes} we conclude that $[\Phi_{\tau}^*\omega_0-\omega_0]_{BC}$ induces the global section $\frac{\epsilon}{2}\tau \in H^{0}(\mathbb{C}, R^1\pi_{*}\mathcal{O}_{X})$.
 
\end{proof}

We end by making the following definition.

\begin{defn}\label{defn: rational}
We say that the de Rham cohomology class of a K\"ahler metric $[\omega]_{dR}\in H_{dR}^2(X,\mathbb{R})$ is {\em rational near infinity} if 
\[
[\omega]_{dR}.[C_{m_1,m_2}] =0
\]
for some quasi-bad cycle $C_{m_1,m_2}$.
\end{defn}

Fix a choice of bad cycle $[C]\in H_{2}(X,\mathbb{Z})$ and let $[F]$ denote the class of a fiber.  For $m_1, m_2 \in \mathbb{Z}$ relatively prime define
\begin{equation}\label{eq: rationalPlanes}
V_{m_1,m_2} = \{ [\alpha]_{dR} \in H^2(X,\mathbb{R}) : [\alpha].(m_1[F]+m_2[C])=0\}
\end{equation}
Then the K\"ahler classes which are rational at infinity are precisely those which lie in $V_{m_1,m_2}$ for some $(m_1,m_2)$.  Note that, by Lemma~\ref{lem: cohomProps}, any K\"ahler class which is not rational at infinity is topologically obstructed from admitting a special Lagrangian torus fibration.

\section{Special Lagrangian Fibrations}\label{sec: slag}

Let $Y$ be a rational elliptic surface, $D$ an $I_k$ singular fiber and $X=Y\setminus D$.  By Theorem~\ref{thm: nonStHein}, we may equip $X$ with a complete Calabi-Yau metric $\omega_{CY}$ asymptotic to a (possibly non-standard) semi-flat metric.  In this section we prove the existence of special Lagrangian torus fibrations on $(X, \omega_{CY})$ assuming that the de Rham class $[\omega_{CY}]_{dR}$ is rational near infinity in the sense of Definition~\ref{defn: rational}. The proof proceeds in three steps, following the techniques developed by the authors in \cite{Collins-Jacob-Lin}.   First, we study the existence of special Lagrangian tori in the model geometry $(X_{mod}, \omega_{sf})$, where $\omega_{sf}$ is a semi-flat metric which is either standard or quasi-regular.  Secondly, we transplant these model special Lagrangians into Lagrangians in $(X,\omega_{CY})$ with well-controlled geometry.  We then run the Lagrangian mean curvature flow and establish the convergence of the flow to a special Lagrangian limit.  Finally, we use the theory of holomorphic curves in combination with a hyperK\"ahler rotation trick to deform this family to a special Lagrangian fibration.  Since most of the technical ingredients for this argument were developed by the authors in \cite{Collins-Jacob-Lin}, this argument essentially reduces to understanding special Lagrangians in the model geometry.

\subsection{Ansatz special Lagrangian tori and estimates of geometric quantities} \label{Ansatz SLag} \label{9999}
Let $\Delta = \{ z\in \mathbb{C}: |z|<1\}$, and set $\Delta^* = \Delta \setminus\{0\}$. We begin by first working on the model fibration $\pi_{mod}:X_{mod}\rightarrow \Delta^*$, defined by   
 \be
 \label{modelfibration}
X_{mod} := (\Delta^*\times\mathbb{C})/ \Lambda(z) \qquad \Lambda (z) := \mathbb{Z} \oplus \mathbb{Z}\frac{ k}{2\pi \sqrt{-1}}\log(z).
 \ee
Let $x$ be the standard coordinate on $\mathbb{C}$, and consider logarithmic coordinates  $y=-{\rm log}(z)$ on the universal cover $\mathcal H_{>0}$ of $\Delta^*$. The holomorphic volume form can be expressed in these coordinates as $\Omega={\kappa(e^{-y})}dy\wedge dx$ for a non-vanishing holomorphic function $\kappa: \Delta \rightarrow \mathbb{C}$, which we may assume satisfies $\kappa(0)=1$. 
 
 As a first step we will discuss the case when $\kappa\equiv 1$.  In this case, using Definition \ref{defn: sf metric}, the standard model semi-flat metric is given by
\be
\omega_{sf, \epsilon} :=  \sqrt{-1}  W^{-1}\frac{dy\wedge d\bar{y} }{\epsilon}+ \frac{\sqrt{-1}}{2} W \epsilon (dx+Bdy)\wedge \overline{(dx+Bdy)}\nonumber
\ee
where
\be
W=\frac{4\pi }{k(y+\bar{y})} \qquad B = \frac{\bar{x}-x}{y+\bar{y}}.\nonumber
\ee
It is convenient to work in real coordinates. Set $y={\ell+\sqrt{-1}\theta}$ and $x=x_1+\sqrt{-1}x_2$. Then 
    \bea
\omega_{sf, \epsilon}&=&\left(\frac{k\ell }{\pi\epsilon}+|B|^2\frac{2\pi\epsilon}{k\ell }\right)(d\ell\wedge d\theta)+\frac{2\pi\epsilon}{k\ell }dx_1\wedge dx_2\nonumber\\
    &&+ {\sqrt{-1}}\frac{\pi\epsilon}{k\ell }\left(\bar Bdx\wedge d\bar y-Bd \bar x\wedge dy\right).\nonumber
    \eea
   Using that $B$ is purely imaginary, the second line above reduces to
    \be
    {\sqrt{-1}} \frac{ \pi\epsilon}{k\ell }(-B+\bar B)(dx_1\wedge d\ell+dx_2\wedge d\theta)=\frac{2\pi\epsilon x_2}{k\ell^2}(dx_1\wedge d\ell+dx_2\wedge d\theta). \nonumber
    \ee
Putting everything together
    \bea
    \omega_{sf, \epsilon}&=&\left(\frac{k\ell }{\pi\epsilon}+\frac{2\pi\epsilon x_2^2}{k\ell^3 }\right)(d\ell\wedge d\theta)\nonumber\\
    &&+\frac{2\pi\epsilon}{k\ell }dx_1\wedge dx_2+\frac{2\pi\epsilon x_2}{k\ell^2}(dx_1\wedge d\ell+dx_2\wedge d\theta).\nonumber
    \eea
    Using the complex structure $d\ell\mapsto d\theta$ and $dx_1\mapsto dx_2$,  the corresponding Riemannian metric is given by
    \bea
    \label{sfmetric}
    g_{sf,\epsilon}&=&\left(\frac{k\ell }{\pi\epsilon}+\frac{2\pi\epsilon x_2^2}{k\ell^3}\right)(d\ell^2+d\theta^2 )+\frac{2\pi\epsilon}{k\ell }(dx_1^2+dx_2^2)\\
    &&+\frac{2\pi\epsilon x_2}{k\ell^2}(-dvd\theta -d\ell ds).\nonumber
    \eea
Recall the standard bad cycle $C_0:=\{x_2=0, \ell ={\rm const}\}$. There is a two parameter family of such cycles, parametrized by $x_2, \ell$.   For simplicity, we consider the subfamily with $x_2=0$ and denote this family $C_0(r)$, where $r$ is the intrinsic distance of the cycle from a fixed point in the base $z_0$. See Remark~\ref{rk: locModFib}. We can assume $\ell >|z_0|$.

As mentioned above, $C_0(r)$ is an embedded torus, with tangent space spanned by $\{\frac{\partial}{\partial \theta},\frac{\partial}{\partial x_1}\}$. Right away we see that
    \be
g_{sf,\epsilon}|_{C_0(r)}=\frac{k\ell }{\pi\epsilon}d\theta^2+\frac{2\pi\epsilon}{k\ell }dx_1^2,\nonumber
    \ee
    and since $\ell $ is constant on the cycle the induced metric is flat and the volume form is simply ${\rm vol}_{C_0(r)}=2d\theta\wedge dx_1$. Thus the volume of $C_0(r)$ is independent of $r$. Also note that the cycles defined by $\theta=const$ collapse as $\ell \rightarrow \infty$, while the cycles defined by $v=const$ (and thus the diameter of $C_0(r)$), grow like $\ell ^\frac12$ as $\ell \rightarrow\infty$.

Note that $C_0(r)$ is  a Lagrangian submanifold. Furthermore, under our assumption $k\equiv 1$, the holomorphic volume form is given by $\Omega=dy\wedge dx$, and so
    \be
    \Omega|_{C_0(r)}=\sqrt{-1}d\theta\wedge dx_1=\frac{\sqrt{-1}}2{\rm vol}_{C_0(r)}.\nonumber
    \ee
Thus $C_0(r)$ has constant phase $\frac\pi2$, and is therefore a special Lagrangian. In particular $C_0(r)$ has vanishing mean curvature $H=0$.
    
    Next we turn to the second fundamental form $\Pi$ of $C_0(r)$, which we will estimate using Gauss' Theorem. Since $H=0$, we can bound $\Pi$ by the curvature of $g_{sf,\epsilon}$ and the induced metric on $C_0(r)$. Furthermore, the induced metric on $C_0(r)$ is flat, so we only need to estimate the curvature of $g_{sf,\epsilon}$, which by \cite[Section 3.3]{Hein} satisfies    
    
     \be
     \label{curvsf}
    |Rm_{g_{sf,\epsilon}}|_{C^0(g_{sf,\epsilon})}=O(r^{-2}) .
    \ee
 At a point $p\in C_0(r)$, let $\{E_1,E_2\}$ be an orthonormal basis of the normal space $(TC_0(r))^\perp\subset TM$, and $\{V,W\}$ and basis for $TC_0(r)$. The second fundamental form can be expressed as
    \be
    \Pi(V,W)=\langle S_1(V),W\rangle E_1+\langle S_2(V),W\rangle E_2,\nonumber
    \ee
    where the shape operators $S_i$ are defined by $\langle S_i(V),W\rangle =\langle E_i,\nabla_VW\rangle.$ The mean curvature vector is given by
    \be
    \vec{H}=\frac12\left({\rm Tr}(S_1)E_1+{\rm Tr}(S_2)E_2\right),\nonumber
    \ee
    and since $C_0(r)$ is a minimal surface, we have Tr$(S_i)=0$ for each $i$. Thus each $S_i$ is a trace-free symmetric matrix, and in the basis $\{V,W\}$ we denote
    \be
    S_1=\left( 
    \begin{array}{cc}
    	a & b \\
    	b & -a \end{array} 
    \right),\qquad\qquad S_2=\left( 
    \begin{array}{cc}
    	e & d \\
    	d & -e \end{array} 
    \right).\nonumber
    \ee
 Let $K$ denote the sectional curvature of $g_{sf,\epsilon}$.  Since $g_{sf,\epsilon}|_{C_0(r)}$ is flat,  Gauss' Theorem gives
    \be
    -K(V,W)=\langle \Pi(V,V),\Pi(W,W)\rangle-|\Pi(V,W)|^2.\nonumber
    \ee
    Using our notation for $S_i$ we see
    \bea
    -K(V,W)&=&\langle aE_1+eE_2,-aE_1-eE_2\rangle-\langle bE_1+dE_2,bE_1+dE_2\rangle\nonumber\\
    &=&-a^2-e^2-b^2-d^2.\nonumber
    \eea
    Since $|K(V,W)|_{C^0(g_{sf,\epsilon})}=O(r^{-2})$ by \eqref{curvsf}, the square of every matrix entry for each shape operator must be in $O(r^{-2})$, and so 
    \be
    |\Pi|_{C^0(g_{sf,\epsilon})}=O(r^{-1}).\nonumber
    \ee

We now turn to the non-collapsing scale for $C_0(r)$. We say $C_0(r)$ is $\alpha$ non-collapsed at scale $\delta_0$ if, for every $0<\delta<\delta_0$, and for every $p\in C_0(r)$, we have
\be
{\rm Vol}\left(B(p,\delta ) \subset C_0(r)\right) \geq \alpha \delta^{2},\nonumber
\ee
where all quantities on the left are defined using the induced metric $g_{sf,\epsilon}|_{C_0(r)}$.  Pick a point $p\in C_0(r)$. Since $C_0(r)$ is a flat torus, the local geometry near $p$ can be modeled as follows. Fix an open set $U$ containing $p$ with coordinates $(x_1,\theta)$, and consider the mapping $\Phi:U\rightarrow \mathbb R^2$ defined by $(x_1,\theta)\mapsto (\sqrt{\frac{2\pi\epsilon}{k\ell }} x_1, \sqrt{\frac{k\ell }{\pi\epsilon}}\theta)$. Then $g_{sf,\epsilon}|_{C_0(r)}$ can be viewed as the pullback of the Euclidean metric on $\mathbb R^2$ via $\Phi$. For any $0<\rho<1$, we show the above non-collapsing inequality holds for the geodesic ball $B:=B(p,\rho\sqrt{\frac{2\pi\epsilon}{k\ell }})$. Consider a geodesic square  $\mathcal S$,  with side length $ \rho\sqrt{\frac{4\pi\epsilon}{k\ell }}$ with respect to  $g_{sf,\epsilon}|_{C_0(r)}$,  inscribed inside $B$. If one instead measures the lengths of the sides with the Euclidean metric on $U$, the side of $\mathcal S$ with coordinate $x_1$ will have length  $\sqrt{2}\rho$, while the side with coordinate $\theta$ will have length $\rho\frac{2\pi\epsilon}{k\ell }$. Since $C_0(r)$ has volume form $2dv\wedge d\theta$, we see:
 \be
\int_Bd{\rm vol}_{C_0(r)}\geq \int_{\mathcal S} 2dv\wedge d\theta=  \sqrt{2}\rho^2\left(\frac{4\pi\epsilon}{k\ell }\right)=\sqrt{2}\left(\rho\sqrt{\frac{2\pi\epsilon}{k\ell }}\right)^2.\nonumber
\ee
Setting $\sqrt{2}=\alpha$, it follows that $C_0(r)$ is $\sqrt{2}$-non collapsed  at scale $\sqrt{\frac{2\pi\epsilon}{k\ell }}.$
 By the discussion in Section~\ref{sec: semiflatmet}, we know for $\ell$ large $r\sim \ell^{\frac32}$, and thus $C_0(r)$ is $\sqrt{2}$-non collapsed  at scale $O(r^{-\frac13})$.

  Finally, we need to understand the asymptotics of the first non-zero eigenvalue of the Laplacian on $C_0(r)$, which we denote by $\lambda_1$. Since $g_{sf,\epsilon}|_{C_0(r)}$ is flat we may appeal to work of Li-Yau (for instance \cite[p.116]{SY}) which gives 
   	      \begin{align*}
   	          \lambda_1\geq \frac{-1}{C({\rm diam})^2},
   	      \end{align*} 
	      where $C$ is a constant depending only on dimension. Since diam$({C_0(r)})=O(\ell^{1/2})$, the estimate of Li-Yau gives $\lambda_1\geq O(\ell^{-1})$. Furthermore, we can estimate this eigenvalue from above using the Rayleigh quotient formula (as in \cite{Collins-Jacob-Lin}), to achieve $\lambda_1\leq O(\ell^{-1})$. Thus $\lambda_1= O(r^{-\frac23})$, where we used $r\sim \ell^{\frac32}$.

Summing up, we have now demonstrated the following:
    \begin{lem} \label{model1}
    	Consider the model elliptic fibration $X_{mod}\rightarrow \Delta$ defined by \eqref{modelfibration}, equipped with the holomorphic volume form $\Omega = \frac{dz\wedge dx}{z}$, and the standard model semi flat metric $\omega_{sf, \epsilon}$ relative to $\Omega$. Let $C_0(r)$ be the standard bad-cycle, with $r$  the intrinsic distance of the cycle from a fixed point $z_0$. Then we have the following estimates for the relevant geometric quantities: 
    	\begin{enumerate}
    		\item The second fundamental form   satisfies $  |\Pi|_{C^0(\omega_{sf,\epsilon} )}= O(r^{-1})$.  
    		\item The mean curvature vanishes identically $H\equiv 0$. 
    		\item The volume  is independent of $r$. 
    		\item $C_0(r)$ is $\sqrt{2}$-non-collapsed  at scale $ O(r^{-1/3})$. 
		\item The first eigenvalue of the Laplacian satisfies $\lambda_1 =O(r^{-2/3})$.
    	\end{enumerate}
    \end{lem}

  We now consider the general case, when $\Omega = \kappa(z) \frac{dx\wedge dz}{z}$, for a holomorphic function $\kappa$ that satisfies $\kappa(0)=1$.  Let $\omega_\kappa$ denote the model semi-flat metric relative to the  holomorphic volume form $ \kappa(z) \frac{dx\wedge dz}{z}$, and let $\omega_1$ denote the semi-flat metric relative to $\frac{dx \wedge dz}{z}$.  Let $g_\kappa, g_1$ denote the corresponding Riemannian metrics.   Note that the family of bad cycles $C_0(r)$ will be  Lagrangian with respect to both $\omega_1$ and $\omega_{\kappa}$, however, they are no longer special Lagrangian when measured against $\kappa(z) \frac{dx\wedge dz}{z}$.  We show that the defect is exponentially suppressed as $r \rightarrow \infty$, and so we can easily control their geometry in terms of the estimates in Lemma~\ref{model1}.

 Consider the following difference, using coordinates $(y,x)$:
  \be
  \omega_\kappa-\omega_1 = {\sqrt{-1}}(|\kappa(e^{-y})|^2-1)\frac{k|y|}{2\pi\epsilon} {dy\wedge d\bar{y}}.\nonumber
  \ee
 As above it is convenient to work in real coordinates $y = \ell+\sqrt{-1}\theta$. Since $\kappa(0)=1$ we can expand
  \be
|\kappa(z)|^2-1 = \sum_{n>0} \frac{a_{n}}{n!}e^{-ny}\nonumber
\ee
and hence there is a uniform constant $C>0$ so that 
\be
C^{-1}\omega_{1} \leq \omega_{\kappa} \leq C \omega_1.\nonumber
\ee
Furthermore, an easy induction shows that
\be
\pl_{\ell}^m (\ell \kappa(e^{-y})) = m\pl_{\ell}^{m-1}\kappa(e^{-y}) + \ell\pl_{\ell}^m \kappa(e^{-y})\nonumber
\ee
and so, since $a_0=0$, as $\ell \rightarrow \infty$ we have 
\be
\begin{aligned}
\pl_{\ell}^m (\omega_\kappa - \omega_1) &= \left(m\frac{\pl_{\ell}^{m-1}\kappa(e^{-y})}{\ell} + \pl_{\ell}^m \kappa(e^{-y})\right) \frac{\ell}{2\pi \epsilon} dy \wedge d\bar{y}\\
& = O(e^{-\ell})  \frac{\ell}{2\pi \epsilon} dy \wedge d\bar{y}.
\end{aligned}\nonumber
\ee
Since $\omega_\kappa - \omega_{1}$ is independent of the remaining variables we get
\be
|\nabla^{m} (g_\kappa- g_1)| = O(e^{-r^{\frac23}})\nonumber
\ee
as $r\rightarrow +\infty$.  Here $\nabla^m$ denotes the $m$-th covariant derivative with respect to either $g_1$ or $g_\kappa$ and the norm is similarly measured with respect to either metric.  As a consequence, we obtain

    \begin{lem} \label{model}

	Consider the model elliptic fibration $X_{mod}\rightarrow \Delta$ defined by \eqref{modelfibration}, equipped with the holomorphic volume form $\Omega = \kappa(z)\frac{dz\wedge dx}{z}$, and the standard model semi flat metric $\omega_{sf, \epsilon}$ relative to $\Omega$. Let $C_0(r)$ be the standard bad-cycle, with $r$  the intrinsic distance of the cycle from a fixed point $z_0$. Then we have the following estimates for the relevant geometric quantities: 
		    	\begin{enumerate}
    		\item The second fundamental form of $C_0(r)$ satisfies $  |\Pi|_{C^0(\omega_{sf,\epsilon})}\sim O(r^{-1})$.  
    		\item The mean curvature vanishes satisfies $|H|_{C^0(\omega_{sf,\epsilon} )}\sim O(e^{-\delta r^{\frac23}})$.
    		\item The volume satisfies $C^{-1} < {\rm Vol}(C_0(r)) < C$. 
    		\item $C_0(r)$ is $C^{-1}$-non-collapsed at scale $\sim O(r^{-1/3})$. 
		\item The first eigenvalue of the Laplacian satisfies $\lambda_1=O(r^{-2/3})$.
    	\end{enumerate}
	Here $C, \delta>0$ are uniform constants. 
    \end{lem}

Additionally, one can consider quasi-bad cycles. Recall that a quasi-bad cycle is given by
\be
C_{m_1,m_2}(r):=\{x_2=\frac{m_1}{m_2}(-\frac k{2\pi}{\rm log}|z|)\frac\theta{2\pi},|z|={\rm const.}\}\nonumber
\ee
for $m_2, m_1$ relatively prime with $m_1>0$.  Let $\eta_{m_1, m_2}=\frac{m_2k}{2m_1(2\pi\sqrt{-1})^2}({\rm log}(z))^2$ be a local multivalued section of $X_{mod}$. The embedded torus $C_{m_1,m_2}(r)=T_{\eta_{m_1,m_2}}^{-1}(C_0(r))$, is Lagrangian with respect to the quasi-regular semi-flat metric $T_{\eta_{m_1,m_2}}^*\omega_{sf, \epsilon}$.  Furthermore, since $T_{\eta_{m_1,m_2}}$ defines a local isometry between the model $C_0(r)$ with the standard semi-flat metric and the quasi-bad cycle equipped with the metric induced by $T_{\eta_{m_1,m_2}}^*\omega_{sf, \epsilon}$, all local geometric estimates carry over verbatim.  However, we need to be slightly more careful with integral estimates. Recall that $\pi(C_{m_1,m_2}(r))$ covers the loop $\{|z| = const.\}$ $m_1$ times, thus ${\rm Vol}(C_{m_1,m_2}(r))=m_1{\rm Vol}(C_{0}(r))$. A factor of $m_1$ also appears in the computation for the first eigenvalue of the laplacian, since the Rayleigh quotient formula involves integration. However, since $m_1$ is fixed, all the results of Lemma~\ref{model} all carry over to the quasi-bad case as well.
  
    \subsection{Existence of the Special Lagrangian Fibration}
  
We now return to the general setting $\pi: X_{\Delta} \rightarrow \Delta$, which we assume has no singular fibers in $\pi: X_{\Delta^*} \rightarrow \Delta^* = \Delta \setminus\{0\}$.  As above $\pi^{-1}(0) = D$ is a singular fiber of type $I_k$.  Let $\sigma: \Delta^*\rightarrow X_{\Delta^*}$ be a holomorphic section of the fibration.  Again, using the Abel-Jacobi map with respect to $\sigma$, we obtain a holomorphic map $F_{AJ,\sigma} : X_{\Delta^*} \rightarrow X_{mod}$,  identifying our space $X_{\Delta^*}$ with the model fibration $X_{mod}$. Since we have just constructed a family of special Lagrangians with respect to the model semi-flat metric on $X_{mod}$, this identification gives a family of special Lagrangians with respect to  $\omega_{sf,\sigma,\epsilon}$, which we denote by $L_r$,   living in a neighborhood of $D$ in $X_\Delta^*$.

 We now construct a special Lagrangian fibration in $X_\Delta^*$ with respect to the Ricci flat metric $\omega_{CY}$. The key input is the following theorem from \cite{Collins-Jacob-Lin}.      
 
 \begin{thm}[Theorem 5.5, \cite{Collins-Jacob-Lin}]\label{thm: oneToMany}
      Let $(X,g)$ be a complete hyperK\"ahler surface,   fix a point $x_0 \in X$, and let $r(x)= d(x_0,x)$.  Suppose that
      \begin{enumerate}
      	\item The sectional curvature of $(X,g)$ is bounded.
      	\item There is a non-increasing function $f : [0,\infty) \rightarrow \mathbb{R}_{>0}$ such that \\ $\int_{0}^{+\infty} f(s)ds = +\infty$, and
        $      
      	{\rm inj}(x) \geq f(r(x)).
      	$
      	\item $X$ has finite Euler characteristic; $\chi(X) < +\infty$.
      \end{enumerate}
      Assume that there exists a (possibly immersed) special Lagrangian torus $L$ with $[L]\in H_{2}(X,\mathbb{Z})$ primitive, and $[L]^2=0$.  Then
      \begin{enumerate}
      	\item $X$ admits a special Lagrangian fibration with $L$ as one of the fibers.
      	\item There are at most $\chi(X)$ singular fibers, each classified by Kodaira, and no fiber is multiple.
      	\item $L$ is a smooth embedded torus.
      \end{enumerate}
  \end{thm}

Using Theorem~\ref{thm: oneToMany}, we prove the main result of this section:
  \begin{thm} \label{154}
     Assume that $Y$ is a rational elliptic surface and $D$ is a singular fiber of type $I_k$. Let $X:=Y\setminus D$, and let $[\omega]_{dR}\in H^2(X,\mathbb{R})$ be a K\"ahler class which is rational near infinity.  If $\omega_{CY}$ denotes any of the Calabi-Yau metrics constructed in Theorem~\ref{thm: nonStHein}, then $(X,\omega_{CY})$ admits a special Lagrangian fibration with fibers homologous to the unique quasi-bad cycle $[C_{m_1,m_2}]$ satisfying $[\omega].[C_{m_1,m_2}]=0$. 
  \end{thm}   
   \begin{proof}
   	  By Theorem \ref{thm: oneToMany}, it suffices to construct a single special Lagrangian tori representing a fixed class $[L]$. First, write the Ricci-flat metric  as $\omega_{CY}=\alpha \omega_{sf, \sigma,\frac{\epsilon}{\alpha}}+i\partial \bar\partial\phi$, where the function $\phi$ satisfied the exponential decay
	      \be
	   |\nabla^k\phi|_{\omega_{\alpha}}=O(e^{-\delta r^{\frac 23}}).\nonumber
	   \ee
Here $\omega_{\alpha}$ is the background metric constructed in Theorem~\ref{thm: nonStHein}, and $\delta >0$ a uniform constant.  Recall $r$ is the geodesic distance from a fixed point $x_0\in X_{\Delta^*}$ towards $D$. First, using the section $\sigma$, we transplant the model Lagrangians $C(r)$ constructed above to obtain Lagrangians $L_r\subset X_{\Delta^*}$  with respect to $\alpha \omega_{sf, \sigma,\frac{\epsilon}{\alpha}}$. Following the argument of  the proof of  \cite[Theorem 1.1]{Collins-Jacob-Lin}, we can now apply Moser's trick to deform the Lagrangians $L_r$ to disjoint Lagrangian tori $\ti L_{r}$ with respect to  $\omega_{CY}$.  Furthermore, combining Lemma \ref{model}, the estimates of \cite[Section 2]{Collins-Jacob-Lin}, and the decay estimates of $\phi$,  there is a constant $C>0$ such that 
	   	     	    \begin{enumerate}
   	    	
   	    	\item The second fundamental form $\Pi$ of $\ti L_r$ satisfies $\parallel \Pi \parallel \sim O(r^{-1})$.
   	    	\item The mean curvature $H$ of $\ti L_r$ satisfies $\parallel H\parallel\sim O(e^{-\delta r^{ 2/3}})$.
		\item The volume of $\ti L_{r}$ satisfies $C^{-1} < {\rm Vol}(\ti L_{r}) < C$.
		\item $\ti L_r$ is $\kappa=C^{-1}$ non-collapsed at scale $\sim O(r^{-1/3})$. 
   	    	\item The first eigenvalue $\lambda_1$ of the Laplacian on $\ti L_r$ satisfies $\lambda_1= O(r^{-2/3})$.
   	    
   	    \end{enumerate}	
	    
	    This control of the geometry of $\ti L_r$ is precisely what we need to carry out the mean curvature flow argument from \cite[Section 4.3]{Collins-Jacob-Lin}.  According to \cite[Definition 4.16]{Collins-Jacob-Lin} the Lagrangians $\ti L_{r}$ have $(C,K,\delta)$ bounded geometry and hence the argument from  \cite[Theorem 4.22]{Collins-Jacob-Lin} can be applied verbatim.  
   \end{proof}
  \begin{rk}
  	 Different choices of quasi-bad cycle $C_{m,n}$ will arise as special Lagrangian fibers of Ricci-flat metrics $\omega_{CY}$ arising from Theorem \ref{154} that are asymptotic to different (in general, not uniformly equivalent) semi-flat metrics. 
  \end{rk}
   \begin{rk}\label{rk: locModFib}
   In fact, it's not hard to see that the manifolds $C_{t}(r) = \{ {\rm Im}(x) = t, |z|=r \}$ are special Lagrangian in the model space $X_{mod}$ with the holomorphic volume form $\Omega= \frac{dz\wedge dx}{z}$.  Furthermore, these manifolds, equipped with the restriction of the standard semi-flat metric, can easily be shown to be isometric to the manifolds $C_0(r)$ in the above discussion.  Thus, if $[\omega_{CY}]_{dR}$ is rational near infinity, then $(X,\omega_{CY})$ admits a Lagrangian fibration in a neighborhood of infinity with fibers having $(C,K, \delta)$ bounded geometry in the sense of \cite[Definition 4.16]{Collins-Jacob-Lin}.  By \cite[Proposition 5.24]{Collins-Jacob-Lin}, this approximately special Lagrangian fibration is deformed by the Lagrangian mean curvature flow to the special Lagrangian fibration.
   \end{rk}
   
 \begin{rk}
   Given a log Calabi-Yau surface $(Y,D)$, where $D$ is a wheel of $k$ rational curves, Gross-Hacking-Keel \cite{GHK} constructed an integral affine manifold $U^{trop}$, which is the substitute for the base of SYZ fibration for the Gross-Siebert mirror symmetry program. Pascaleff \cite[Section 5]{Pa} proved that the algebraic affine structure $U^{trop}$ agrees with the complex affine structure of the model fibration induced by Hitchin's construction~\cite{Hit}. From \cite[Proposition 5.24]{Collins-Jacob-Lin}, the special Lagrangian constructed in Theorem \ref{154} is asymptotic to the model fibration. Therefore, the complex affine structure of the special Lagrangian fibration constructed in Theorem \ref{154} is asymptotic to $U^{trop}$, after identifying points in the base of the fibration by the Lagrangian mean curvature flow.
 \end{rk}
 \begin{rk}
 	  Given the special Lagrangian fibration in $X$ from Theorem \ref{thm: oneToMany}, here we make some speculation related to the work of Gross-Hacking-Keel. Following the argument in \cite{Lin}, one can show that the open Gromov-Witten invariants defined in \cite{Lin2} coincide with a weighted count of tropical discs. Here the tropical discs have each edge affine with respect to the complex affine structure from the special Lagrangian fibration. One can further prove that the weighted count of tropical $\mathbb{A}^1$-curves can be computed from the weighted count of the tropical discs and the former are independent of the deformation of the special Lagrangian fibration. Now if we consider a family of K\"ahler classes with shrinking exceptional divisors, we expect the limit of distances of singular fibres is infinity. Then we speculate that the special Lagrangian fibrations inside a compact set of $X$ are converging to the pull-back of the toric fibration from the toric model of $(Y,D)$. On the other hand, the open Gromov-Witten invariants for pull-back of the toric fibration from the toric model of $(Y,D)$ is computed in \cite{BCHL} and they coincide with the counting of $\mathbb{A}^1$-curves. This will lead to a proof for the folklore conjecture that the counting of $\mathbb{A}^1$-curves can be computed via certain open Gromov-Witten invariants. 
    \end{rk}

    We conclude this section by noting the following proposition
   
   \begin{prop}
   Let $Y$ be a rational elliptic surface with complex structure $J$ and let $D\subset Y$ be an $I_k$ singular fiber.  Let $[\omega]_{dR}\in H^{2}_{dR}(X,\mathbb{R})$ be a K\"ahler class which is rational near infinity and let $\omega_{CY}$ be the Calabi-Yau metric in $[\omega]_{dR}$ constructed by Theorem~\ref{thm: nonStHein} asymptotic to a quasi-regular semi-flat metric.  Denote by $\pi: Y\setminus D \rightarrow \mathbb{R}^2$ the special Lagrangian fibration of $(Y\setminus D, \omega_{CY})$ constructed in Theorem~\ref{154}.  Then, after hyperK\"ahler rotating to a complex structure $\check{J}$ so that so that $\pi: (Y\setminus D, \check{J}) \rightarrow \mathbb{C}$ is a holomorphic genus $1$ fibration, there is a rational elliptic surface $\pi': Y'\rightarrow \mathbb{P}^1$ with an $I_k$ singular fiber $D'= (\pi')^{-1}(\infty)$ such that $(Y\setminus D, \check{J}) \cong Y'\setminus D'$ and $\pi = \pi'|_{Y'\setminus D'}$.
   \end{prop}
   \begin{proof}
   The proof is identical to the proof of \cite[Theorem 6.4]{Collins-Jacob-Lin}.  From Proposition \cite[Proposition 5.24]{Collins-Jacob-Lin}, the model special Lagrangian fibration in $X_{mod}$ flows by Lagrangian mean curvature flow to the special Lagrangian fibration in the neighborhood of the infinity in $X$. In particular, the monodromy of the model special Lagrangian coincides with the monodromy at infinity of the special Lagrangian fibration constructed in Theorem \ref{154}. The former is conjugate to $\begin{pmatrix}
   1 & k \\
   0 & 1
   \end{pmatrix}$.    
   Following the proof of \cite[Lemma 6.1]{Collins-Jacob-Lin}, one can compactify $X'$ to an elliptic surface $Y'$ by adding an $I_k$ fiber at infinity. By the classification of compact complex surfaces, $Y'$ is a rational elliptic surface. For more details we refer the details of the proof to \cite[Theorem 6.4]{Collins-Jacob-Lin}.

   \end{proof}

\section{Moduli of Ricci-flat metrics and applications}\label{sec: moduli}

The goal of this section is to to understand the moduli of the Ricci-flat K\"ahler metrics asymptotic to semi-flat metrics.  The main result of this section is the following uniqueness theorem.

\begin{thm}\label{thm: unique}
Let $Y$ be a rational elliptic surface and $D$ an $I_k$ singular fiber.   Suppose $\omega_1, \omega_2$ are complete K\"ahler metrics on $X= Y\setminus D$ satisfying
\begin{enumerate}
\item[$(i)$]  $[\omega_1]_{dR}=[\omega_2]_{dR}$ in $H^{2}_{dR}(X,\mathbb{R})$,
\item[$(ii)$] for $i=1,2$ we have 
\[
\omega_i^2 = \alpha^2 \Omega \wedge \overline{\Omega}
\]
for some $\alpha >0$,
\item[$(iii)$] there is a (possibly non-standard) semi-flat metric $\omega_{sf, \epsilon}$ such that, for $i=1,2$ we have
\[
|\omega_i - \alpha\omega_{sf,\frac{\epsilon}{\alpha}}|_{\alpha\omega_{sf,\frac{\epsilon}{\alpha}}} \leq Cr^{-4/3},
\]
where $r$ denotes the distance from a fixed point $x\in X$ with respect to either $\omega_1$ or $\omega_2$.
\end{enumerate}
Then $\omega_1=\omega_2$.
\end{thm}

The decay rate $r^{-4/3}$ is essentially optimal.  Indeed, by Lemma~\ref{lem: transAsymp}, given a semi-flat K\"ahler metric $\omega_{sf, \sigma, b_0, \epsilon}$, one can produce, via a translation, a new semi-flat metric $\omega_{sf, \sigma', b_0, \epsilon}$ which is de Rham cohomologous, but not Bott-Chern cohomologous, and decays towards $\omega_{sf, \sigma, b_0, \epsilon}$ at the rate $r^{-4/3}$.  This indicates that the decay rate $r^{-4/3}$ is the best possible rate that can be obtained without imposing some extra conditions on the model semi-flat metric at infinity.

Theorem~\ref{thm: unique} yields the following corollary, which shows that the Calabi-Yau metrics constructed by Hein's method in Theorem~\ref{thm: nonStHein} are unique, so in particular do not depend on any choices made during the construction.

\begin{cor}\label{cor: HeinUnique}
Fix a Bott-Chern cohomology class $[\omega]_{BC} \in H_{BC}^{1,1}(X,\mathbb{R})$ containing a K\"ahler form $\omega$.  Then we have
\begin{itemize}
\item[$(i)$] There is a unique semi-flat metric $\omega_{sf , \sigma, b_0, \epsilon}$ in the Bott-Chern cohomology class $[\omega]_{BC}\in H^{1,1}_{BC}(X_{\Delta^*}, \mathbb{R})$.
\item[$(ii)$] For all $\alpha > 0$  there is a unique Calabi-Yau metric $\omega_{CY}$ with $[\omega_{CY}]_{BC}=[\omega]_{BC}$ solving
\[
\omega_{CY}^2= \alpha^2 \Omega\wedge \overline{\Omega}
\]
and satisfying
\[
|\omega_{CY} - \alpha \omega_{sf, \sigma, b_0, \frac{\epsilon}{\alpha}}|_{\omega_{CY}} \leq Cr^{-4/3},
\]
where $r$ is the distance from a fixed point with respect to the Calabi-Yau metric.
\item[$(iii)$] Furthermore, for all $k\in \mathbb{N}$ there is a constant $C_k$ so that $\omega_{CY}$ satisfies the improved decay estimates
\[
|\nabla^{k}(\omega_{CY} - \alpha \omega_{sf, \sigma, b_0, \frac{\epsilon}{\alpha}})|_{\omega_{CY}} \leq C_ke^{-\delta r^{2/3}}.
\]
where $\nabla$ is defined using $\omega_{CY}$.
\end{itemize}
\end{cor}
\begin{proof}
$(i)$ is Corollary~\ref{cor: locBCKahler}.  The existence statement in part $(ii)$ is Hein's theorem, Theorem~\ref{thm: nonStHein}.  The uniqueness part of $(ii)$ is Theorem~\ref{thm: unique}. Part $(iii)$ follows from Hein's estimates; see Theorem~\ref{thm: nonStHein}.
\end{proof}

The study of uniqueness of moduli of complete K\"ahler-Ricci flat metrics is a well-studied topic in K\"ahler geometry and we will not endeavor to give a complete overview of this active subject.  For ALE K\"ahler manifolds we refer the reader to the work of Kronheimer \cite{Kron}, Joyce \cite{Joyce} and Carron \cite{Car}, as well as the references therein.  Uniqueness results for ALF gravitational instantons were obtained by Minerbe \cite{Min} and Chen-Chen \cite{CCII}.  In the ALH  setting, uniqueness results were obtained by Hein \cite{Hein}, and by Chen-Chen \cite{CC} and Chen-Viaclovsky-Zhang \cite{CVZ2} in the ALG and ALG$\,^*$ cases. In higher dimensions, uniqueness results were obtained by Haskins-Hein-Nordstrom \cite{HHN} in the asymptotically cylindrical case, by Conlon-Hein \cite{CH} in the asymptotically conical case and by Chiu-Sz\'ekelyhidi \cite{CS} in the case of maximal volume growth.  Typically, uniqueness results for complete Calabi-Yau metrics are obtained by establishing a relevant Hodge theory in weighted spaces for forms with certain decay properties, which then implies the $\ddb$-lemma as a consequence.  Once a $\ddb$-lemma is established, uniqueness typically follows in a straightforward manner.  We remark that, in the case of maximal volume growth Calabi-Yau manifolds recent work of Sz\'ekelyhidi \cite{S} and Chiu \cite{Chiu} suggests that uniqueness is captured by both the asymptotic cone and some further algebro-geometric data. 

The proof of Theorem~\ref{thm: unique} follows a similar strategy to the above works, but by-passes the development of a general Hodge theory by instead proving that a $(1,1)$ form with sufficiently fast decay extends to a compactification, and then invoking the $\ddb$-lemma for currents on compact K\"ahler manifolds.  The decay estimates are used in a crucial way in this step since, as discussed in Section~\ref{sec: semiflatmet}, the $\ddb$-lemma failed in general.  Once a scalar potential is obtained, we prove estimates for the potential by invoking elliptic theory along the fibers of the elliptic fibration, using in an essential way the geometry of the semi-flat metric.

	 \begin{prop}\label{prop: ddb Lemma}
	 Suppose $d\beta$ is an exact $2$-form on $X$, of type $(1,1)$.  Suppose that there is a semi-flat metric $\omega_{sf, \sigma, b_0, \epsilon}$ on $X_{\Delta^*}$ such that 
	 \begin{equation}\label{eq: decay dB}
	 |d\beta|_{\omega_{sf,\sigma, b_0, \epsilon}} \leq Cr^{-4/3}
	 \end{equation}
	 where $r$ denotes the distance from a fixed point with respect to $\omega_{sf, \sigma, b_0, \epsilon}$. Let $\pi:Y\rightarrow \mathbb{P}^1$ denote the rational elliptic surface compactifying $X$ induced by $\sigma: \Delta^*\rightarrow X_{\Delta^*}$.  Then there exists a $\lambda \in \mathbb{R}$ and a function $\phi \in  L^{\infty}(Y) \cap C^{\infty}(X)$, unique up to addition of a constant, such that 
	 \[
	 d\beta = \ddb\left(-\lambda \log|s_{D}|^2_{\pi^*h_{FS}} + \phi\right)
	 \]
	where $h_{FS}$ is the Fubini-Study metric on $\mathcal{O}_{\mathbb{P}^1}(1)$ and $s_{D} \in H^{0}(Y, \mathcal{O}_{Y}(D))$ is the section defining $D$.
	 \end{prop}
	 \begin{proof}
	 Let $\pi:Y\rightarrow \mathbb{P}^1$ be the rational elliptic surface compactifying $X$ induced by $\sigma: \Delta^*\rightarrow X_{\Delta^*}$, and let $D$ denote the $I_k$ fiber and $X=Y\setminus D$.  Let $\tilde{\omega}$ be a smooth K\"ahler metric on $Y$. The first step is to translate the estimate~\eqref{eq: decay dB} into an estimate for $d\beta$ with respect to $\tilde{\omega}$.  Recall that in Section~\ref{sec: semiflatmet} we constructed explicit coordinate systems near the $I_k$ fiber in $Y$ compactifying $X$.  Let $\{(u,v): |u|<1, |v|<1\}$ denote one of these coordinates patches. For simplicity, denote $\omega_{sf} = \omega_{sf, \sigma, b_0, \epsilon}$ and let $g_{sf}$ denote the K\"ahler metric.  Using the coordinates on $Y$ and the explicit form for $g_{sf}$ (see Section~\ref{sec: semiflatmet} and Section~\ref{sec: slag}), there is a uniform constant $C>0$ so that following estimates hold for $g_{sf}$ in coordinates $(u,v)$
	 \[
	 \begin{aligned}
	C^{-1} \frac{|\log|uv||}{|u|^2} \leq (g_{sf})_{u\bar{u}} \leq C\frac{|\log|uv||}{|u|^2}, \\
	C^{-1}\frac{|\log|uv||}{|v|^2} \leq  (g_{sf})_{v\bar{v}} \leq C \frac{|\log|uv||}{|v|^2}.
	\end{aligned}
	 \]
	 Thanks to the fact that the semi-flat metric solves the Monge-Amp\`ere equation $\omega_{sf}^2= \Omega \wedge \overline{\Omega}$ we can easily get the bound 
	  \[
	 |(g_{sf})_{u\bar{v}}| \leq C \frac{|\log|uv||}{|uv|}.
	 \] 
	 Using the formula for the inverse of a $2\times 2$ matrix, we get
	 \[
	 \begin{aligned}
	 C^{-1} |u|^2|\log|uv|| \leq (g_{sf})^{u\bar{u}} &\leq C |u|^2|\log|uv||,\\
	 C^{-1}|v|^2|\log|uv||\leq  (g_{sf})^{v\bar{v}} &\leq |v|^2|\log|uv||,\\
	 |(g_{sf})^{u\bar{v}}| &\leq C |uv||\log|uv||.
	 \end{aligned}
	 \]
	By the discussion in Section~\ref{sec: semiflatmet} we have $r \sim (-\log|uv|)^{3/2}$.  Thus, the estimate~\eqref{eq: decay dB} implies 
	 \begin{equation}\label{eq: decay dB uv}
	 |d\beta_{u\bar{u}}| \leq \frac{C}{|u|^2|\log|uv||^3}, \quad |d\beta_{v\bar{v}}| \leq \frac{C}{|v|^2|\log|uv||^3}, \quad |d\beta_{u\bar{v}}| \leq \frac{C}{|uv||\log|uv||^3}.
	 \end{equation}
	 A key point is that each quantity appearing on the right hand side above is $L^{1}$ near $\{u=0\}\cup\{v=0\}$ with respect to the Lebesgue measure on $\mathbb{C}^2$.  With this estimate $d\beta$ extends by zero to a well-defined current $T$ on $Y$ with $L^{1}$-valued coefficients.  We claim that this current is closed and exact. 
	 
	 Let us first prove that $T$ is closed.  For the remainder of the proof we use $N_{\mu}$ to denote a tubular neighborhood of radius $0<\mu \ll1$ around $D$.  Let $\alpha$ be a smooth, compactly supported $1$-form on $Y$.  We need to show that
	 \[
	 T(d\alpha) = \lim_{\mu\rightarrow 0} \int_{Y\setminus N_{\mu}} d\beta \wedge d\alpha =0.
	 \]
	 Integration by parts yields
	 \[
	 \int_{Y\setminus N_{\mu}} d\beta \wedge d\alpha = \int_{\del N_{\mu}} \alpha \wedge d\beta.
	 \]
	It suffices to show that
	\[
	\lim_{\mu\rightarrow 0} \int_{\del N_{\mu}}\big|d\beta\big|_{\del N_{\mu}}\big| d\sigma=0
	\]
	where the norm of $d\beta$ is measured with respect to $\tilde{\omega}$ and $d\sigma$ denotes the surface measure on $\del N_{\mu}$ induced by $\tilde{\omega}$.  We may as well assume that in our local coordinates we have
	\[
	\del N_{\mu} = \{|u|=\mu, 0 \leq |v|\leq 1\} \cup \{|v|=\mu, 0 \leq |u|\leq 1\}.
	\]
	Consider the piece of $\del N_{\mu}$ given by $\{|u|=\mu, \mu \leq |v|\leq 1\}$.  We break this up into sets
	\[
	A_1= \{|u|=\mu, \mu \leq |v|\leq \mu^{\frac{1}{2}}\}, \qquad A_{2} := \{|u|=\mu, \mu^{\frac{1}{2}} \leq |v|\leq 1\}.
	\]
	On $A_1$, the estimates in~\eqref{eq: decay dB uv} gives
	\[
	|d\beta| \leq C\frac{1}{\mu^2(-\log(\mu))^3}.
	\]
	On the other hand, the area of $A_1$ with respect to the smooth K\"ahler metric on $Y$ is of order $\mu^2$.  Thus we have
	\[
	\int_{A_1} |d\beta| d\sigma \leq C \frac{1}{(-\log\mu)^3}
	\]
	and the right hand side converges to zero as $\mu\rightarrow 0$.  On $A_2$ we need to be slightly more careful.  Note that $d\beta_{u\bar{u}}|_{A_2} =0$, and so it suffices to estimate only the terms $d\beta_{u\bar{v}}, d\beta_{v\bar{v}}$. Thus, on $A_2$ we have
	\[
	\big|d\beta|_{N_{\mu}}\big| \leq C\left( \frac{1}{\mu|v|(-\log(\mu|v|))^3} + \frac{1}{\mu (-\log\mu)^3}\right).
	\]
	On the other hand, $A_2$ is a cylinder with cross-section of circumference $\mu$, and thus total has area of order $\mu$ with respect to the smooth K\"ahler metric on $Y$.  We get
	\[
	\begin{aligned}
	\int_{A_2} |d\beta\big|_{\del N_{\mu}}| d\sigma &\leq C\left( \frac{1}{(-\log\mu)^3}+ \int_{|v|=\mu^{\frac{1}{2}}}^{|v|=1}\frac{dv\wedge d\bar{v}}{|v|(-\log(|v|))^3} \right)\\
	&\leq  C\left( \frac{1}{(-\log\mu)^3}\right)
	\end{aligned}
	\]
	and the right hand side converges to zero as $\mu\rightarrow 0$.  The remaining piece of $\del N_{\mu}$ is treated identically, and the claim follows.  
	
	Next we identify the cohomology class of $T$.  Since $T$ is closed, it defines a cohomology class $[T]_{dR}\in H^{2}_{dR}(Y,\mathbb{R}) = H^{1,1}(Y,\mathbb{R})$.  Since $T|_{X} = d\beta$ is exact, $[T]_{dR}$ is in the kernel of the restriction map $H^{2}_{dR}(Y,\mathbb{R})\rightarrow H^{2}_{dR}(X,\mathbb{R})$.  Since there is no ambiguity on $Y$, we drop the subscripts and denote $[T]_{dR}=[T]$ to ease notation.  By Lemma~\ref{lem: deRham} we have
	\[
	[T] \in {\rm Span}_{\mathbb{R}} \{ [D_i]: 1 \leq i \leq k\},
	\]
	where $D_i \cong \mathbb{P}^1$ are the irreducible components of $D$.  We claim that if $[T].[D_i] =0$ for each $i$ then $[T]= \lambda [D]$.  To see this, order the $D_i$ so that $D_{i}.D_{i+1}= D_i.D_{i-1}=1$.  If $[T]= \sum_{j=1}^k a_j [D_j] $ and we have
	\[
	[T].[D_i] = a_{i+1} +a_{i-1} -2a_i ,
	\]
	where $i$ is taken ${\rm mod}\, k$.  This implies that $[T]=\lambda \sum_{i=1}^{k}[D_i]= \lambda [D]$ for some $\lambda \in \mathbb{R}$.  Indeed, if not, let $a_*=\min_{1 \leq i \leq k} a_i$ and choose $1 \leq j \leq k$ such that $a_j=a_*$ and $a_{j+1} >a_*$ where $j+1=1$ if $j=k$.  Then $[T].[D_j] >0$, a contradiction.

	To prove that $[T].[D_i]=0$, we let $h_i$ be a smooth metric on $\mathcal{O}_{Y}(D_i)$, and let $\Theta_i$ denote the curvature of $h_i$.  Then, since $[\Theta_i] = [D_i]$ we have
	\[
	[T].[D_i]  =\int_{Y}T\wedge \Theta_i = \lim_{\mu \rightarrow 0} \int_{Y\setminus N_{\mu}}d\beta \wedge \Theta_i.
	\]
Let $s_i \in H^{0}(Y, \mathcal{O}_{Y}(D_i))$ be the defining section of $D_i$.  Since $s_i$ is non-vanishing and holomorphic on $Y\setminus N_{\mu}$ we have
	\[
	[T].[D_i] =  -\lim_{\mu \rightarrow 0} \int_{Y\setminus N_{\mu}}d\beta \wedge \ddb \log|s_i|^2_{h_i}.
	\]
	We now integrate by parts on the second term to get
	\[
	[T].[D_i] = -\lim_{\mu \rightarrow 0} \int_{\del N_{\mu}}d\beta \wedge \sqrt{-1}\,\dbar \log|s_i|^2_{h_i}.
	\]
	It suffices to consider the contribution to this integral from a coordinate patch $\{(u,v): 0\leq |u|, |v|\leq 1\}$ where $D_i = \{u=0\}$.  Again, we can assume that $\del N_{\mu} = \{|u|=\mu\}$.  In this case we have
	\[
	\sqrt{-1}\dbar \log|s_i|^2_{h_i} \sim \sqrt{-1}\frac{d\bar{u}}{\bar{u}} + {\rm smooth}
	\]
	and so by our previous estimates we only need to show that
	\[
	-\lim_{\mu \rightarrow 0} \int_{|u|=\mu}d\beta \wedge \sqrt{-1}\frac{d\bar{u}}{\bar{u}} =0.
	\]
	Convert to polar coordinates, writing $u=re^{\sqrt{-1}\theta}, v= se^{\sqrt{-1}\psi}$.  Then we have
	\[
	-\lim_{\mu \rightarrow 0} \int_{|u|=\mu}d\beta \wedge \sqrt{-1}\frac{d\bar{u}}{\bar{u}} =-\lim_{\mu\rightarrow 0} \int_{r=\mu}\int_{|v|\leq 1} d\beta_{v\bar{v}} dv\wedge d\bar{v} \wedge d\theta.
	\]
	Plugging in the estimate for $d\beta_{v\bar{v}}$ yields
	\[
	\begin{aligned}
	\bigg|\lim_{\mu \rightarrow 0} \int_{|u|=\mu}d\beta \wedge \sqrt{-1}\frac{d\bar{u}}{\bar{u}} \bigg| &\leq \lim_{\mu \rightarrow 0}  C\int_{0\leq |v|\leq 1}\frac{1}{|v|^2|\log|\mu v||^3}dv\wedge d\bar{v}\\
	&= \lim_{\mu \rightarrow 0}  C\int_{0\leq |w|\leq \mu}\frac{1}{|w|^2|\log|w||^3}dw\wedge d\bar{w}\\
	&= \lim_{\mu \rightarrow 0}  \frac{C}{|\log(\mu)|^2}\\
	& =0
	\end{aligned}
	\]
	where in the second line we made the substitution $w= \mu v$.

	 Thus, we have shown that $[T] = \lambda [D]$ for some $\lambda \in \mathbb{R}$.  We observe that $\mathcal{O}_{Y}(D) = \pi^*\mathcal{O}_{\mathbb{P}^1}(1)$.  Let $h_{FS}$ denote the Fubini-Study metric on $\mathcal{O}_{\mathbb{P}^1}(1)$, and let $s_{D}= \pi^*z$ be the section of $\mathcal{O}_{Y}(D)$ vanishing on $D$ (where we have identified $\pi(D)=0 \in \mathbb{P}^1$).  Using $\sigma$ to pass to $X_{mod}$, and working in the coordinates $(x,z)$ we have
	 \[
	 \log|\sigma_{D}|^2_{h_{FS}} = \log\left(\frac{|z|^2}{1+|z|^2}\right)
	 \]
	 and from the explicit form of the semi-flat metric $\omega_{sf,b_0,\epsilon}$ we see that
	 \begin{equation}\label{eq: expDecayFirstOrderCorr}
	 |\ddb \log|s_{D}|^2_{\pi^*h_{FS}}|_{g_{sf}} \leq e^{-\delta r^{2/3}}
	 \end{equation}
	 for some $\delta >1$. Since $[T]=\lambda [D]$ in $H^{1,1}(Y, \mathbb{R})$, the $\ddb$-lemma for currents implies that we can write 
	 \[
	 T= -\lambda \ddb \log \pi^*h_{FS} + \ddb \phi
	 \]
	 for some $\phi \in L^{1}(Y)$, which is unique up to addition of a constant.  Furthermore, since $T$ is smooth on $X=Y\setminus D$, elliptic regularity implies that $\phi \in L^{1}(Y)\cap C^{\infty}(Y\setminus D)$, and from the estimate~\eqref{eq: expDecayFirstOrderCorr}, the bound $|d\beta|_{\omega_{sf,b_0,\sigma, \epsilon}} \leq Cr^{-4/3}$ implies that
	 
	 \begin{equation}\label{eq: decayddbPhi}
	 |\ddb \phi |_{\omega_{sf,b_0,\sigma, \epsilon}} \leq Cr^{-4/3}
	 \end{equation}
	 
	 It remains to prove that $\phi$ is bounded. The above bound~\eqref{eq: decayddbPhi} implies that, along the fiber $\pi^{-1}(z)$ we have
	\begin{align} \label{bound for f}
	(-\log|z|)\big| \phi_{x\bar{x}} \big|\leq C\frac{1}{(-\log|z|)^2}.
    \end{align}
	For fixed $z$, we view $\phi$ as a periodic function on $\mathbb{C}$.  Recall that the lattice generating the fiber $\pi^{-1}(z)$ is spanned by $1, \frac{k}{2\pi\sqrt{-1}}\log(z)$.  We claim that there is a uniform constant $C$ such that the following estimate holds for $\phi$;
	\begin{equation}\label{eq: osc bnd claim}
	{\rm osc}_{\pi^{-1}(z)}  \phi \leq \frac{C}{(-\log|z|)^{1/2}}
	\end{equation}
	provided $|z|$ is sufficiently small.   In what follows $C$ will be a constant which can change from line to line, but is always understood to be independent of $z$.  Let $R= k\sqrt{1 + \frac{1}{2\pi}|\log|z||^2}$, and define $y= R^{-1}x$. The point of this rescaling is that a fundamental domain for the torus $\pi^{-1}(z)$ is rescaled to lie within the ball of radius $1$.  Define
	 \[
	 f = \phi_{x\bar{x}}, \qquad A_{z}=\frac{2\pi}{k(-\log|z|)} \int_{\pi^{-1}(z)} \phi\, dV_x
	 \]
	 where we use the notation $dV_x:=\frac{\sqrt{-1}dx\wedge d\bar{x}}{2}$.  Set
	 \[
	 \widetilde{\phi}(y) = R^{-2}(\phi(Ry)-A_{z}),  \qquad \widetilde{f}(y) =f(Ry)= \widetilde{\phi}_{y\bar{y}}(y) .
	 \]
	  Since $\widetilde{\phi}$ is defined on all of $\mathbb{C}$, the standard elliptic regularity estimate yields
	  \[
	  \|\widetilde{\phi}\|_{W^{2,2}(B_2)} \leq C(\|\widetilde{f}\|_{L^{2}(B_4)} + \|\widetilde{\phi}\|_{L^{2}(B_{4})}).
	  \]
	  Recall that in dimension $2$ the Sobolev Imbedding theorem gives $W^{2,p} \hookrightarrow C^{0}$ for any $p>1$.  Combining this with H\"older's inequality, we obtain
	  \[
	  \|\widetilde{\phi}\|_{L^{\infty}(B_1)} \leq C\left(\|\widetilde{f}\|_{L^{2}(B_4)} + \|\widetilde{\phi}\|_{L^{2}(B_{4})}\right).
	  \]
	  Using $\|\widetilde{f}\|_{L^{2}(B_4)}  \leq C\|f\|_{L^{\infty}(\pi^{-1}(z))}$ and writing everything in terms of $x$, we get
	  \[
	 R^{-2} \|\phi(x)- A_{z}\|_{L^{\infty}(\pi^{-1}(z))} \leq C\left( \|f\|_{L^{\infty}} + R^{-3}\|\phi-A_{z}\|_{L^{2}(B_{4R})}\right).
	  \]
	  Note there is a fixed integer $N$ so that $B_{4R}$ is covered by $\sim NR$ translates of the fundamental domain and so
	  \[
	  \|\phi(x)- A_{z}\|_{L^{\infty}(\pi^{-1}(z))} \leq C\left( \frac{1}{(-\log|z|)} + \|\phi-A_{z}\|_{L^{2}(\pi^{-1}(z),\, dV_x)}\right).
	  \]
	  Recall that the flat torus $(\pi^{-1}(z),dV_x)$ has a uniform Poincar\'e inequality of the form
	  \[
	  \int_{\pi^{-1}(z)} |\phi -A_z|^2 \,dV_x  \leq C (-\log|z|)^2 \int_{\pi^{-1}(z)} |\phi_{x}|^2 \, dV_x .
	  \]
	  From the equation $\phi_{x\bar{x}}=f$ and integration by parts we obtain
	  \[
	  \int_{\pi^{-1}(z)} |\phi_{x}|^2\, dx  \leq C\|\phi-A_{z}\|_{L^{2}(\pi^{-1}(z))}\cdot \|f\|_{L^{2}(\pi^{-1}(z))}.
	  \]
	  Combining these three estimates with the bound for $f$ in \eqref{bound for f} yields
	   \[
	  \|\phi(x)- A_{z}\|_{L^{\infty}(\pi^{-1}(z))} \leq C\left( \frac{1}{(-\log|z|)}+ \frac{1}{(-\log|z|)^{1/2}} \right),
	  \]
	  which implies~\eqref{eq: osc bnd claim}.
	  
	  In order to obtain the $L^{\infty}$ bound we will show that there exists an $x$ the such that $\limsup_{z\rightarrow 0} |\phi(x,z)|<+\infty$.  In fact, we claim that for almost every $c \in \mathbb{C}$ the function $\phi(x,z)$ is $L^{\infty}$ when restricted to the slice $\{x=c\}$.  Since the coordinates $(x,z)$ do not extend over the $I_k$ fiber $D$, it is more convenient to work in the local coordinates $(u,v)$.  In these coordinates a slice $\{x={\rm const.}\}$ is given by $\{u ={\rm const.}\}$, and we use $v$ as a coordinate on the slice.  
	  
	  Consider the functions $\phi, \phi_{v\bar{v}}$.  Since each of these functions is smooth on $X$, and $L^{1}$ on $Y$, it follows from Fubini's theorem that they are $L^{1}$ on almost every slice $\{u =c\}$.  We claim that for almost every $c$, $\phi_{v\bar{v}}$ is the weak Laplacian of $\phi$ restricted to the slice $\{u =c\}$.  Let $\eta(v)$ be a smooth, compactly supported function.  We need to show that for almost every $u$
	  \[
	  \int \phi_{v\bar{v}}(u,v) \eta(v) dv =  \int \phi(u,v) \eta_{v\bar{v}}(v) dv .
	  \]
	  By Fubini's theorem both sides of this equation are $L^1$ functions of $u$.  Let $\psi(u)$ be any smooth compactly supported function.  Then $\psi(u)\eta(v)$ is compactly supported and so
	   \[
	 \int \int \phi_{v\bar{v}}(u,v) \eta(v)\psi(u) du  dv =  \int \int \phi(u,v) \eta_{v\bar{v}}(v)\psi(u) du dv.
	  \]
	  Thus we have
	  \[
	 \int \left(\int \bigg(\phi_{v\bar{v}}(u,v) \eta(v)-\phi(u,v) \eta_{v\bar{v}}(v)\bigg) dv \right)\psi(u) du= 0.
	  \]
	  Since this holds for every smooth compactly supported function $\psi(u)$ we have
	  \[
	  \int \bigg(\phi_{v\bar{v}}(u,v) \eta(v)-\phi(u,v) \eta_{v\bar{v}}(v)\bigg) dv =0 \quad \text{ a.e. } u.
	  \] 
	  
	  We can therefore choose a slice $\{u=c\}$, $c\ne 0$ such that $\phi(c,v)$ is $L^{1}$ and, as a function of $v$ on the slice, the distributional Laplacian $\phi(c,v)_{v\bar{v}}$ is also $L^1$.  At the same time, the estimate~\eqref{eq: decay dB uv} implies that there is a constant $C>0$ so that
	  \begin{equation}\label{eq: ddb phi est on slice}
	  |\phi(c,v)_{v\bar{v}}| \leq \frac{C}{|v|^2(-\log|v|)^3}.
	  \end{equation}
	  Consider
	  \begin{equation}\label{eq: phiplus}
	  \phi_+(v) := \phi(c,v) + \frac{C}{(-\log|v|)}.
	  \end{equation}
	  The function $\phi_+$ is $L^1$, smooth away from $v=0$ and by estimate~\eqref{eq: ddb phi est on slice}, $\phi_+$ satisfies $ \ddb \phi_+ \geq 0$ in the sense of distributions.  Thus, we may redefine $\phi_+$ at $v=0$ to make it upper semi-continuous.
	  By the maximum principle for subharmonic functions, $\phi_+$ is bounded above.  But since $\frac{1}{(-\log|v|)}$ is bounded, we get an upper bound $\sup_{v\ne 0} \phi(c,v) \leq C'$ (note that we omit the origin from the supremum, since $\phi$ may differ from $\phi_{+}$ there). Arguing in the same way for the superhamonic function
	  \begin{equation}\label{eq: phiminus}
	   \phi_-(v) := \phi(c,v) - \frac{C}{(-\log|v|)}
	\end{equation}
	yields a lower bound $\inf_{v\ne 0} \phi(c,v) >-C'$.  This yields the estimate 
	\[
	\|\phi(c,v)\|_{L^{\infty}} <C'
	\]
	for a uniform constant $C'$.  Combining this estimate with~\eqref{eq: osc bnd claim} we conclude that $\phi \in L^{\infty}(Y)$.  
		 \end{proof}

	 With this Proposition~\ref{prop: ddb Lemma} we are in position to prove Theorem~\ref{thm: unique}
	 
	 \begin{proof}[Proof of Theorem~\ref{thm: unique}]
	 Let $\omega_1, \omega_2$ be as in the statement of the theorem.  By Proposition~\ref{prop: ddb Lemma} there is a $\lambda \in \mathbb{R}$ and a function $\phi \in L^{\infty}(Y)\cap C^{\infty}(X)$  such that 
	 \[
	 \omega_2= \omega_1+\ddb \psi  \qquad \psi =-\lambda \log|s_{D}|^2_{\pi^*h_{FS}} + \phi
	 \]
	 After possibly swapping $\omega_1, \omega_2$, we may assume that $\lambda \leq 0$, so that $\psi$ is bounded from above. Since $\omega_1, \omega_2$ solve the same Monge-Amp\`ere equation we have
	 \[
	 1  = \left(\frac{\omega_2^2}{\omega_1^2}\right)^{\frac{1}{2}} \leq \frac{1}{2}{\rm Tr}_{\omega_1}\omega_2 =  1+\frac{1}{2}\Delta_{\omega_1}\psi
	 \]
	 and so $\psi$ is subharmonic.  Note that if $\lambda <0$, then we contradict the maximum principle and so we conclude $\lambda =0$.  In any case, since $\omega_1$ has volume growth
	 \[
	 {\rm Vol}_{\omega_1}(B_{R}) \sim R^{4/3}
	 \]
	 it follows from \cite[Theorem 3.5]{Karp} that $\psi$ is constant.
	 
	 \end{proof}
	 \begin{rk}
	 	 It is likely that the argument above can be generalized to the complement of other singular fibres in a rational elliptic surfaces.  In particular, this may give a different proof of Torelli theorem of gravitational instantons of type $ALG,ALG^*$ by Chen-Viaclovsky-Zhang \cite{CVZ2}. However, determining exactly the relevant decay rates, and whether these decay rates are compatible with the action of the group of local sections on the space of semi-flat metrics would likely require a case-by-case analysis similar to that carried out in Section~\ref{sec: semiflatmet}.
	 \end{rk}
	 	
	In \cite{TY}, Tian-Yau developed a robust technique for constructing complete Ricci-flat metrics on non-compact K\"ahler manifolds with trivial canonical bundle. The authors mention \cite[p. 579]{TY} that the uniqueness of such metrics is likely related to the automorphism group of the manifold.  The following proposition confirms this expectation in the present setting.

	 \begin{prop}\label{prop: reducedModuli}
	Let $Y$ be a rational elliptic surface and $D$ and $I_k$ singular fiber.  Suppose $\omega_1, \omega_2$ are two complete Calabi-Yau metrics on $X= Y\setminus D$ with the following properties
	\begin{itemize}
	\item[$(i)$] $\omega_i^2= \alpha^2 \Omega \wedge \overline{\Omega}$, for $i=1,2$, and
	\item[$(ii)$] $[\omega_1]_{dR} = [\omega_2]_{dR} \in H^{2}_{dR}(X,\mathbb{R})$.
	\item[$(iii)$] There are (possibly non-standard) semi-flat metrics $\omega_{sf, \sigma_i, b_{0,i}, \epsilon_i}$ such that
	\[
	|\omega_i-\alpha \omega_{sf, \sigma_i, b_{0,i}, \frac{\epsilon_i}{\alpha}}| \leq Cr_i^{-4/3}
	\]
	where $r_i$ is the distance from a fixed point with respect to $\omega_i$.
	\end{itemize}
	 Then there is a fiber preserving holomorphic map $\Phi \in {\rm Aut}_0(X,\mathbb{C})$ such that $\Phi^*\omega_2=\omega_1$.
	 \end{prop}
	 \begin{proof}
	 From $(ii), (iii)$ it follows that
	 \[
	 [ \omega_{sf, \sigma_1, b_{0,1}, \frac{\epsilon_1}{\alpha}}]_{dR} =  [\omega_{sf, \sigma_2, b_{0,2}, \frac{\epsilon_2}{\alpha}}]_{dR} \in H^{2}(X_{\Delta^*}, \mathbb{R})
	 \]
	 and hence the parameters appearing in the semi-flat metrics satisfy $\epsilon_1= \epsilon_2=\epsilon$ and $b_{0,1} =b_{0,2}=b_0$.  Using $\sigma_1$ we may identify $X_{\Delta^*}$ with $X_{mod}$.  Then, setting $\omega_{sf,i} =  \omega_{sf, \sigma_i, b_{0}, \frac{\epsilon}{\alpha}}$ then we have
	 \[
	 \omega_{sf,2} = T_{h}^*\omega_{sf,1}
	 \]
	 for $h(z)[\sqrt{-1}W d\bar{x}] \in H^{0}(\Delta^*, R^1\pi_*\mathcal{O}_{X})$ where $h(z)$ is holomorphic on $\Delta^*$.  We claim that there is a section $\tau \in H^{0}(\mathbb{P}^1-\{\infty\}, R^1\pi_{*}\mathcal{O}_{X})$ such that 
	 \[
	  h(z)[\sqrt{-1}W d\bar{x}] -\tau|_{\Delta^*} = h^+(z) [\sqrt{-1}W d\bar{x}]
	 \]
	 where $h^+(z)$ is holomorphic in the whole disk $\Delta$. Assuming this claim for the moment, we have
	 \[
	 T_{-\tau}^*\omega_{sf,2} = T_{h^+}^*\omega_{sf,1}.
	 \]
	 Now, since $h^+$ is holomorphic on $\Delta$ it follows from Lemma~\ref{lem: transAsymp} that
	 \[
	 |T_{-\tau}^*\omega_{sf,2} - \omega_{sf,1}|_{\omega_{sf,1}} \leq Cr^{-4/3}
	 \]
and hence $T_{-\tau}^*\omega_2, \omega_1$ satisfy the assumptions of Theorem~\ref{thm: unique} and so we conclude that $\omega_1= T_{-\tau}^*\omega_2$, as desired.  It only remains to prove the claim.

First we claim that if $Y$ is the rational elliptic surface compactification $X$ induced by $\sigma_1$, then the section $[\sqrt{-1}Wd\bar{x}]$ extends to a trivialization of $R^1\pi_{*}\mathcal{O}_{Y}$ on $\Delta$.  Indeed, by a direct calculation using the coordinates in Section~\ref{sec: compAndCoord} one can check that $\frac{dx\wedge dz}{\pi^*dz}$ extends over the $I_k$ fiber as a non-vanishing section of  $K_{Y/\mathbb{P}^1}$.  Using that $\pi_*K_{Y/\mathbb{P}^1} =(R^1\pi_{*}\mathcal{O}_{Y})^{\vee}$,   one then observes that $[\sqrt{-1}Wd\bar{x}]$ is precisely dual to $\frac{dx\wedge dz}{\pi^*dz}$ by relative duality. 

Recall that $R^1\pi_{*}\mathcal{O}_{Y} \simeq \mathcal{O}_{\mathbb{P}^1}(-1)$ \cite[Chapter I, Lemma 3.18]{FM}.  Let $[z_1:z_2]$ be homogeneous coordinates on $\mathbb{P}^1$ and identify $\infty = [0:1]$.  Let $U_i= \{z_i\ne 0\}\subset \mathbb{P}^1$.  Let $\hat{z} = \frac{z_2}{z_1}$ be a holomorphic coordinate on $U_1$, and $z= \hat{z}^{-1}$ be a holomorphic coordinate on $U_2$.  We fix the usual trivializations of $\mathcal{O}_{\mathbb{P}^1}(-1)$
\[
\begin{aligned}
\eta_1= (1, \frac{z_2}{z_1}) \quad \text{ on } U_1, \qquad \eta_2= (\frac{z_1}{z_2}, 1) \quad \text{ on } U_2.
\end{aligned}
\]
Write $h(z)[\sqrt{-1} W d\bar{x}] = \tilde{h}(z) \eta_2$ where $\tilde{h}(z)$ is holomorphic on a punctured neighborhood of $\infty \in U_2$.  Decompose $\tilde{h} = \tilde{h}^+ + \tilde{h}^{-}$ where $\tilde{h}^{-}$ is the principal part of the Laurent series expansion of $\tilde{h}(z)$ and $\tilde{h}^+$ extends as a holomorphic function over $\infty$.  We seek a holomorphic function $f=f(\hat{z})$ defined on all of $\mathbb{C}$ such that
\[
\hat{z}f(\hat{z}) = \tilde{h}^{-}(\hat{z}).
\]
Since $\hat{z} = \frac{1}{z}$ and $\tilde{h}^{-}$ is the principle part of the Laurent series expansion of $\tilde{h}(z)$, this formula defines $f(\hat{z})$ as an entire holomorphic function.  Thus $\tau = f\eta_1$ is the desired section in $H^{0}( \mathbb{P}^1-\{\infty\}, R^1\pi_{*}\mathcal{O}_{X})$.  Now, since $[\sqrt{-1}Wd\bar{x}]$ defines a holomorphic section of $R^1\pi_{*}\mathcal{O}_{Y}$ in a neighborhood of $\infty$ we have
\[
\tilde{h}(z) \eta_2- \tau = \tilde{h}^+(z)\eta_{2} = \tilde{h}^+(z)e^{\rho(z)} [\sqrt{-1}Wd\bar{x}]
\]
where $\rho(z)$ is holomorphic in a neighborhood of $\infty \in \mathbb{P}^1$. The claim is proved.

\end{proof}

	 \subsection{Applications}
	 
	 Our main important application of the uniqueness result is to define the K\"ahler moduli for the non-compact Calabi-Yau manifold $X=Y\setminus D$ where $Y$ is a rational ellliptic surface and $D$ is an $I_k$ singular fiber.  
	 
	 \begin{defn}\label{defn: CYmoduli}
	 Define $\widetilde{\mathcal{K}_{CY}}$ to be the set of complete K\"ahler metrics $\omega$ on $X$ satisfying:
	 \begin{itemize}
	 \item[$(i)$] $\omega^2= \Omega \wedge \overline{\Omega}$,
	 \item[$(ii)$] there is a constant $C>0$ such that $\omega$ satisfies
	 \[
	 |\omega- \omega_{sf,\sigma, b_0, \epsilon}|_{\omega} \leq Cr^{-4/3}
	 \]
	 where $\omega_{sf,\sigma, b_0, \epsilon}$ is the unique semi-flat metric in $[\omega]_{BC} \in H^{1,1}(X_{\Delta^*}, \mathbb{R})$, and $r$ is the distance from a fixed point with respect to $\omega$.
	 \end{itemize}
	 We define
	 \[
	 \mathcal{K}_{CY} := \widetilde{\mathcal{K}_{CY}}/{\rm Aut}_0(X,\mathbb{C})
	 \]
	 to be the moduli of asymptotically semi-flat, complete Calabi-Yau metrics.
	 \end{defn}

	 \begin{cor}
	 $\mathcal{K}_{CY} \sim \mathcal{K}_{dR,X}$.  In particular, $\mathcal{K}_{CY}$ has dimension $11-k$.
	 \end{cor}
	 \begin{proof}
	 For any K\"ahler class $[\omega] \in \mathcal{K}_{dR,X}$, Theorem~\ref{thm: nonStHein} (or more precisely Corollary~\ref{cor: HeinUnique}) implies that there is a Calabi-Yau metric $\omega_{CY} \in \widetilde{\mathcal{K}}_{CY}$ with $[\omega_{CY}] = [\omega]$.  Furthermore, by Proposition~\ref{prop: reducedModuli} and Hein's estimates, as recalled in Corollary~\ref{cor: HeinUnique}, this metric is unique up to the action of ${\rm Aut}_0(X,\mathbb{C})$.
	 \end{proof}
	 
	 \begin{rk}\label{rk: RicModuli}
	 Rather than defining the moduli space of Calabi-Yau metrics asymptotic to a semi-flat metric, one could instead study the space of Calabi-Yau metrics asymptotic to a rescaled semi-flat metric $\alpha \omega_{sf, \sigma, b_0, \frac{\epsilon}{\alpha}}$. In this case the uniqueness result Proposition~\ref{prop: reducedModuli} implies that, after modding out by the action of ${\rm Aut}_{0}(X,\mathbb{C})$, we get a moduli space, denoted $\mathcal{M}_{Ric} \sim \mathcal{K}_{dR, X} \times \mathbb{R}_{>0}$ where $\mathcal{K}_{dR, X}$ is the de Rham K\"ahler cone of $X$, and $\mathbb{R}_{>0}$ is identified with the parameter $\alpha>0$.  This yields a $12-k$ dimensional moduli space.  Furthermore, $\mathcal{M}_{Ric}$ has a natural projection onto $\mathcal{K}_{CY}$ with $\mathbb{R}_{>0}$ fibers.
	 \end{rk}

	 Let us describe a second application of the uniqueness result. Let $\check{Y}$ be a del Pezzo surface of degree $0<k\leq 9$ and $\check{D}=s^{-1}(0)$ be a smooth anti-canonical divisor with $s\in H^0(\check{Y},-K_{\check{Y}})$, and let $\check{\Omega} = \frac{1}{s}$ be a holomorphic volume form on $\check{Y}\setminus \check{D}$.  Let $h$ be a positively curved hermitian metric on $-K_{\check{Y}}$ such that the restriction to $\check{D}$ agrees with the (unique up to scale) hermitian metric on $-K_{Y}|_{D}$ whose curvature $\check{\omega}_{\check{D}}$ is the flat metric on $\check{D}$ in the K\"ahler class $c_{1}(\check{Y})|_{\check{D}}$. Tian-Yau \cite{TY} proved that there is a complete, exact K\"ahler metric $\check{\omega}_{TY}$ on $\check{Y}\setminus \check{D}$ asymptotic to the model 
	 \[
	 \omega_{mod} =  \frac{\sqrt{-1}}{2\pi}\partial \bar{\partial}\big(-\log{ |s|_{h_{\check{D}}}^2}   \big)^{\frac{3}{2}}
	 \]
	 solving the Monge-Amp\`ere equation $2\check{\omega}_{TY}^2 = \check{\Omega} \wedge \overline{\check{\Omega}}$. This model geometry is often referred to as the Calabi model. In the late 1980s, Yau \cite{YauComm} asked for a characterization of the hyperK\"ahler rotation of the Tian-Yau metric, both in terms of the symplectic structure and a conjectural compactification.  In the authors' previous work \cite{Collins-Jacob-Lin} we proved (see Theorem~\ref{513}) that any choice of simple closed geodesic $\gamma\subset D$ gives rise to a special Lagrangian fibration
	 \[
	 \check{\pi}_{\gamma}:\check{Y}\setminus \check{D} \rightarrow \mathbb{R}^2.
	 \]
	Normalize the holomorphic volume form so that ${{\rm Im}(\check{\Omega})_{\check{\pi}_{\gamma}^{-1}(b)}=0}$ for $b\in \mathbb{R}^2$, the base of the special Lagrangian fibration, the authors proved
	 \begin{thm}[Theorem 1.5, \cite{Collins-Jacob-Lin}] \label{HKthm} 
	 With the notation above. Let $X$ be the hyperK\"ahler rotation of $\check{X}$ with K\"ahler form and holomorphic volume form given by 
	  \begin{align} \label{HKrel}
   	 \omega&={\rm Re}\check{\Omega}, \notag \\
   	  \Omega&=\check{\omega}-\sqrt{-1}{\rm Im}\check{\Omega}.
   \end{align} 
   Then $X$ can be compactified to a rational elliptic surface $Y$ by adding an $I_k$ fiber $D$ at infinity.
\end{thm} 
	
	Recall that the Tian-Yau metric on $\check{Y}\setminus\check{D}$ is asymptotic (with exponential decay) to the {\em Calabi model} $(\mathcal{C}, \Omega_{\mathcal{C}}, \omega_{\mathcal{C}})$, see \cite[Section 3]{HSVZ}, or Appendix~\ref{app: hkRot} for more discussion.  In Appendix~\ref{app: hkRot} we show that the hyperK\"ahler rotation in Theorem~\ref{HKthm}, applied to the Calabi model, produces exactly a rescaled semi-flat metric
	\[
	{\rm Re}\check{\Omega}_{\mathcal{C}} = \alpha \omega_{sf, \sigma, b_0, \frac{\epsilon}{\alpha}}.
	\]
	Furthermore, the explicit dependence of $\alpha, \epsilon$ on the modulus of the elliptic curve $D$ is computed.  Generic choices of the modulus $\tau$ of the elliptic curve $D$ yield comparatively small values of $\alpha, \epsilon$.  However, if $D_t$ is a family of tori approaching a nodal curve and the geodesic $\gamma_t$ is chosen to be the vanishing cycle, then one can easily check that the calculations in Appendix~\ref{app: hkRot} yield $\alpha \rightarrow +\infty$.  Therefore, we have
	 
	 \begin{prop}\label{prop: hkCal}
	 In the above setting, the symplectic form of the hyperK\"ahler rotated Calabi-Yau structure converges exponentially fast to a rescaled semi-flat metric.  Furthermore, this symplectic form is asymptotic to a standard semi-flat symplectic form, in the sense of Definition~\ref{defn: rel sf metric}, if and only if after the action of $SL(2,\mathbb{Z})$ we have
	 \[
	 D = \mathbb{C}/(\mathbb{Z}+\tau \mathbb{Z})
	 \]
	 with ${\rm Re} (\tau) =0$,  and $[\gamma]$ is the cycle generated by $0, \tau$.  In particular, the hyperK\"ahler rotation of the Tian-Yau metric is equal to a generalized Hein metric produced by Theorem~\ref{thm: nonStHein}.
	 \end{prop}
	 
	 \begin{rk}\label{rk: HSVZII}
	  In \cite[Remark 2.4]{HSVZ}, Hein-Sun-Viaclovsky-Zhang observed that the Calabi model is related to the semi-flat metric on the complement of an $I_k$ singular fiber by hyperK\"ahler rotation.  They commented that this observation could lead to an identification of the global hyperK\"ahler rotation of the Tian-Yau space with the asymptotically semi-flat metrics constructed by Hein \cite{Hein}.   Following the appearance of this paper, Hein-Sun-Viaclovsky-Zhang \cite{HSVZII} proved a general result concerning compactification of ``asymptotically Calabi" manifolds, which yields, in particular, a different proof of Theorem~\ref{HKthm}.
	  \end{rk}

\section{Mirror Symmetry and Applications} \label{sec: MS}
In this section we combine our results from Sections~\ref{sec: semiflatmet}-~\ref{sec: moduli} to prove a version of SYZ mirror symmetry.

 \subsection{Collapsing of Del Pezzo Surfaces near the Large Complex Structure Limit}
 The authors are not aware of a notion of the large complex structure limit for pairs $(\check{Y},\check{D})$. Here we will propose a notion of large complex structure limit from the viewpoint of Strominger-Yau-Zaslow conjecture
 
 \begin{defn}\label{def: LCSL}
 A large complex structure limit of a Calabi-Yau pair $(\check{Y},\check{D})$ with $\check{Y}, \check{D}$ smooth is a family of pairs $\pi: (\check{\mathcal{Y}}, \check{\mathcal{D}}) \rightarrow \Delta$ such that, for each $t\in \Delta$, $\check{Y}_{t}:= \pi^{-1}(t)$ is smooth, $\check{Y}_{1} \cong \check{Y}$, $\check{D}_t:= \pi^{-1}(t)$ is smooth for $t\in \Delta^*$ and $\pi:\check{\mathcal{D}}\rightarrow \Delta$ is a large complex structure limit of $\check{D}$.
 \end{defn}
 
In our setting of a pair of a del Pezzo surface with a smooth anti-canonical divisor, $(\check{Y}_t,\check{D}_t)$ goes to a large complex structure limit as $\check{D}_t$ converges to a nodal curve. Kontsevich-Soibelman proposed that the Calabi-Yau manifolds collapsed to affine manifolds with singularities at the large complex structure limit \cite{KS}. This proposal has been confirmed in the case of K3 surfaces and hyperK\"ahler manifolds with abelian fibrations \cite{GW, CVZ, TZ} and a weak formulation is proved in the case of Fermat hypersurfaces \cite{Li19}. 
 
 Let $(\check{Y}_t,\check{D}_t)$ be a flat family of pairs of smooth del Pezzo surfaces $\check{Y}_t$ and anti-canonical divisors $\check{D}_t\in |-K_{\check{Y}_t}|$ over a disc $\Delta \subseteq \mathbb{C}$ such that $\check{D}_t$ are smooth for $t\neq 0\in \Delta$. Denote $\check{X}_t=\check{Y}_t\setminus \check{D}_t$. Let $\check{\omega}_t$ be the corresponding Tian-Yau metric on $\check{X}_t$, for $t\neq 0$. Assume that $\check{D}_0$ is a nodal curve. Let $\alpha,\beta\in H_1(\check{D}_t,\mathbb{Z})$ be a basis and $\alpha$ be the vanishing cycle. Let $\tilde{\alpha}, \tilde{\beta}\in H_2(\check{X}_t,\mathbb{Z})$ denote the homology class of unit $S^1$-bundle (in the normal bundle of $D_t$) over $\alpha,\beta$. Then there exists a unique holomorphic volume form $\Omega_t$ on $\check{X}_t$ such that 
 \[
 2\check{\omega}_t^2=\check{\Omega}_t\wedge \bar{\check{\Omega}}_t, \qquad \int_{\tilde{\alpha}}\check{\Omega}_t\in \mathbb{R}_+.
 \]
From Theorem \ref{513} there exists a special Lagrangian fibration on $\check{X}_t$ with respect to $(\check{\omega}_t,\check{\Omega}_t)$ with fiber class $\tilde{\alpha}\in H_2(\check{X}_t,\mathbb{Z})$. We will explain the behavior of the special Lagrangian fibration as $t\rightarrow 0$. 
 
 Let $l_t(\alpha),l_t(\beta)$ denote length of the geodesics in $D_t$ with respect to the flat metric associated to $c_1(-K_{Y_t})|_{D_t}$. Then $l_t(\alpha)=O((\log{|t|})^{-1/2})\rightarrow 0$ and $l_t(\beta)=O((\log{|t|})^{1/2})\rightarrow \infty$ as $t\rightarrow 0$. One has 
 \begin{align*}
 |\int_{\tilde{\alpha}}\check{\Omega}_t|\rightarrow 0,\quad
 |\int_{\tilde{\beta}}\check{\Omega}_t|\rightarrow \infty,
 \end{align*} 
 as $t\rightarrow 0$ from \cite[Lemma 4.6]{Collins-Jacob-Lin}. 
 We are interested in describing how the corresponding Riemannian metrics $\check{g}_t$ degenerate. Since the Tian-Yau metric is asymptotic to the Calabi ansatz, the torus fibers homologous to $\tilde{\alpha}$ are volume collapsing from equations \cite[Equations (4.1),(4.2)]{Collins-Jacob-Lin}.  To sum up, we have      
 \begin{lem}\label{lem: collapsing}
 	Consider the log Calabi-Yau surface $(\check{X}_t, \check{g}_t)$ with special Lagrangian fibration with fibers homologous to $\tilde{\alpha}$.  Then, as $t\rightarrow 0$, the volume of the special Lagrangian fibers collapses to zero. 
 \end{lem}
Lemma \ref{lem: collapsing} and the remark below justify Definition~\ref{def: LCSL} from the SYZ perspective. 
 \begin{rk}
 	Given sequence of triple $(\check{Y}_i,\check{D}_i,p_i )$ such that $(\check{Y},\check{D})$ converging to the large complex structure limit in Definition \ref{def: LCSL},  the point $p_i$ in the ``finite region" of $\check{X}_i=\check{Y}_i\setminus \check{D}_i$. Rescale the Tian-Yau metric on $\check{X}_i$ such that distances between the singular fibres of the special Lagrangian fibrations corresponding to $\tilde{\alpha}$ are bounded independent of $i$. Then the special Lagrangian fibrations collapse with respect to the pointed Gromov-Hausdorff topology \cite{LT}. 
 \end{rk}

 \subsection{SYZ Mirror symmetry between Del Pezzo surfaces and rational elliptic surfaces}
 It is well-known that the del Pezzo surfaces and the rational elliptic surfaces are mirror pairs. In this section, we will focus on SYZ mirror symmetry between del Pezzo surfaces and rational elliptic surfaces. Recall from Conjecture~\ref{conj: introSYZ} that the Strominger-Yau-Zaslow conjecture predicts the existence of a special Lagrangian fibration on a Calabi-Yau $\check{X}$ and a $T$-dual special Lagrangian fibration giving the mirror Calabi-Yau $X$. The special Lagrangian fibrations are dual in the sense that they interchange the induced complex and symplectic affine structures on the base of the SYZ fibrations, see \cite[Conjecture 6.6]{G2}. In this section we will define a mirror map from the complex moduli of del Pezzo pairs to the K\"ahler moduli of a rational elliptic surface with an $I_k$ singular fiber, demonstrating that the special Lagrangian fibrations constructed by the authors in \cite{Collins-Jacob-Lin} (see Theorem~\ref{513}) and Theorem \ref{154} are $T$-dual. 
 
Let us first recall some mirror symmetry of log Calabi-Yau surfaces.
 The mirror of a del Pezzo surface $\check{Y}$ of degree $k$ relative to an anti-canonical divsior with at worst nodal singularities is a Landau-Ginzburg superpotential $W:M\rightarrow \mathbb{C}$, where $M$ is a complex surface and $W$ is a holomorphic function. The fibers are punctured elliptic curves which can be partially compactified to obtain a rational elliptic surface with an $I_k$ fiber at infinity. Auroux-Kartzarkov-Orlov \cite{AKO} proved that the Fukaya-Seidel category of $W$ is equivalent to the derived category of coherent sheaves on $\check{Y}$. Notice that the Fukaya-Seidel category doesn't depend on the (almost) complex structure on $M$.  Hacking-Keating \cite{HK} recently generalized the work of \cite{AKO} to obtain homological mirror symmetry of log Calabi-Yau surfaces $(Y,D)$ where $D$ is a singular nodal curve in $|-K_{Y}|$. They showed that the mirror of $(Y,D)$ is a Calabi-Yau surface $M$ with a mirror superpotential $W:M\rightarrow \mathbb{C}$, which encodes the counting of holomorphic discs intersecting $D$ and with boundaries on the putative SYZ fibers. Hacking-Keating further found that there exists a distinguished complex structure $(Y_e,D_e)$ in the deformation family of $(Y,D)$ such that the derived category of $Y_e$ is equivalent to the Fukaya-Seidel category of $W$ with an exact symplectic form.

Now we review the Torelli theorem of log Calabi-Yau surfaces of Gross-Hacking-Keel \cite{GHK1} (see also \cite{Fr}). Let $(Y,D)$ be a log Calabi-Yau pair. The restriction of Picard group $\mbox{Pic}(Y)\rightarrow \mbox{Pic}(D)\cong \mathbb{C}^*$ induces the marking 
    \begin{align*}
       \alpha_{(Y,D)}: \Lambda(Y,D)\rightarrow \mathbb{C}^*, 
    \end{align*} 
    where $\Lambda(Y,D)$ denotes the subgroup of $\mbox{Pic}(Y)$ perpendicular to each component of $D$.
 The Torelli theorem for log Calabi-Yau surfaces \cite{GHK1} says that a deformation of $(Y,D)$ is determined by its marking. We will use the following weak Torelli theorem; 
 \begin{thm}[Theorem 1.8, \cite{GHK1}]\label{thm: Torelli}
 	If $(Y,D),(Y',D')$ are two deformation equivalent log Calabi-Yau pairs such that $\alpha_{(Y,D)}=\alpha_{(Y',D')}$ under the identification $H^2(Y,\mathbb{Z})\cong H^2(Y',\mathbb{Z})$ from parallel transport, then $(Y,D)\cong (Y',D')$. 
 \end{thm} 
From the surjectivity of the period map \cite[Theorem 3.17]{Fr}, there exists a distinguished point $(Y_e,D_e)$ such that $\alpha_{(Y_e,D_e)}\equiv 1$, which is the distinguished complex structure in the work of Hacking-Keating \cite{HK} mentioned above. 
 The marking can also be understood via the classical periods. Consider the following long exact sequence 
 \begin{align}\label{relative homology}
 0 = H_3({Y})\rightarrow H_3({Y},{Y}\setminus {D})\rightarrow H_2({Y}\setminus {D})\rightarrow H_2(\check{Y})\rightarrow H_2({Y},{Y}\setminus {D}).
 \end{align}
 By Poincar\'e duality $H_i({Y},{Y}\setminus \check{D})\cong H^{4-i}({D})$. Any $2$-cycle in ${Y}$ having zero intersection with each component of ${D}$ can be deformed to a cycle in $H_2(X)$ (see \cite[p.22]{Fr}). Recall from Section~\ref{sec: semiflatmet} that the bad cycle $[C]$ is well-defined up to a multiple of the elliptic fiber.  It follows that the integral $\int_{[C]}\Omega$ is well-defined and hence we normalize it to be $1$. Then 
 \begin{align}\label{classical period}
 \big( \gamma\mapsto \exp{(2\pi i\int_{\gamma}\Omega)}  \big) \in \mbox{Hom}(\Lambda(Y,D),\mathbb{C}^*)
 \end{align}
 from the exact sequence (\ref{relative homology}). Moreover, the classical period \eqref{classical period} coincides with $\alpha_{(Y,D)}$ by \cite[Proposition 3.12]{Fr}. 

With the above understanding of the periods, the Torelli theorem limits the possible rational elliptic surfaces obtained from Theorem \ref{HKthm}.
\begin{prop}
	If $(Y,D)$ is a log Calabi-Yau pair coming from Theorem \ref{HKthm}, then
	its period $\alpha_{(Y,D)}\in \mbox{Hom}(\Lambda(Y,D),S^1)$.
\end{prop}
\begin{proof}
	This is the direct consequence of the fact that ${\rm Im}\Omega=\omega$ is exact from equation \eqref{HKrel} and the description of the periods in \eqref{classical period}.
\end{proof}	
 
 	From the classification of surfaces there are $10$ families of del Pezzo surfaces: one for each degree $k\neq 8$ and two for $k=8$ which are $\mathbb{P}^1\times \mathbb{P}^1$ and the Hirzebruch surface $\mathbb{F}_1$. 	
 	On the other hand, it is well-known that there are $10$ deformation families of rational elliptic surfaces with an $I_k$ fiber, for $k\in \{1,\cdots, 9\}$ (see \cite[Propositions 9.15, 9.16]{Fr}). There is one for each $k\neq 8$ and two for $k=8$ which correspond to the mirror families of $\mathbb{P}^1\times \mathbb{P}^1$, and $\mathbb{F}_1$. Recall that the mirror superpotential for $\mathbb{P}^1\times \mathbb{P}^1$ is 
 	\begin{align*}
 	W_{\mathbb{P}^1\times \mathbb{P}^1}=x+y+x^{-1}+y^{-1}:(\mathbb{C}^*)^2\rightarrow \mathbb{C},
 	\end{align*} 
	with fibers quadruple-punctured elliptic curves. One can fiberwise compactify to an elliptic fibration and add an $I_8$ fiber at infinity to get a rational elliptic surface $Y_8$. We can similarly get another rational elliptic surface $Y_{8'}$ from the mirror superpotential $W_{\mathbb{F}_1}=x+y+\frac{1}{xy}+\frac{1}{x}$. 
 	We claim that if $\check{Y}\cong \mathbb{P}^1\times \mathbb{P}^1$, then the rational elliptic surface $Y$ obtained from hyperK\"ahler rotation according to Theorem~\ref{HKthm} belongs to the deformation family containing $Y_8$ and that if $\check{Y}\cong \mathbb{F}_1$, then $\check{Y}$ is in the same deformation family as $Y_{8'}$.   This can be determined from purely topological reasoning.  Indeed, we first look at the long exact sequence of relative homology 
 	\begin{align} \label{long exact sequence check}
 	H_2(\check{Y})\rightarrow H^2(\check{D})\cong \mathbb{Z}\rightarrow H_1(\check{X})\rightarrow H_1(\check{Y})=0.
 	\end{align} 
	When $\check{Y}\cong \mathbb{P}^1\times \mathbb{P}^1$, $H_2(\check{Y})$ is generated by the two rulings and each generator intersects $\check{D}$ twice and so $H_1(\check{X})\cong \mathbb{Z}_2$. When $\check{Y}\cong \mathbb{F}_1$, $H_2(\check{Y})$ is generated by the fibers and the $(-1)$-section. The latter has intersection $1$ with $\check{D}$ and thus $H_1(\check{X})=0$. On the other hand, applying the long exact sequence of relative homology on a rational elliptic surface  
 	\begin{align}\label{long exact sequence}
 	H_2(Y)\rightarrow H^2(D)\cong \mathbb{Z}^8\rightarrow H_1(X)\rightarrow H_1(Y)=0, 
 	\end{align} 
	where the first map is $ H_2(Y)\ni C\mapsto (D_i \mapsto C.D_i)$.  Here we denote by $D_1,\cdots, D_8$ the components of $D$ in cyclic order. In the case of $Y=Y_8$, there are four sections intersecting $D_1,D_3, D_5, D_7$ (with suitable cyclic permutation of the labeling). Direct calculation shows that $H_1(X)\cong \mathbb{Z}_2$. On the other hand, in the case of $Y=Y_{8'}$, there are four sections intersecting $D_1,D_2,D_4, D_7$. Direct calculation shows that $H_1(X)=0$. Since hyperK\"ahler rotation does not change the topology, the rational elliptic surfaces in Theorem \ref{HKthm} always fall in the expected deformation families from the view point of mirror symmetry. 
 
 With the above discussion we will prove the SYZ mirror symmetry between del Pezzo surfaces and rational elliptic surfaces. Let us set up the notation to be used below: let $\check{\mathcal{M}}_{k}$ be the moduli space of pairs $(\check{Y},\check{D})$, where $\check{Y}$ is a del Pezzo surface of degree $k$, $\check{D}$ is a smooth anti-canonical cycle. There is a Torelli theorem for marked pairs due to McMullen \cite{M}: first there is a fibration from $\check{\mathcal{M}}_k$ to the $j$-line of elliptic curves by sending a pair to the $j$-invariant of the anti-canonical divisor. Recall the long exact sequence 
 \begin{align*}
    H^1(\check{D})\rightarrow H_2(\check{X})\rightarrow H_2(\check{Y})\rightarrow H^2(\check{D}).
 \end{align*} Let $\tilde{\alpha},\tilde{\beta}$ be generators of the image of  $H^1(\check{D})\rightarrow H_2(\check{X})$. 
   The periods $\tau_1=\int_{\tilde{\alpha}}\check{\Omega},\tau_2=\int_{\tilde{\beta}}\check{\Omega}$ determine the complex structure of $\check{D}$ from  \cite[Lemma 4.6]{Collins-Jacob-Lin}. Let $\gamma_i$, $1\leq i \leq 9-k$ be the elements in $H_2(\check{X})$ forming a basis with $\tilde{\alpha},\tilde{\beta}$. Then a marking of the pair $(\check{Y},\check{D})$ is given by the data 
     \begin{align*}
       \int_{\gamma_i}\check{\Omega}\in \mathbb{C}/(\mathbb{Z}\tau_1\oplus \mathbb{Z}\tau_2)\cong \check{D}, 
     \end{align*} which is equivalent to an element of $\mbox{Hom}(\Lambda(\check{Y},\check{D}),\mathbb{C}^*)$. The later (up to a quotient of finite group) is the fibre of the projection from $\check{\mathcal{M}}_k$ to the $j$-line. 
 In particular, the moduli space $\check{\mathcal{M}}_k$ has complex dimension $10-k$.

There is a connected $\mathbb{Z}^2$-covering of $\check{\mathcal{M}}_k$ with fibers being elements of the first fundamental group of the corresponding smooth anti-canonical divisor. Let $\check{\mathcal{M}}_{cpx}$ be the loci where the last component is primitive and denote the fiber of the universal family over $\check{q}\in \check{\mathcal{M}}_{cpx}$ by $(\check{Y}_{\check{q}},\check{D}_{\check{q}},\alpha_{\check{q}})$.
 Write $\check{X}_{\check{q}}=\check{Y}_{\check{q}}\setminus \check{D}_{\check{q}}$ with Tian-Yau metric $\check{\omega}_{\check{q}}$ and holomorphic volume $\check{\Omega}_{\check{q}}$.  We fix the above following data (including normalizations) by requiring
 \begin{enumerate}
 	\item The complex Monge-Ampere equation  $2\omega_{\check{q}}^2=\check{\Omega}_{\check{q}}\wedge \overline{\check{\Omega}}_{\check{q}}$ holds.
 	\item $\int_{\tilde{\alpha}_{\check{q}}}\check{\Omega}_{\check{q}}\in \mathbb{R}_+$, where $\tilde{\alpha}_{\check{q}}\in H_2(\check{X}_{\check{q}},\mathbb{Z})$ is the special Lagrangian class corresponding to $\alpha_{\check{q}}$.   
 	\item Fix a primitive class $\beta_{\check{q}}\in H_1(\check{D}_{\check{q}},\mathbb{Z})$ with $\langle \alpha_{\check{q}},\beta_{\check{q}}\rangle=m\in \mathbb{Z}$. Then  $\int_{\tilde{\beta}_{\check{q}}}\mbox{Im}\check{\Omega}_{\check{q}}=m$, where $\tilde{\beta}_{\check{q}}\in H_2(\check{X}_{\check{q}},\mathbb{Z})$ the special Lagrangian class corresponding to $\beta_{\check{q}}$. \footnote{The normalization does not depend on the particular choice of $\beta_q$ chosen.}
 \end{enumerate}

 On the other hand, we let $X_e=Y_e\setminus D_e$ denote the complement of the $I_k$ fiber in the distinguished rational elliptic surface and let $E$ be an elliptic fiber. Let $\mathcal{M}_{K}$ be the complexified K\"ahler moduli of $X_e$, which is defined to be
 \begin{align*}
 \mathcal{M}_{\text{K\"ah}}=\{\mathbf{B}+\sqrt{-1}[\omega]: \mathbf{B}\in H^2(X_e,\mathbb{R}/2\pi \mathbb{Z}), [\omega]\in \bigcup_{\substack{(m_1,m_2) \in \mathbb{Z}^2\\ m_1>0,\, {\rm gcd}(m_1,m_2)=1}} V_{m_1,m_2}\}. 
 \end{align*} 
 where $V_{m_1,m_2}$ is defined in~\eqref{eq: rationalPlanes}.  In other words, we require that $[\omega]_{dR}$ is rational near infinity in the sense of Definition~\ref{defn: rational}.\footnote{This condition is necessary and sufficient for the existence of a special Lagrangian fibration}
 
 \begin{rk}
 The moduli space $\mathcal{M}_{\text{K\"ah}}$ is precisely the complexification of the points in $\mathcal{K}_{CY}$ (see Definition~\ref{defn: CYmoduli}) which are rational near infinity, in the sense of Definition~\ref{defn: rational}.
 \end{rk}
 
 Given $q\in \mathcal{M}_{\text{K\"ah}}$, let $\mathbf{B}_q+\sqrt{-1}[\omega_q]$ to be the corresponding complexified K\"ahler class. Let $\omega_q$ be a Ricci-flat metric on $X_e$ in the K\"ahler class $[\omega_q]_{dR}$ and asymptotic to the semi-flat metric $\alpha_q\omega_{sf,\sigma', b_0,\frac{\epsilon}{\alpha_q}}$ guaranteed by Corollary \ref{cor: HeinUnique} for $\alpha_{q}\gg 0$ (this parameter will be fixed below).  Here $\epsilon$ is the size of the elliptic fiber with respect to $[\omega_q]_{dR}$ and $\sigma'$ is a local holomorphic section in a neighborhood of infinity. Let $\Omega_q$ be the unique holomorphic volume form with simple pole along $D_e$ such that 
   \begin{enumerate}
   	\item $2\omega_q^2=\Omega_q\wedge \bar{\Omega}_q$, and
   	\item ${\rm Im} \Omega_q$ is exact on $X_e$\footnote{The phase we choose here is $\pi/2$ different from the previous sections.}. 
   \end{enumerate} 
   The special Lagrangian in the class of the quasi-bad cycle $C_{q}$ defined by the cohomology class $[\omega_q]_{dR}$ (which is rational near infinity) has volume $\int_{C_q} \Omega_q=m\alpha_q$ if $C_q$ is an $m$-quasi-bad cycle. 
From \eqref{long exact sequence}, the dimension of $\mathcal{M}_{\text{K\"ah}}$ is $10-k$ and hence matches with the dimension of $\check{\mathcal{M}}_{cpx}$.

 \begin{thm} \label{SYZ MS}
 	There exists a mirror map $\check{\mathcal{M}}_{cpx}\rightarrow \mathcal{M}_{K}$ sending $\check{q}\in \check{\mathcal{M}}_{cpx}$ to $q(\check{q})\in \mathcal{M}_{\text{K\"ah}}$ 
 	such that 
 	\begin{enumerate}
 		\item 	the special Lagrangian fibration on $X_{q(\check{q})}$ with respect to $(\omega_{q(\check{q})},\Omega_{q(\check{q})})$ from Theorem \ref{154} and the special Lagrangian fibration on $\check{X}_{\check{q}}$ constructed in Theorem~\ref{513} exchange their complex and symplectic affine structures. 
 		\item the volumes of torus fibers are inverse of each other.
 	\end{enumerate}
 
 \end{thm} 
 \begin{proof}

 	Given a triple $(\check{Y}_{\check{q}},\check{D}_{\check{q}}, \alpha_{\check{q}})\in \check{\mathcal{M}}_{cpx}$, there exists a special Lagrangian fibration $\check{X}_{\check{q}}\rightarrow B_{\check{q}}$ from Theorem \ref{513}.  Moreover, from Theorem \ref{HKthm}, the same underlying space of $\check{X}_{\check{q}}$ with K\"ahler form $\mbox{Re}\check{\Omega}_{\check{q}}$ and holomorphic $2$-form $\check{\Omega}_{\check{q}}'={\rm Im}\check{\Omega}_{\check{q}}+\sqrt{-1}\check{\omega}_{\check{q}}$ can be compactified to a rational elliptic surface $\check{Y}_{\check{q}}'$ by adding an $I_k$ fiber $\check{D}'_{\check{q}}$ at infinity and $\check{\Omega}_{\check{q}}'$ extends as a meromorphic $2$-form with simple pole along $\check{D}'_{\check{q}}$. Denote $\check{X}'_{\check{q}}=\check{Y}'_{\check{q}}\setminus \check{D}'_{\check{q}}$. Then $\tilde{\beta}_{\check{q}}$ is an $m$-quasi bad cycle of the rational elliptic surface $\check{Y}'_{\check{q}}$ and $\int_{\tilde{\beta}_{\check{q}}}\check{\Omega}'_{\check{q}}=m$ from the normalization of $\check{\Omega}_{\check{q}}$. Let $\tau_{\check{q}}$ be in the fundamental region such that $\check{D}_{\check{q}}\cong \mathbb{C}/(\mathbb{Z}\oplus \mathbb{Z}\tau_{\check{q}})$ where $\alpha_{\check{q}}$ corresponds to the cycle generated by $0,1$.  It follows that $\beta_{\check{q}}$ corresponds to $n+m\tau_{\check{q}}$ for some $n\in \mathbb{Z}$. The volume of the special Lagrangian fibers is $\int_{\tilde{\alpha}_{\check{q}}}\check{\Omega}_{\check{q}}=\frac{1}{{\rm Im}\tau_{\check{q}}}$ thanks to the choice of normalization $\int_{\tilde{\beta}_{\check{q}}}{\rm Im}\check{\Omega}_{\check{q}}=m$.
 	
 	On the other hand, given $\mathbf{B}_q+\sqrt{-1}[\omega]_q\in \mathcal{M}_{\text{K\"ah}}$ and a choice of $m$-quasi bad cycle $C_q$,  Theorem \ref{154} yields the existence of a special Lagrangian fibraton $X_{q}\rightarrow B_q$ with respect to $(\omega_q,\Omega_q)$ with fibers homologous to $C_q$. Moreover, the same underlying space $X_e$ with K\"ahler form ${\rm Re}{\Omega}_{q}$ and holomorphic $2$-form
 	 \begin{align}\label{HK1}
 	 	  \Omega'_q=\omega_q-\sqrt{-1}{\rm Im}{\Omega}_q
 	 \end{align}	 
  can be compactified to a rational elliptic surface $Y'_q$ by adding an $I_k$ fiber $D'_q$ at infinity and $\Omega'_q$ is meromorphic on $Y'_q$ with simple pole along $D'_q$. Denote $X'_{q}=Y'_q\setminus D'_q$. Then the integral of $\Omega'_q$ on an $m$-quasi bad cycle of $Y'_q$ is $m\alpha_q$.
 	Therefore, to prove the theorem it suffices to show that the rational elliptic surfaces $Y'_{q}$ and $\check{Y}'_{\check{q}}$ are biholomorphic and the meromorphic volume forms $\omega_q-\sqrt{-1}{\rm Im}\Omega_q$, and ${\rm Im}\check{\Omega}_{\check{q}}+\sqrt{-1}\check{\omega}_{\check{q}}$ have the same phase. 
 	
 	Fix a reference point $(\check{Y}_{\check{q}_0}, \check{D}_{\check{q}_0}, \alpha_{\check{q}_0})\in \check{\mathcal{M}}_{cpx}$ and a diffeomorphism $\check{X}_{\check{q}_0}\cong X_e$ sending the class of SYZ fibers in $\check{X}_{\check{q}_0}$ to a class  which is not the class of elliptic fibers of $X_e$. Such a diffeomorphism can be constructed as follows: choose a primitive class $\beta_{q_0}\in H_1(\check{D}_0,\mathbb{Z})$ with $\langle \alpha_{\check{q}_0},\beta_{\check{q}_0}\rangle =m$. Again by Theorem \ref{513}, there exists a special Lagrangian fibration on $\check{X}_{\check{q}_0}$ with fiber class $\tilde{\beta}_{\check{q}_0}$.  By Theorem~\ref{HKthm}, after hyperK\"ahler rotating to an elliptic fibration $\check{X}_{\check{q}_0}''\rightarrow \mathbb{C}$, $\check{X}_{\check{q}_0}''$ can be compactified to a rational elliptic surface  $\check{Y}_{\check{q}_0}''$ by adding an $I_k$ fiber $\check{D}''_{\check{q}_0}$ at infinity. Since there are exactly $10$ families of rational elliptic surfaces corresponding to the $10$ families of del Pezzo surfaces, $(\check{Y}_{\check{q}_0}'',\check{D}_{\check{q}_0}'')$ and $(Y_e, D_e)$ are deformation equivalent. In particular, there exists a diffeomorphism $\check{X}_{\check{q}_0}''\cong X_e$ sending $\tilde{\beta}_{\check{q}_0}$ to the class of the elliptic fiber. Since the hyperK\"ahler rotation does not change the underlying space, we have a diffeomorphism $\check{X}_{\check{q}_0}\cong X_e$ sending $\tilde{\alpha}_{\check{q}_0}$ to a non-fiber class, which is necessarily an $m$-quasi-bad cycle $C_q$. 
 	
 	From the Mayer-Vietoris sequence this diffeomorphism induces an isomorphism of lattices $H^2(Y'_{\check{q}_0},\mathbb{Z})\cong H^2(Y'_{q},\mathbb{Z})$, for any $q\in \mathcal{M}_{\text{K\"ah}}$. Again $Y'_{q}$ and $Y'_{\check{q}_0}$ are in the same deformation family and therefore, by Theorem~\ref{thm: Torelli}, it suffices to check that they share the same periods under the natural identification of the lattices arising from the deformation.

 	Under the identification $\check{X}_{\check{q}_0}\cong X_e$ and parallel transport $\check{X}_{\check{q}}\cong \check{X}_{\check{q}_0}$, one has a diffeomorphism $\psi_{\check{q}}:\check{X}_{\check{q}}\cong  X_e$. Define $\mathbf{B}_{q(\check{q})}+i[\omega_{q(\check{q})}]\in \mathcal{M}_{\text{K\"ah}}$ such that 
 	\begin{align}
 	\mathbf{B}_{q(\check{q})}+\sqrt{-1}\frac{m[\omega_{q(\check{q})}]}{\alpha_{q(\check{q})}}&=\psi_{\check{q}}^*[\mbox{PD}([\check{\sigma}_{\check{q}}])+\check{\Omega}_{\check{q}}] \label{mirror map I}\\
 	 {\rm Im}\tau_{\check{q}}&=m\alpha_{q(\check{q})}\label{mirror map II}.
 	\end{align} 
	where $\check{\sigma}_{\check{q}}$ is a choice of topological section which is flat with respect to $\check{q}$, and $\text{PD}$ denotes the Poincar\'e dual.
 	Here, we view 
 	\begin{align*}
 	[\check{\sigma}_{\check{q}}]\in H_2(\check{X}_{\check{q}},\partial\check{X}_{\check{q}};\mathbb{Z})\cong  H_2(\check{X}_{\check{q}_0},\partial\check{X}_{\check{q}_0};\mathbb{Z})\cong H^2(\check{X}_{\check{q}_0},\mathbb{Z})\cong H^2(X_e,\mathbb{Z}).
 	\end{align*} 
	Notice that \eqref{mirror map I} determines $[\omega_{q(\check{q})}]$ up to $\mathbb{R}^*$-scaling and such scaling is determined by \eqref{mirror map II}. Similar mirror maps also appeared in \cite{G2}\cite{HK} and one expects to read it off from \cite[Theorem 4.4]{RS}. One will see from the proof that the mirror map arises naturally from the SYZ perspective. Indeed,
 	straightforward calculation shows that the volume of the special Lagrangian in the class $C_q$ with respect to $(\omega_q,\Omega_q)$ is $|\int_{C_q}\Omega_q|=m\alpha_q$. Together with \eqref{mirror map II}, this implies the second part of the theorem. Recall that the elliptic fiber $E$ of $X_e$ becomes an $m$-quasi-bad cycle in $Y'_{q(\check{q})}$ and $\int_{E}\omega_q-\sqrt{-1}{\rm Im}\Omega_q=\alpha_q$. Thus, $\frac{m}{\alpha_{q(\check{q})}}\Omega'_{q(\check{q})}$ is the meromorphic volume form on $Y'_{q(\check{q})}$ with integral $1$ on the bad cycle of $Y'_{q(\check{q})}$. From Theorem \ref{thm: Torelli} of log Calabi-Yau surfaces there exists an isomorphism $\phi_{\check{q}}:\check{Y}'_{\check{q}}\cong Y'_{q(\check{q})}$. Notice that since the imaginary parts of both $\phi_{\check{q}}^*\big(\omega_{q(\check{q})}-\sqrt{-1}{\rm Im}\Omega_{q(\check{q})}\big)$ and ${\rm Im}\check{\Omega}_{\check{q}}+\sqrt{-1}\check{\omega}_{\check{q}}$ are exact, we have 
 	 \begin{align*}
 	    \phi_{\check{q}}^*\big(\omega_{q(\check{q})}-\sqrt{-1}{\rm Im}\Omega_{q(\check{q})}\big)={\rm Im}\check{\Omega}_{\check{q}}+\sqrt{-1}\check{\omega}_{\check{q}}. 
 	 \end{align*} 
	 $\phi_{\check{q}}$ induces an isomorphism on the bases of $\check{Y}'_{\check{q}}$ and $Y'_{q(\check{q})}$, denoted by  $\underline{\phi}_{\check{q}}:\mathbb{P}^1\rightarrow \mathbb{P}^1$. In particular, the restriction of $\underline{\phi}_{\check{q}}$ yields a map (still denoted by $\underline{\phi}_{\check{q}}$) identifying the bases of the two special Lagrangian fibrations, 
 	    \begin{align*}
 	       \xymatrix{
 	        \check{X}_{\check{q}}\ar[d] & X_{q(\check{q})} \ar[d] \\
 	        B_{\check{q}} \ar[r]^{\underline{\phi}_{\check{q}}}     & B_{q(\check{q})} .
          }
 	    \end{align*} 
 	    For any $b\in B_{\check{q}}$ not in the discriminant locus, $\phi_{\check{q}}$ induces an isomorphism 
 	        \begin{align*}
 	            (\phi_{\check{q}})_*:H_1\big((\check{X}_{\check{q}})_b,\mathbb{Z}\big) \cong H_1\big((X_{q(\check{q})} )_{\underline{\phi}_{\check{q}}(b)} ,\mathbb{Z}\big).
 	        \end{align*}
 	     Thus, for any $v\in T_bB_{\check{q}}$ and $\gamma\in H_1\big((\check{X}_{\check{q}})_b,\mathbb{Z}\big)$ 
 	    \begin{align*}
 	       \int_{\gamma}\iota_v \check{\omega}_{\check{q}}=\int_{(\phi_{\check{q}})_*\gamma} \iota_{(\underline{\phi}_{\check{q}})_*v} \mbox{Im}\Omega_{q(\check{q})} \\  
 	        \int_{\gamma}\iota_v \mbox{Im}\check{\Omega}_{\check{q}}=\int_{(\phi_{\check{q}})_*\gamma} \iota_{(\underline{\phi}_{\check{q}})_*v} \omega_{q(\check{q})}. 
 	    \end{align*} 
 	    In other words, the symplectic and complex affine structures of the two special Lagrangian fibrations are exchanged.

 \end{proof}

\begin{rk}
   In the proof of Theorem \ref{SYZ MS}, there is a choice of diffeomorphism $\check{X}_{\check{q}_0}\cong X_e$ sending the fiber class $\tilde{\alpha}_{\check{q}_0}\in H_2(\check{X}_{\check{q}_0},\mathbb{Z})$ to an $m$-quasi-bad cycle $C_{m,m'}\in H_2(X_e,\mathbb{Z})$. After fixing such choice, the mirror map sends $\check{\mathcal{M}}_{cpx}$ to $V_{m,m'}\subseteq \mathcal{M}_{\text{K\"ah}}$. It is worth pointing out that the monodromy group of $\check{\mathcal{M}}_{cpx}$ induces a group action on $H_1(\check{D}_{\check{q}_0},\mathbb{Z})$  acting transitively on primitive elements. Therefore, for different choices of the diffeomoprhisms $\check{X}_{\check{q}_0}\cong X_e$, the mirror maps differ by composing with a self-diffeomorphsim of $X_e$ (or $\check{X}_{\check{q}_0}$) naturally identifying components of $\mathcal{M}_{\text{K\"ah}}$. 
\end{rk}

\begin{rk}
We do not need to assume that the special Lagrangian is simple in the sense of \cite{G2}.
\end{rk}
\begin{rk}
   In summary, the mirror pairs $(\check{Y}_{\check{q}},\check{D}_{\check{q}})$ and $(Y_{q(\check{q})}, D_{q(\check{q})})$ are related by two hyperK\"ahler rotations of the form \eqref{HK1}\eqref{HKrel} from the proof. However, it is worth noticing that the symplectic structures on $\check{Y}'_{\check{q}}\cong Y'_{q(\check{q})}$ induced by hyperK\"ahler rotation are generally not the same. 
\end{rk}

 \subsection{Automorphisms of del Pezzo surfaces} Another application of the special Lagrangian fibrations constructed in \cite{Collins-Jacob-Lin} and the uniqueness result in Theorem \ref{thm: unique} is to produce automorphisms of del Pezzo surfaces fixing a given smooth anti-canonical divisor. Let $(\check{Y},\check{D})$ be a pair of a del Pezzo surface of degree $k$ and $D$ a smooth anti-canonical divisor. Equip $\check{X}=\check{Y}\setminus \check{D}$ with the Tian-Yau metric $\check{\omega}$ and a holomorphic volume form $\check{\Omega}$ having a simple pole on $\check{D}$.  Fix a choice of primitive class $[\gamma] \in H_1(\check{D}, \mathbb{Z})$ and let $X$ be the hyperK\"ahler rotation of $\check{X}$ with K\"ahler form $\omega$ and holomorphic volume form $\Omega$ given by \eqref{HKrel}.  From Proposition~\ref{prop: hkCal}, the K\"ahler form $\omega$ satisfies.
 \[
 |\omega-\omega_{sf,\sigma'',b_0,\epsilon}|_{\omega}\leq  C e^{-\delta r^{2/3}}
 \]
 for certain data $\sigma'', b_0,\epsilon$ (see Appendix~\ref{app: hkRot}).  Let $(Y,D)$ denote the pair of a rational elliptic surface with an $I_k$ fiber obtained by compactifying $X$ using the section $\sigma''$, as described in Section~\ref{sec: semiflatmet}.
 
The automorphism groups of rational elliptic surfaces $Y$ were studied by Karayayla \cite{Ka}; since $K_Y\cong \mathcal{O}_Y(-E)$, where $E$ is any fiber, any automorphism $\sigma\in \mbox{Aut}(Y)$ of $Y$ must send fibers to fibers and induce an automorphism of the base $\mathbb{P}^1$. Moreover, Karayayla proved that 
   \begin{align} \label{Aut RES}
      \mbox{Aut}(Y)=MW(Y)\rtimes \mbox{Aut}_{\sigma}(Y),
   \end{align} where $MW(Y)$ is the Modell-Weil group and $\mbox{Aut}_{\sigma}(Y)$ is the subgroup of $\mbox{Aut}(Y)$ fixing the section $\sigma$. Recall that $MW(Y)$ is the group of the sections of $Y$ with a choice of zero section $\sigma$. For $\sigma'\in MW(Y)$, we denote $T_{\sigma'}$ be the translation of the section $\sigma'$ with respect to $\sigma$, the corresponding element in $\mbox{Aut}(Y)$. 
   
 Denote by $\mbox{Aut}(Y,D)$ the group of automorphisms of $Y$ preserving $D$ and denote $\mbox{Aut}(\check{Y},\check{D})$ similarly. We first have the following proposition:
 \begin{prop} \label{aut}
 	Let $\phi \in \mbox{Aut}(Y,D)$ such that $\phi^*\Omega=\Omega$.  Then the same underlying map induces an symplectomorphism $\check{\phi}\in \mbox{Sympl}(\check{X},\check{\omega})$. Furthermore, if $\phi^*[\omega]=[\omega]\in H^2(X,\mathbb{R})$, then $\check{\phi}\in \mbox{Aut}(\check{Y},\check{D})$. 
 \end{prop}
 \begin{proof}
 	The first part of the proposition directly follows from (\ref{HKrel}). 
 	
 	Now assume that $\phi^*[\omega]=[\omega]$ and $\phi^*\Omega=\Omega$. From Appendix A, we have $|\omega-\omega_{sf,\sigma'',b_0,\epsilon}|_{\omega}\leq  C e^{-\delta r^{2/3}}$. We may replace the local section $\sigma''$ by a global section $\sigma$ up to loosening the estimate to $|\omega-\omega_{sf,\sigma,b_0,\epsilon}|_{\omega}\sim O(r^{-\frac{4}{3}})$ thanks to Lemma \ref{lem: localSecDes} and Lemma \ref{lem: transAsymp}. From \eqref{Aut RES}, we can write $\phi=\phi_1\phi_2$, where $\phi_1=T_{\sigma'}^*$ is the translation of $\sigma'$ with respect to $\sigma$ for some global holomorphic section $\sigma'$ of $Y$ and $\phi_2\in \mbox{Aut}_{\sigma}(Y)$ automorphism fixing $\sigma$.  
 	Since any automorphism $\phi_2\in \mbox{Aut}_{\sigma}(Y)$ preserves the $j$ invariant of the elliptic fiber and $\sigma$, (so on the local model $X_{mod}$ it acts as $z\mapsto e^{i\theta}z, x\mapsto x$) we have $\phi_2^*\omega_{sf,\sigma,b_0,\epsilon}=\omega_{sf,\sigma,b_0,\epsilon}$. Since $\sigma,\sigma'$ are global sections of $Y$, Lemma \ref{lem: localSecDes} and Lemma \ref{lem: transAsymp} and explicit calculation show that $|\omega_{sf,\sigma,b_0,\epsilon}-\phi_1^*\omega_{sf,\sigma,b_0,\epsilon}| \sim O(r^{-\frac{4}{3}})$. To sum up, we have $|\omega_{sf,\sigma,b_0,\epsilon}-\phi^*\omega_{sf,\sigma,b_0,\epsilon}|\sim O(r^{-\frac{4}{3}})$. 
Then $\phi^*\omega=\omega$ by Theorem \ref{thm: unique}. 
   Therefore, $\phi$ induces an automorphism $\check{\phi}$ of $\check{X}$. Since $\phi$ preserves a tubular neighborhood of $D$, $\check{\phi}$ preserves a tubular neighborhood of $\check{D}$. Since a locally bounded holomorphic function extends over the divisors, $\check{\phi}$ extends to $\check{Y}$. 
 	
 \end{proof}
 
 Proposition~\ref{aut} provides an analytic approach to study the automorphism groups $\mbox{Aut}(\check{Y},\check{D})$. 
For instance, the above proposition recovers the classical result for plane cubics (see \cite[p. 22]{Web}) in the case $\check{Y}=\mathbb{P}^2$. 
 \begin{cor}
 	Given any smooth plane cubic $E\subseteq \mathbb{P}^2$, there exists $\mathbb{Z}_3\oplus \mathbb{Z}_3\subseteq \mbox{Aut}(\mathbb{P}^2)$ fixing $E$.
 \end{cor}
 \begin{proof} From Theorem \ref{513} and Theorem \ref{HKthm}, there exists a special Lagrangian fibration on $\check{X}=\mathbb{P}^2\setminus E$ such that after hyperK\"ahler rotation, it compactifies to the extremal rational elliptic surface
 	with singular configuration $I_9I_1^3$ \cite[Corollary 1.4]{Collins-Jacob-Lin}. Any automorphism $\phi$ of $Y$ preserves the $I_9$ fiber $D$ and thus $\mbox{Aut}(Y)=\mbox{Aut}(Y,D)$.
 	From \cite[Table 11]{Ka}, there exists $\mathbb{Z}_3\oplus \mathbb{Z}_3\subseteq \mbox{Aut}(Y)$ preserving $\Omega$. With the decomposition $\phi=\phi_1\phi_2$ as in Proposition~\ref{aut}, it suffices to check that $\phi_i^*[\omega]=[\omega]\in H_2(X)$, $i=1,2$. Since $\phi_1$ is translation by a section, it is isotopic to the identity. It is easy to see that $\phi_2$ preserves the homology class of the bad cycle. Since $H_2(X)$ is generated by the fiber and the bad cycle, we have $\phi_2^*[\omega]=[\omega]$ as well. Then the corollary follows from Proposition~\ref{aut}. 
 	
 \end{proof}

\appendix

\section{hyperK\"ahler rotation and semi-flat metrics}\label{app: hkRot}

In this appendix we demonstrate how the Calabi ansatz Ricci-flat metric behaves under hyperK\"ahler rotation by constructing an explicit hyperK\"ahler rotation along a special Lagrangian fibration from the Calabi model to the semi-flat model near an $I_k$ fiber.

We begin with an explicit Ricci-flat metric on the total space of an ample line bundle over a torus, called the  Calabi model. Let $D$ be a complex torus, and $L\rightarrow D$ be a holomorphic line bundle of degree $k>0$.  Fix a primitive homology class $[\gamma_1]\in H_1(D,\mathbb{Z})$ represented by a simple closed loop, and let $[\gamma_2]\in H_1(D,\mathbb{Z})$ be a complementary simple closed loop so that $[\gamma_1].[\gamma_2]=1$.  After fixing a point $q_0\in D$, the Abel-Jacobi map identifies $D\cong\mathbb{C}/\Lambda$ for the lattice $\Lambda := \mathbb{Z} + \tau \mathbb{Z}$ with ${\rm Im}(\tau)>0$.  We choose $\gamma_2$ in such a way that $\tau$ lies in a fundamental domain for the $PSL(2,\mathbb{Z})$ action on the upper half-plane; namely ${\rm Re}(\tau) \in [-\frac{1}{2}, \frac{1}{2})$ and $|\tau| \geq 1$.  Let $\xi= \xi_1 +\sqrt{-1}\xi_2$ denote the standard complex coordinate on $\mathbb{C}$.  Through the Abel-Jacobi map we can identify $L$ with the quotient of $\mathbb{C}\times \mathbb{C}$ by the action of $\Lambda$ by
\[
\gamma. (\xi,w) \longmapsto (\xi+\gamma, e_{\gamma}(\xi)w),
\]
where 
\[
e_{\gamma}(\xi) =  a(\gamma)e^{\frac{k\pi}{{\rm Im}(\tau)}({\overline{\gamma}}\xi+\frac{|\gamma|^2}{2})}, \qquad a \in {\rm Hom}(\Lambda, S^1).
\]
Evidently, $a \in {\rm Hom}(\Lambda, S^1)$ is determined by $a(1), a(\tau)$.  Write
\[
a(1) = e^{-\sqrt{-1}k\pi\beta_1}, \qquad a(\tau) = e^{-\sqrt{-1}k\pi\beta_2}.
\]
Let $q= -\beta_2+ \beta_1\tau$.  Let $T_{q}:D\rightarrow D$ denote the translation map.  Then one can check that, up to isomorphism, $T_{q}^*L$ is the bundle determined by
\[
\tilde{e}_{\gamma}(\xi) =  \tilde{a}(\gamma)e^{\frac{k\pi}{{\rm Im}(\tau)}({\overline{\gamma}}\xi+\frac{|\gamma|^2}{2})},
\]
where $\tilde{a} \in {\rm Hom}(\Lambda, S^1)$ has $\tilde{a}(1) = \tilde{a}(\tau)=1$. In particular, after changing the base point for the Abel-Jacobi map, we can assume that $L$ is identified with the bundle of degree $k$ defined by $\tilde{a}\in {\rm Hom}(\Lambda, S^1)$.  For a nice discussion of theory of holomorphic line bundles on abelian varieties we refer the reader to \cite{Beauville}.  We can choose a metric $h=e^{-\phi}$ on $L$ for which the associated curvature form is a multiple of the flat area form, i.e.
\[
\phi(z) = \frac{k\pi}{{\rm Im}(\tau)}|z|^2.
\]
Equivalently, the curvature two from $\Theta(h)= -\ddb \log(h)$ induces a flat Riemannian metric on $D$ given by
\[
g_{D} =  \frac{2\pi k}{{\rm Im}(\tau)} (d\xi_1^2 +d\xi_2^2).\nonumber
\]
Define the space $\mathcal{C}$ to be the punctured tubular neighborhood of the zero section
\[
{\mathcal C}:=\{\zeta\in L\,|\, 0<|\zeta|_h<1\}.\nonumber
\]
The Calabi model is the space $\mathcal C$ equipped with the natural complex structure $J$,  the symplectic form
\[
\omega_{ {J}} = \frac{2}{3}\ddb(-\log|\xi|^2_{h})^{\frac{3}{2}} = -\frac{1}{3}dJd (-\log|\xi|^2_{h})^{\frac{3}{2}}.\nonumber
\]
and the holomorphic volume form with a simple pole on the zero section of $L$ (see below for an explicit formula).  We work in coordinates $(\xi_1,\xi_2)$ on the fundamental domain for $D$. Equip $L$ with complex coordinate $w=re^{\sqrt{-1}\psi}$, so a section $\zeta$ has norm $|\zeta|^2_h=|w|^2e^{-\phi}=r^2 e^{-\phi}$. Furthermore define
\[
\ell:= (-\log|\zeta|^2_{h})^{\frac{1}{2}}\nonumber
\]
and note that 
\[
d\ell = -\ell^{-1}(\frac{dr}{r} - \frac{1}{2}d\phi), \qquad Jd\ell = \ell^{-1}(d\psi +\frac{1}{2}Jd\phi).\nonumber
\]
Introduce the $1$-form $\theta := \ell J d\ell=d\psi +\frac{1}{2}Jd\phi$, and compute
\[
\begin{aligned}
\omega_{J} = -\frac{1}{3}dJd(-\log|\zeta|^2_{h})^{\frac{3}{2}} &= -\frac{1}{2}d\left((-\log|\zeta|^2_{h})^{\frac{1}{2}}J(-2\frac{dr}{r} + d\phi)\right)\\
&=d\left(\ell J(\frac{dr}{r} - \frac{1}{2}d\phi)\right)\\
&= d\ell \wedge (-\theta) -\ell\frac{1}{2}dJd\phi\\
& = \theta \wedge d\ell +\frac{2\pi k \ell}{{\rm Im}(\tau)} d\xi_1\wedge d\xi_2,
\end{aligned}\nonumber
\]
where in the last line we used our explicit formula for the curvature  $\Theta(h)$. Since the Riemannian metric is given by $g(\cdot, \cdot,) = \omega (\cdot, J \cdot)$, one can easily check 
\[
\begin{aligned}
g&= \frac{1}{\ell} \left(\left(\frac{dr}{r}-\frac{1}{2}d\phi\right)^2 + (d\psi +\frac{1}{2}Jd\phi)^2\right) + \ell g_{D}\\
&= \ell(d\ell^2+ g_{D}) + \ell^{-1}\theta^2.
\end{aligned}\nonumber
\]
Written in this way, the metric is in the form of the Gibbons-Hawking ansatz (see e.g. \cite{GW, HSVZ}). Here the harmonic function (usually denoted $V$) is simply the coordinate function $\ell$.

To get a better feel for the one-form $\theta$ is precisely the connection on the $S^1$ principal bundles $\{|s|_{h}={\rm const}\} \subset L$. induced by the Chern connection.  In particular, a parallel section of $L$ lies in the kernel of $\theta$.  

The Calabi space $\mathcal{C}$ has a natural holomorphic volume form with a simple pole on the zero section of $L$.  For our purposes we choose the normalization
\[
\begin{aligned}
\Omega_{J} &:= \sqrt{-1}\frac{\bar{\tau}}{|\tau|}\left(\frac{2\pi k}{{\rm Im}(\tau)}\right)^{\frac{1}{2}}\frac{dw}{w}\wedge (d\xi_1+\sqrt{-1}d\xi_2)\\
& = \sqrt{-1}\frac{\bar{\tau}}{|\tau|} \left(\frac{2\pi k}{{\rm Im}(\tau)}\right)^{\frac{1}{2}}(\frac{1}{2} d\phi -\ell d\ell + \sqrt{-1}d\psi)\wedge (d\xi_1+\sqrt{-1}d\xi_2).
\end{aligned}
\]
With this normalization we have
\[
\omega_{J}^2 = \frac{1}{2}\Omega_{J}\wedge \overline{\Omega}_{J}.
\]
Let $L_{(\xi_1, \xi_2)}$ denote the fiber of $L$ over $(\xi_1, \xi_2) \in D$.  Consider the torus fibration of $\mathcal{C}$ by 
\begin{equation}\label{eq: McK}
M_{c, K} = \bigg\{(\xi_1,\xi_2, s) \in D\times L_{(\xi_1,\xi_2)}\bigg| \begin{aligned} &{\rm Im}(\tau)\xi_1- {\rm Re}(\tau)\xi_2=c\\  &(-\log|s|^2_{h})^{\frac{1}{2}} = K \end{aligned} \bigg\}.
\end{equation}
The curves $\{{\rm Im}(\tau)\xi_1- {\rm Re}(\tau)\xi_2=c\}$ represent the homology class ${[\gamma_2]\in H_1(D,\mathbb{Z})}$.  Note that, up to the action of $PSL(2,\mathbb{Z})$, our choice of simple closed curve $[\gamma_2]\in H_1(D,\mathbb{Z})$ is arbitrary.  We parametrize these lines by the constant $c$, chosen so that this line intersects the $\xi_2=0$ line in the fundamental domain (ie. $\frac{c}{{\rm Im}(\tau)} \in [0,1)$); such a choice is possible since ${\rm Im}(\tau)\ne 0$.  We claim that the $M_{c,K}$ form a special Lagrangian torus fibration; ie.
\begin{equation}\label{eq: McksLag}
\omega_{J}\big|_{M_{c,K}}=0 \qquad {\rm Im}(\Omega_{J})\bigg|_{M_{c,K}}=0.
\end{equation}
The first equality is clear since $\ell=K$ on $M_{c,K}$ by definition.  For the second equality, observe that
\[
d\xi_1 +\sqrt{-1}d\xi_2\big|_{M_{c,K}} = \tau \frac{d\xi_2}{{\rm Im}(\tau)}.
\]
Since $\phi$ depends only on $\xi_1,\xi_2$, we have
\[
\begin{aligned}
\Omega_{J}\bigg|_{M_{c,K}} &= \sqrt{-1}\frac{\bar{\tau}}{|\tau|}\left(\frac{2\pi k}{{\rm Im}(\tau)}\right)^{\frac{1}{2}}\left(\sqrt{-1}d\psi \wedge \tau \frac{d\xi_2}{{\rm Im}(\tau)}.\right)\\
&= -\frac{|\tau|}{{\rm Im}(\tau)}\left(\frac{2\pi k}{{\rm Im}(\tau)}\right)^{\frac{1}{2}}d\psi \wedge d\xi_2
\end{aligned}
\]
which proves the second equality in~\eqref{eq: McksLag}.  Thus, the fibration induced by the map
\[
\pi(\ell, \psi, \xi_1, \xi_2)= (\ell, {\rm Im}(\tau)\xi_1 - {\rm Re(\tau)}\xi_2),
\]
is a special Lagrangian fibration for the Calabi-Yau structure $(\omega_J, \Omega_{J})$.  We can now hyperK\"ahler rotate so that this fibration becomes a genus $1$ holomorphic fibration.  To ease notation, let us define
\begin{equation}\label{eq: atauDef}
a_{\tau} = \frac{{\rm Im}(\tau)}{|\tau|}, \qquad b_{\tau} = -\frac{{\rm Re}{\tau}}{|\tau|},\qquad c_{\tau} = \left(\frac{2\pi k}{{\rm Im}(\tau)}\right)^{\frac{1}{2}}.
\end{equation}
Consider the symplectic forms
\begin{equation}\label{eq: ansatzHKTrip}
\begin{aligned}
\omega_{J} &= \theta \wedge d\ell +c_{\tau}^2 d\xi_1\wedge d\xi_2\\
\omega_{I} &= c_{\tau}\left(\theta\wedge d\xi_2 + \ell d\ell \wedge d\xi_1\right)\\
\omega_{K} &=c_{\tau}\left( d\xi_1\wedge\theta +\ell d\ell \wedge d\xi_2\right).
\end{aligned}
\end{equation}
One can easily check that these symplectic forms are closed and generate a hyperK\"ahler triple (using, e.g. \cite{Don06}). The associated complex structures $J, I, K$ are given by
\begin{align}
J d\ell &= \ell^{-1}\theta&Jd\xi_1 &= -d\xi_2\nonumber\\
I d\xi_2 &=c_{\tau}^{-1}\ell^{-1}\theta&    Id\xi_1 &=c_{\tau}^{-1}d\ell\nonumber\\
K d\xi_1 &= -c_{\tau}^{-1}\ell^{-1}\theta&  K d\xi_2&= c_{\tau}^{-1}d\ell.\nonumber
\end{align}

\noindent Using this hyperK\"ahler triple, we consider the symplectic form
\begin{equation}\label{eq: omegaTau}
\begin{aligned}
\omega_{\tau} &= a_{\tau}\omega_{I} +b_{\tau} \omega_{K}\\
 &= c_{\tau}\bigg( \theta \wedge (a_{\tau}d\xi_2-b_{\tau}d\xi_1) +\ell d\ell \wedge (a_{\tau}d\xi_1+b_{\tau}d\xi_2)\bigg),
 \end{aligned}
\end{equation}
along with the associated complex structure $J_{\tau}:=a_\tau I + b_{\tau} K$.

Define a $J_{\tau}$ holomorphic coordinate by
\[
y = y_1 +\sqrt{-1}y_2 :=|\tau| c_{\tau}^{-1}\ell + \sqrt{-1}({\rm Im}(\tau) \xi_1-{\rm Re}(\tau)\xi_2).
\]
Note that under the lattice $\mathbb{Z}+\tau\mathbb{Z}$ we have
 \[
 y_{2} \sim y_{2} +{\rm Im}(\tau)\cdot \mathbb{Z}.
 \]
It follows that the fibration $\pi$ is holomorphic with respect to $J_{\tau}$ with the fibers given by $\{y={\rm const}\}$. We next need to construct a $J_{\tau}$ holomorphic coordinate on the (universal cover of the) total space $\mathcal{C}$ which restricts to a coordinate along the fibers of $\pi$.  We will call this coordinate $x= x_1+\sqrt{-1}x_2$.  Define 
\[
x_2= c_{\tau}\ell \xi_2
\]
 which is well defined on the universal cover. Then 
\begin{equation}\label{eq: Jdx2}
J_{\tau} dx_2 = a_{\tau}\theta - c_{\tau}^2(a_{\tau}\xi_2d\xi_1+b_{\tau}\xi_2 d\xi_2) +b_{\tau}\ell d\ell
\end{equation}
which is a well-defined, closed $1$-form, since
\[
d\theta= \frac{1}{2}dJd\phi = -\ddb \phi = -c_{\tau}^2d\xi_1\wedge d\xi_2.
\]
In order to construct the holomorphic coordinate $x$, it suffices to find a $J_{\tau}$ holomorphic section $\sigma$ of the fibration such that $Jdx_2\big|_{\sigma} =0$.  If such a section can be found then we can integrate the closed $1$-form $J_{\tau}dx_2$ to find $x_1$. Consider a general map
\[
\sigma(y) \longmapsto (\ell(y),  \psi(y), \xi_1(y), \xi_2(y)) = \left(\frac{c_\tau}{|\tau|} y_1, \psi(y_1,y_2), \frac{y_2}{{\rm Im}(\tau)},0\right).
\]
The easiest way to find a holomorphic section of the fibration is to find a section of the fibration which is special Lagrangian for the data $(\omega_J, \Omega_J)$.  We can appeal to equations~\eqref{eq: ansatzHKTrip} to see that 
\[
\begin{aligned}
\omega_{J}\big|_{\sigma} &=\theta \wedge d\ell = \theta \wedge \frac{c_{\tau}}{|\tau|}dy_1\\
{\rm Im}(\Omega_{J})\big|_{\sigma} &=\sqrt{-1}\frac{\bar{\tau}}{|\tau|} c_{\tau}(\frac{1}{2}d\phi -\ell d\ell +\sqrt{-1}d\psi)\wedge d\xi_1\\
&=\sqrt{-1}\frac{\bar{\tau}}{|\tau|} c_{\tau}(-\ell d\ell +\sqrt{-1}d\psi)\wedge d\xi_1\\
&= \sqrt{-1}\frac{\bar{\tau}}{|\tau|} c_{\tau}(-\frac{c_{\tau}^2}{|\tau|^2}y_1 dy_1 +\sqrt{-1}d\psi)\wedge \frac{dy_2}{{\rm Im}(\tau)}.
\end{aligned}
\]
From the equation $\omega_{J}|_{\sigma}$ we see that $\theta|_{\sigma}=0$, or in other words, we should look for a parallel section $s : \{\xi_2=0\} \rightarrow L\big|_{\{\xi_2=0\}}$.  Let us show that such a section exists.  Let $s_0 \in L_{(0,0)}$ be a unit length section. From the formula for $\phi$ we have 
\[
\frac{1}{2}Jd\phi = \frac{k\pi}{{\rm Im}(\tau)}  \xi_2d\xi_1-\xi_1d\xi_2,
\]
and so restricting to $\{\xi_2=0\}$ gives $\theta\big|_{\{\xi_2=0\}} = d\psi$.  Thus, under parallel transport we have $\psi(\xi_1) = \psi(0)$.  Recall that, from our choice of base point for the Abel-Jacobi map, the identification of the fibers $L_{(0,0)}$ and $L_{(1,0)}$ is given by the transition function $e_1(0)= e^{\frac{k\pi}{2{\rm Im}(\tau)}}$, which is real.  Since the connection is unitary, parallel transport along $x_2=0$ has trivial monodromy and so any parallel section is well-defined.  Let $s_0(\xi_1)$ be the parallel transport of $s_0$ and consider the map
\[
(\ell, \xi_1)\mapsto\sigma(\ell, x_1):= e^{-\frac{\ell^2}{2}-\sqrt{-1}\frac{b_\tau \ell^2}{2a_{\tau}}}s_0\left(\xi_1\right)  \in L_{\xi_1,0},
\]
or, written in terms of the coordinates on the base of the fibration,
\[
(y_1,y_2) \mapsto\sigma(y_1,y_2):=  e^{-\frac{c_{\tau}^2y_1^2}{2}-\sqrt{-1}\frac{b_\tau c_{\tau}^2 y_1^2}{2a_{\tau}}}s_0\left(\frac{y_2}{{\rm Im}(\tau)}\right) \in \pi^{-1}(y_1,y_2).
\]
It remains to show that this section is holomorphic and that $J_{\tau}dx_2\big|_{\sigma}=0$.  We will work in coordinates $\xi_1, \ell$ to avoid unnecessary factors of $c_{\tau}$.  Since $s_0(\xi_1)$ is parallel, we have
\[
\theta|_{\sigma} = -\frac{b_\tau \ell}{a_{\tau}}d\ell
\]
and so $\omega_{J}\big|_{\sigma}=0$.  Similarly, we have
\[
d\psi = \frac{\del \psi}{\del \xi_1} d\xi_1 + \frac{\del \psi}{\del \ell} d\ell = -\frac{1}{2}\frac{\del \phi}{\del \xi_2} d\xi_1 -\frac{b_\tau \ell}{a_{\tau}}d\ell
\]
where we used again that $s_0(\xi_1)$ is parallel.  Therefore
\[
\begin{aligned}
\Omega_{J}\big|_{\sigma} &= \sqrt{-1}\frac{\bar{\tau}}{|\tau|} c_{\tau}(-\ell d\ell -\sqrt{-1}\frac{b_\tau \ell}{a_{\tau}}d\ell)\wedge d\xi_1\\
&= \sqrt{-1}\frac{\bar{\tau}}{|\tau|} \frac{c_{\tau}}{|\tau|a_{\tau}}(-{\rm Im}(\tau) +\sqrt{-1}{\rm Re}(\tau))\ell d\ell\wedge d\xi_1\\
&=-\frac{c_{\tau}}{a_{\tau}}\ell d\ell\wedge d\xi_1.
\end{aligned}
\]
So indeed we have ${\rm Im}(\Omega_{J})\big|_{\sigma} =0$ and hence $\sigma$ is holomorphic with respect to $J_{\tau}$. Furthermore, using again the parallel transport equation we have
\[
J_{\tau}dx_2\big|_{\sigma} = a_{\tau}( -\frac{b_\tau \ell}{a_{\tau}}d\ell)+ b_{\tau}\ell d\ell =0.
\]
This completes the construction of the holomorphic coordinate $x$. 

Let us now identify the lattice of the torus fibration.  Without loss of generality, we can assume that $\psi(s_0(0,0)) =0$. First, it is clear that $x \sim x+a_{\tau}2\pi$, which follows from computing the integral
\[
\int_{(\psi, \xi_2)=(-\frac{b_{\tau}}{2a_{\tau}}\ell^2, 0)}^{(2\pi-\frac{b_{\tau}}{2a_{\tau}}\ell^2, 0)} dx_1. 
\]
Next we claim that
\begin{equation}\label{eq: appxsim}
x\sim x+a_{\tau}2\pi, \qquad x\sim x+ \sqrt{-1}c_{\tau}^2y_1\frac{{\rm Im}(\tau)}{|\tau|} -\frac{2\pi k}{|\tau|}y_2.
\end{equation}
To see this, recall that the fiber of $\pi$ over $(\ell, y_2)=(K, c)$ is given by~\eqref{eq: McK}.

Under the action of $\mathbb{Z}+\tau \mathbb{Z}$ we have
\begin{equation}\label{eq: startEndPts}
\left(\frac{c}{{\rm Im}(\tau)}, 0s_0(\xi_1)\right) \sim \left(\frac{c}{{\rm Im}(\tau)}+ {\rm Re}(\tau), {\rm Im}(\tau), a(\tau)e^{\frac{k\pi}{{\rm Im}(\tau)}(\bar{\tau}\xi_1)+\frac{|\tau|^2}{2}}s_0(\xi_1)\right),
\end{equation}
and from our choice of origin we have $a(\tau)=1$.  Thus, it suffices to compute the integral of $dx_1= Jdx_2$ over any curve in $M_{c,K}$ connecting the points on the left and right hand sides of~\eqref{eq: startEndPts}. Along the fiber we can write $dx_1=a_{\tau}d\psi - \frac{k\pi}{{\rm Im}(\tau)|\tau|}y_2 d\xi_2$, and so
\[
\begin{aligned}
\int_{(\psi, \xi_2)=(-\frac{b_{\tau}}{2a_{\tau}}\ell^2, 0)}^{(-k\pi\frac{y_2}{{\rm Im}(\tau)}-\frac{b_{\tau}}{2a_{\tau}}\ell^2, {\rm Im}(\tau))} dx_1=&\int_{(\psi, \xi_2)=(-\frac{b_{\tau}}{2a_{\tau}}\ell^2, 0)}^{(-k\pi\frac{y_2}{{\rm Im}(\tau)}-\frac{b_{\tau}}{2a_{\tau}}\ell^2, {\rm Im}(\tau))} a_{\tau}d\psi - \frac{k\pi}{{\rm Im}(\tau)|\tau|}y_2 d\xi_2\\
&=-a_{\tau} k\pi\frac{y_2}{{\rm Im}(\tau)} - \frac{k\pi}{|\tau|}y_2 = -\frac{2k\pi}{|\tau|}y_2.
\end{aligned}
\]
Writing these relations in terms of $y_1, y_2$ yields~\eqref{eq: appxsim}.  Define 
\[
\tilde{x}= \frac{x}{a_{\tau}2\pi}.
\]  
Then using the definition of $a_{\tau}, c_{\tau}$ we arrive at
\[
\tilde{x}\sim \tilde{x}+1, \qquad \tilde{x} \sim \tilde{x} + \sqrt{-1}\frac{ky_1}{{\rm Im}(\tau)}- \frac{k}{{\rm Im}(\tau)}y_2.
\]
Set $ z= e^{-\frac{2\pi}{{\rm Im}(\tau)}(y_1+\sqrt{-1}y_2)}$ so that $z$ is a well defined coordinate on $\Delta = \{|z|<1\} \subset \mathbb{C}$.  Then the holomorphic fibration $\pi: (\mathcal{C}, J_{\tau})\rightarrow \Delta^*$ is determined by the lattice
\[
\mathbb{Z} \oplus \frac{k}{2\pi \sqrt{-1}}\log(z) \cdot \mathbb{Z}.
\]
Let us write the symplectic form $\omega_{\tau}$, given in equation~\eqref{eq: omegaTau}, in the $(\tilde{x}, z)$ coordinates.  As a first step, we will write the symplectic form in the coordinates $(x,y)$.  The most laborious term to rewrite is the term in~\eqref{eq: omegaTau} involving $\theta$, and so we will take this on first.  Recall that these coordinates are given by
\[
\begin{aligned}
dx_1&= a_{\tau}\theta - c_{\tau}^2(a_{\tau}\xi_2d\xi_1+b_{\tau}\xi_2 d\xi_2) +b_{\tau}\ell d\ell,& \qquad x_2&= c_{\tau}\ell\xi_2,\\
y_1&= |\tau| c_{\tau}^{-1}\ell,& \qquad y_2 &=({\rm Im}(\tau) \xi_1-{\rm Re}(\tau)\xi_2).
\end{aligned}
\]
As a first step, we calculate $d\xi_1, d\xi_2$ in terms of $(x,y)$.  Note that we can write $x_2= c_{\tau}^2|\tau|^{-1}y_1\xi_2$, and so
\[
d\xi_2 = \frac{|\tau|}{c_{\tau}^2y_1}\left(dx_2-\frac{x_2}{y_1}dy_1\right).
\]
From the formula for $y_2$ we get
\[
 d\xi_1 = \frac{dy_2}{a_{\tau}|\tau|}-\frac{|\tau|b_{\tau}}{a_{\tau}c_{\tau}^2y_1}\left(dx_2-\frac{x_2}{y_2}dy_1\right).
\]
Combining these formulas gives
\[
\begin{aligned}
a_{\tau}d\xi_2-b_{\tau}d\xi_1 &= \left(\frac{a_{\tau}|\tau|}{c_{\tau}^2}+\frac{|\tau|b_{\tau}^2}{a_{\tau}c_{\tau}^2}\right)\frac{1}{y_1}\left(dx_2-\frac{x_2}{y_2}dy_1\right)- \frac{b_{\tau}}{a_{\tau}|\tau|}dy_2\\
&= \frac{|\tau|^2}{2\pi ky_1}\left(dx_2-\frac{x_2}{y_1}dy_1\right)- \frac{b_{\tau}}{a_{\tau}|\tau|}dy_2,
\end{aligned}
\]
where in the last line we used the definition of $a_{\tau}, b_{\tau}, c_{\tau}$ (see ~\eqref{eq: atauDef}).  Let us now rewrite $\theta$. From~\eqref{eq: Jdx2} we have
\[
\begin{aligned}
a_{\tau}\theta &= dx_1+c_{\tau}^2(a_{\tau}\xi_2d\xi_1+b_{\tau}\xi_2 d\xi_2) -b_{\tau}\ell d\ell\\
&= dx_1 + \frac{x_2}{y_1}dy_2 - \frac{b_{\tau} c_{\tau}^2}{|\tau|^2} y_1dy_1.
\end{aligned}
\]
Define a $(1,0)$-form by
\[
\begin{aligned}
E &:= \frac{|\tau|^2}{2\pi k} \left(dx+\frac{x_2}{\sqrt{-1}y_1}dy\right) - \frac{b_{\tau}}{|\tau|a_{\tau}}y_1dy\\
&= \frac{|\tau|^2}{2\pi k} \left(\left(dx+\frac{x_2}{\sqrt{-1}y_1}dy\right) - \frac{b_{\tau}c_{\tau}^2}{|\tau|^2}y_1dy\right),
\end{aligned}
\]
where we used again the definitions of $a_{\tau}, c_{\tau}$.  Then we can write
\[
\theta=\frac{2\pi k}{a_{\tau}|\tau|^2}{\rm Re}(E), \quad a_{\tau}d\xi_2-b_{\tau}d\xi_1 = \frac{1}{y_1} {\rm Im}(E)
\] 
and so
\[
\omega_{\tau} = \frac{\sqrt{-1}}{2} c_{\tau}^3 \left( \frac{y_1}{|\tau|^3}dy\wedge d\bar{y} + \frac{1}{y_1|\tau|} E\wedge \overline{E}\right).
\]
We now write this in terms of $(\tilde{x}, z)$, where we recall that
\[
\tilde{x} = \frac{x}{2\pi a_{\tau}}, \qquad \log z = -\frac{2\pi}{{\rm Im}(\tau)} y.
\]
We have
\[
y_1 dy\wedge d\bar{y} = \left(\frac{{\rm Im}(\tau)}{2\pi}\right)^3 (-\log|z|)\frac{dz\wedge d\bar{z}}{|z|^2}.
\]
We also compute
\[
E= \frac{|\tau|{\rm Im}(\tau)}{k}\left(\left(d\tilde{x} - \frac{\tilde{x}_2}{\sqrt{-1}z(-\log|z|)} dz\right) + \frac{b_{\tau}k}{4\pi^2|\tau|} (-\log|z|) \frac{dz}{z}\right).
\]
After some routine manipulations we arrive at
\[
\omega_{\tau} = \alpha \left(\frac{W^{-1}\alpha}{\epsilon} \sqrt{-1}\frac{dz \wedge d\bar{z}}{|z|^2} + \frac{W\epsilon}{\alpha}\frac{\sqrt{-1}}{{2}} \tilde{E}\wedge \overline{\tilde{E}}\right)
\]
where
\[
W= \frac{2\pi}{k |\log|z||}, \qquad \epsilon = 2\pi|\tau| \sqrt{\frac{2\pi k}{{\rm Im}(\tau)}},\qquad  \alpha = \frac{\sqrt{k\pi {\rm Im}(\tau)}}{|\tau|}
\]
and we have denoted
\[
 \tilde{E} = \left(d\tilde{x} - \frac{\tilde{x}_2}{\sqrt{-1}z(-\log|z|)} dz\right) + \frac{b_{\tau}k}{4\pi^2|\tau|} (-\log|z|) \frac{dz}{z}.
\]
If ${\rm Re}(\tau)=0$ then this is clearly a rescaling of a standard semi-flat metric in the sense of Definition~\ref{defn: sf metric}. On the other hand  if ${\rm Re}(\tau) \ne 0$ then
\[
\frac{2}{k}\frac{1}{2|\tau|}b_{\tau}k= \frac{b_{\tau}}{|\tau|} = \frac{-{\rm Re}(\tau)}{|\tau|^2} \in(-1,0)\cup (0,1)
\]
since ${\rm Re}(\tau)\in [-\frac{1}{2}, \frac{1}{2})$ and $|\tau| \geq 1$.  Thus, $\frac{1}{2|\tau|}b_{\tau}k \notin \frac{k}{2}\mathbb{Z}$ and so by Lemma~\ref{lem: cohomProps} $\omega_{\tau}$ is not a rescaling of a standard  semi-flat metric.  Thus, if ${\rm Re}(\tau)\ne 0$, the hyperK\"ahler rotation of the Tian-Yau metric is always a non-standard semi-flat metric in the sense of Definition~\ref{defn: non-stand sf metric}.

\section{Hein's construction}\label{app: HeinApp}

In this appendix we provide the necessary details needed for Theorem \ref{thm: nonStHein}, adapted from \cite{Hein}.  In particular, we show that with our modified assumptions one can still construct a background K\"ahler metric with the properties needed to deduce the existence of the desired Calabi-Yau metric from the results of \cite{Hein}.  We emphasize that the discussion here can be extracted directly from Hein's work.

Recall the setting of Theorem \ref{thm: nonStHein}.  We let $\omega_0$ be a smooth K\"ahler metric on $X$ satisfying
\[
[\omega_{0}]_{dR}.[F] =\epsilon,\qquad[\omega_{0}]_{dR}.[C] = \frac{2b_0}k\epsilon \notin \mathbb{Z} \epsilon.
\]
In order construct the background K\"ahler metric, we need to glue $\omega_0$ to a semi-flat metric in the neighborhood of the singular $I_k$ fiber.  By Corollary \ref{cor: local deRham} and Corollary~\ref{cor: locBCKahler} (cf. \cite[Claim 1]{Hein}), there exists a (possibly nonstandard) semi flat metric $\omega_{sf,b_0,\sigma,\epsilon}$ such that
\[
[\omega_{0}]_{BC} = [\omega_{sf, \sigma,b_0,\epsilon}]_{BC} \in H^{1,1}_{BC}(X_{\Delta^*}, \mathbb{R}).
\]
We can then write
\[
\omega_{sf,\sigma, b_0,\epsilon} =\omega_{0}+ \ddb u_1
\]
for some (non-unique) function $u_1:X_{\Delta^*}\rightarrow\mathbb R$. To ease notation, let us suppress the dependence of the metric on the section $\sigma$, writing $\omega_{sf, b_0, \epsilon}$ in place of $\omega_{sf,\sigma, b_0,\epsilon}$. We work on $X_{mod}$ with model coordinates. First, define
\[
\omega_{sf,b_0,\epsilon}(\alpha) =\sqrt{ \alpha }\omega_{sf, b_0,\frac{\epsilon}{\sqrt{\alpha}}}
\]
for $\alpha>0$.  Note that this notation differs slightly from the definition in the main body of the paper where we have used the convention $\alpha \mapsto \alpha^2$.  We have chosen this convention for the appendix as it is consistent with Hein's notation \cite{Hein}.   We have
\[
\omega_{sf,b_0,\epsilon}(\alpha) - \omega_{sf,b_0,\epsilon} = (\alpha-1)\sqrt{-1}|\kappa(z)|^2 \frac{k|\log|z||}{2\pi \epsilon}\frac{dz\wedge d\bar{z}}{|z|^2} =(\alpha-1)\ddb u
\]
for some smooth potential $u$, which exists since $\Delta^*$ satisfies the $\bar\partial$-Poincar\'e lemma. Note that $u$ is the same function  for both standard and non-standard semi-flat metrics and hence has all the same properties as the similarly defined function $u$ in \cite{Hein}.  We now write
\[
\omega_{sf, b_0,\epsilon}(\alpha) = \omega_{0} + \ddb u_{\alpha}
\]
where $u_{\alpha} = u_1 + (\alpha-1)u$.  Since these functions are not unique we fix choices for them once and for all. 

Introduce the notation $\Delta(r) = \{ |z|<r\} \subset \Delta$.    Let $\psi= \psi_{r,s}$ be a radial cut-off function with $\psi \equiv 1$ on $\Delta(r+s)$ and ${\rm supp}(\psi) \subset \Delta(r+2s)$, satisfying
\[
s|\psi_{z}| + s^2 |\psi_{z\bar{z}}| < C_0.
\]
As in \cite{Hein}, we let $C_0$ denote a constant which depends only on fixed data and $C_0(r,s)$ a constant depending on $r,s$.  Let $\beta$ be a  $(1,1)$ form on $\mathbb{P}^1$ such that ${\rm supp}(\beta) \subset \Delta(r+3s)\setminus \Delta(r)$, $0\leq \beta \leq |dz|^2$, and $\beta = |dz|^2$ on $\Delta(r+2s)\setminus \Delta(r+s)$.  We identify $\psi, \beta$ with forms on $X_{\Delta^*}$.

First, since the function $\ddb u$ does not depends on whether the semi-flat metric is standard or non-standard, we have

\begin{lem}[Claim 2, \cite{Hein}]\label{lem: Claim2}
If $C_0r<1, C_0s<r$, and $v$ is harmonic with the same boundary values as $u$ on $\Delta(r+3s)\setminus \Delta(r)$, then
\[
\sup_{\Delta(r+2s)\setminus\Delta(r+s) }(s^{-2}|u-v| +s^{-1}|(u-v)_{z}|) \leq C_0 \sup_{\Delta(r+3s)\backslash \Delta(r)} u_{z\bar{z}}
\]
\end{lem}
Define the $(1,1)$ form
\[
\omega_\alpha( t) := \begin{cases} \omega_0+ t\beta + \ddb (\psi \widetilde{u}_{\alpha}) & \text{ outside} \,\,X_{\Delta(r)}\\
							\omega_0 + t\beta + \ddb u_{\alpha} & \text{ over } X_{\Delta(r+s)^{*}}
				\end{cases}
\]
where 
\[
\widetilde{u}_{\alpha} = u_1+ (\alpha-1) (u-v) \qquad \text{ on } X_{\Delta(r+3s)\backslash \Delta(r)}.
\]
A few remarks are in order.  First, $\beta$ is  not exact on $\mathbb{P}^1$, but it is exact on $\mathbb{P}^1\setminus \{\infty\}$, thus $\omega_\alpha(t)$ is in the same Bott-Chern class as $\omega_0$.  Secondly, note that
\begin{equation}\label{eq: gluedMetric}
\omega_\alpha(t) = \begin{cases} \omega_0 &\text{outside}\,\, X_{\Delta(r+3s)}\\
\omega_{sf,\sigma, b_0,\epsilon}(\alpha) &\text{ over } X_{\Delta(r)^*}
\end{cases}
\end{equation}
We now check that this form is positive.  

\begin{lem}[\cite{Hein}, Claim 3]\label{lem: Claim 3}
There exists a constant $C_0(r,s) > 0$, depending only on $\psi, u_1$, such that if $t>C_0(r,s) + C_0|\alpha-1|\sup_{\Delta(r+3s)\backslash \Delta(r)} u_{z\bar{z}}$, then, on all of $X$ we have
\[
\omega_\alpha(t) \geq \frac{1}{2} (\omega_0 + \psi \ddb u_{\alpha})>0.
\]
\end{lem}
\begin{proof}
Cleary the required estimate holds outside $X_{\Delta(r+2s)}$, since $\psi$ is zero there and $\beta \geq 0$.  Also, on $\Delta(r+s)$, $\psi\equiv 1$ and the estimate is again obvious from the definition.  So we only need to check on $\Delta(r+2s)\setminus \Delta(r+s)$.  On this region we have
\[
\begin{aligned}
\omega_{\alpha}(t) &= \omega_0 + t\beta +\widetilde{u}_{\alpha} \psi_{z\bar{z}} \sqrt{-1}dz \wedge d\bar{z} + \psi \ddb u_{\alpha} + 2{\rm Re}\left(\sqrt{-1}\del \widetilde{u}_{\alpha} \wedge \dbar \psi \right) \\ 
&= (\omega_0 + \psi \ddb u_{\alpha}) + t\beta + \widetilde{u}_{\alpha} \psi_{z\bar{z}} \sqrt{-1}dz \wedge d\bar{z} +  2{\rm Re}\left(\sqrt{-1}\del \widetilde{u}_{\alpha} \wedge \dbar \psi\right)
\end{aligned}
\]
where we used that $\ddb \widetilde{u}_{\alpha} = \ddb u_{\alpha}$ since $v$ is pluriharmonic.  Note that
\[
(\omega_0 + \psi \ddb u_{\alpha}) = (1-\psi)\omega_0 + \psi \omega_{sf,\sigma, b_0,\epsilon}(\alpha)  > 0.
\]
Because $u, v$ are independent of the fiber coordinate $x$, we have $\del_x \widetilde{u}_{\alpha} = \del_x u_1$,
and so we can rewrite the expression for $\omega_{\alpha}( t)$ as
\[
\begin{aligned}
\omega_{\alpha}( t) &=  (\omega_0 + \psi \ddb u_{\alpha}) + t\beta + 2{\rm Re}\left((u_1)_{x}\psi_{\bar{z}} \sqrt{-1}dx\wedge d\bar{z}\right)\\
&+  \left(\psi_{z\bar{z}}\widetilde{u}_{\alpha} + \psi_{z}(\widetilde{u}_{\alpha})_{\bar{z}} + \psi_{\bar{z}}(\widetilde{u}_{\alpha})_{z}\right)  \sqrt{-1}dz\wedge d\bar{z}
\end{aligned}
\]
To deal with the $dx\wedge d\bar{z}$ terms we evaluate on $c_{x} \frac{\del}{\del x}, c_{z} \frac{\del}{\del z}$ and use the fact that
\[
2{\rm Re}(ac_{x}\bar{c_z}) \geq -\frac{|a|}{\delta} |c_z|^2 - \delta |a||c_x|^2,
\]
and so if we choose $\delta$ such that
\[
2\delta \sup_{\Delta(r+2s)\setminus \Delta(r+s)} |\psi_{\bar{z}}(u_1)_{x}|= \inf_{\Delta(r+2s)\setminus \Delta(r+s)} (\omega_0+\psi\ddb u_1)_{x\bar{x}},
\]
 we then can absorb the $|c_x|^2$ term into the corresponding term from $(\omega_0+\psi\ddb u_1)$.  Taking $C_0(r,s) \gg 1$ depending on $u_1, \psi$ we can absorb the $|c_z|^2$ term into $t\beta$.  We now need to deal with the $\sqrt{-1}dz\wedge d\bar{z}$ terms.  This can be done in the following way.   We write
\[
\begin{aligned}
\left(\psi_{z\bar{z}}\widetilde{u}_{\alpha} + \psi_{z}(\widetilde{u}_{\alpha})_{\bar{z}} + \psi_{\bar{z}}(\widetilde{u}_{\alpha})_{z}\right) &=\left(\psi_{z\bar{z}}u_1 + \psi_{z}(u_1)_{\bar{z}} + \psi_{\bar{z}}(u_1)_{z}\right) \\
&+ (\alpha -1)\left(\psi_{z\bar{z}}(u-v) +2{\rm Re}\left( \psi_{z}(u-v)_{\bar{z}}\right)\right)
\end{aligned}
\]
The $u_1$ contribution can be controlled by increasing $C_0(r,s)$ (and hence $t$). The $(\alpha-1)(u-v)$ contribution is controlled uniformly by Lemma~\ref{lem: Claim2}
\[
|(\alpha -1)\left(\psi_{z\bar{z}}(u-v) + 2{\rm Re}(\psi_{z}(u-v)_{\bar{z}})\right) | \leq C_0 |(\alpha-1)| \sup_{\Delta(r+3s)\setminus \Delta(r)} u_{z\bar{z}}.
\]
This completes the proof of the lemma.
\end{proof}


We have shown that, for $\alpha>0$, if $t$ is sufficiently positive then there is a complete K\"ahler metric $\omega_\alpha( t)$ agreeing with the semi-flat metric on $X_{\Delta(r)}$ and agreeing with $\omega_0$ outside of  $X_{\Delta(r+3s)}$.   This proves points $(i)$ and $(ii)$ from the statement of Theorem~\ref{thm: nonStHein}.

It remains to show that, up to adjusting the metric outside a a neighborhood of the $I_k$ singular fiber, we can satisfy the integrability condition
\[
\int_{X}\omega_\alpha(t)^2- \alpha \Omega \wedge \overline{\Omega} =0.
\]
From here the conclusion of Theorem~\ref{thm: nonStHein} follows from Hein's argument \cite{Hein}.  To do this we exploit the fibration structure following the idea of Chen-Chen \cite{CC}. 

Fixing $r,s$ as above, we have a K\"ahler form $\omega_{\alpha}(t)$ on $X$ which coincide with a (possibly non-standard) semi-flat metric $\omega_{sf,\sigma,b_0,\epsilon}(\alpha)$ on $X_{\Delta(r)^*}$.  We now essentially repeat the construction immediately preceding Lemma~\ref{lem: Claim2}.  Fix $r_1,s_1>0$ such that $0<r_1+3s_1<r$. 	Let $\beta_1$ be a $(1,1)$-form on $\mathbb{P}^1$ such that $\mbox{supp}(\beta_1)\subseteq \Delta(r_1+3s_1)\setminus \Delta(r_1)$, $0\leq \beta_1\leq |dz|^2$, and $\beta=|dz|^2$ on $\Delta(r_1+2s_1)\setminus \Delta(r_1+s_1)$. We will again view $\beta_1$ as a form on $X$ via pull-back. Then we take 
	\begin{align*}
			\omega_{\alpha}(a_1,t)=\omega_{\alpha}(t)+a_1\beta_1.
	\end{align*}  
	
	Thanks to the fibration structure and semi-positivity of $\beta_1$ we have
 \[
 \int_{X}\omega_{\alpha}(t)\wedge \beta_1>0
 \]
 and thus, if $a_1 \gg 0$ then
 \[
 \int_X \omega_{\alpha}(a_1,t)^2-\alpha\Omega\wedge \bar{\Omega}>0.
 \]

Now let $r_2>0$ to be determined later and take $\beta_2=\frac{i\delta}{|z|^2}dz\wedge \bar{dz}$, where $\delta$ is a radial cut-off function with $\delta\equiv \delta_0$ (for $\delta_0>0$ to be determined) on $\Delta(r+s)\setminus \Delta(r)$ and $\mbox{supp}(\delta)\subseteq \Delta(r+2s)\setminus \Delta(\frac{1}{2}r_2)$, satisfying 
 \[
 (r+2s-\frac{1}{2}r_2)|\delta_z|+(r+2s-\frac{1}{2}r_2)^2|\delta_{z\bar{z}}|<C_0.
 \]
 Using that $\omega_{\alpha}(a_1,t)\geq \omega_{sf,\sigma,b_0,\epsilon}(\alpha)$ on $\Delta(r)$ and the explicit form of the (non-standard) semi-flat metric, we may choose $\delta_0>0$, depending only on $r,\epsilon$ such that $\omega_{\alpha}(a_1,r_2,t)=\omega_{\alpha}(a_1,t)-\beta_2 >0$  is still a K\"ahler form. Since $\beta_2$ is pulled-back from $\mathbb{P}^1$ and semi-positive, we have 
     \[
    \lim_{r_2\rightarrow 0} \int_X \omega_{\alpha}(a_1,t)\wedge \beta_2= \infty
     \]
     as $r_2\rightarrow 0$ since $\int_{\Delta^*}\beta_2\rightarrow +\infty$ as $r_2\rightarrow 0$.  Therefore,
    \[
    \int_X \omega_{\alpha}(a_1,r_2,t)^2-\alpha \Omega\wedge \bar{\Omega}< 0
    \]
    for $a_1$ fixed and $r_2\rightarrow 0$. By the intermediate value theorem, we can find $r_2,a_1>0$ such that 
    \[
    \int_X \omega_{\alpha}(a_1,r_2,t)^2-\alpha \Omega\wedge \bar{\Omega}=0. 
\]

\end{document}